\documentclass[11pt]{amsart}
\usepackage{geometry}                
\usepackage{graphicx}
\usepackage{subfig}
\usepackage{psfrag}
\usepackage{labelfig}
\usepackage{amssymb}
\usepackage{epstopdf}
\usepackage{color}
\newtheorem{theorem}{Theorem}[section]    
\newtheorem{lemma}[theorem]{Lemma}          
\newtheorem{proposition}[theorem]{Proposition}  
\newtheorem{claim}[theorem]{Claim}  
\newtheorem{corollary}[theorem]{Corollary} 
\newtheorem{definition-prop}[theorem]{Proposition-Definition}

\theoremstyle{definition}
\newtheorem{definition}[theorem]{Definition}

\newtheorem{convention}[theorem]{Convention}
\newtheorem{observation}[theorem]{Observation}
 
\newtheorem{remark}[theorem]{Remark}
\newtheorem{example}[theorem]{Example}   
\newtheorem*{remark-unnumbered}{Remark}
\numberwithin{equation}{section}

\newcommand{\e}{\varepsilon}

\newcommand{\F}{\mathcal F_{ob} }
\newcommand{\cF}{\mathcal F_\xi }
\newcommand{\Int}{\textrm{Int}}
\newcommand{\id}{\textrm{id}}
\newcommand{\R}{\mathbb{R}}
\newcommand{\Q}{\mathbb{Q}}
\newcommand{\Z}{\mathbb{Z}}
\newcommand{\MCG}{\mathcal{MCG}}
\newcommand{\Aut}{{\rm Aut}}
\newcommand{\Val}{{\rm{Val}}}
\newcommand{\Hyp}{{\rm{Hyp}}}
\newcommand{\mF}{\mathcal{F} }

\newcommand{\sgn}{{\tt sgn}}

\title{Essential open book foliations and fractional Dehn twist coefficient}
\author{Tetsuya Ito}
\address{Research Institute for Mathematical Sciences, Kyoto university
Kyoto, 606-8502, Japan}
\email{tetitoh@kurims.kyoto-u.ac.jp}
\urladdr{http://www.kurims.kyoto-u.ac.jp/~tetitoh/}

\author{Keiko Kawamuro}
\address{Department of Mathematics \\ 
The University of Iowa \\ Iowa City, IA 52242, USA}
\email{kawamuro@iowa.uiowa.edu}
\date{\today} 

\begin{document}

\begin{abstract}
We introduce essential open book foliations by  refining open book foliations, 
and develop technical estimates of the fractional Dehn twist coefficient (FDTC) of monodromies and the FDTC for {\em closed braids}, which we introduce as well.

As applications, we quantitatively study the `gap' between overtwisted contact structures and non-right-veering monodromies.  
We give sufficient conditions for a $3$-manifold to be irreducible and atoroidal. 
We also show that the geometries of a $3$-manifold and  the complement of a closed braid are determined by the Nielsen-Thurston types of the monodromies of their open book decompositions. 
\end{abstract}

\maketitle

\section{Introduction}

Let $S=S_{g, d}$ be a compact oriented genus $g$ surface with $d\neq 0$ boundary components. 
Let $\Aut(S, \partial S)$ be the group of isotopy classes of diffeomorphisms of $S$ fixing the boundary $\partial S$ pointwise. 
Abusing notation we will often regard $\phi \in \Aut(S, \partial S)$ as a diffeomorphism representing $\phi$. 
Suppose that $(S,\phi)$ is an open book decomposition of a closed oriented $3$-manifold $M$. 
Consider a compact oriented surface $F$ in $M$ possibly with  boundary. 
Generic intersection of $F$ and the pages of the open book yields a singular foliation on $F$. 
If it satisfies certain conditions we call it {\em open book foliation}  and denote it by $\F(F)$ \cite{ik1-1}. 
Open book foliations has their origin in Bennequin's work \cite{Ben} and Birman and Menasco's braid foliations \cite{BM4, BM2, BM5, BM1, BM6, BM3, BM7, bm1, bm2}, where the underlying manifold $M$ is $S^3$.

Let us recall Giroux's seminal result \cite{g}: 
For a closed oriented $3$-manifold $M$ there exists a one-to-one correspondence between open book decompositions of $M$ up to positive stabilization and  contact structures on $M$ up to contact isotopy. 
Due to the Giroux-correspondence open book foliations have extensive applications to contact geometry. 
For instance, in \cite{ik1-2} we give an alternative proof to  Honda, Kazez and Mati\'c's theorem: The contact structure $(M, \xi)$ is tight if and only if every open book supporting $(M, \xi)$ is right-veering \cite{hkm1}.

Now the notion ``{\em right-veering}'' is a key to tight contact structures. 
For every boundary component $C\subset \partial S$ and every essential arc $\gamma \subset S$ starting at $C$ if $\phi(\gamma)$ lies on the right of $\gamma$ near the starting point, we say that $\phi$ is right-veering. 
As a highly effective tool to detect right-veering-ness Honda, Kazez and Mati\'c in \cite{hkm1} introduce the {\em fractional Dehn twist coefficient}. 
It measures how much $\phi \in \Aut(S, \partial S)$ contributes twisting along a boundary component $C$ and is denoted by $c(\phi, C) \in \Q$.
They show that positivity of $c(\phi, C)$ is almost  equivalent to right-veering-ness of $\phi$.

In this paper we introduce ({\em strongly}) {\em essential} open book foliations and establish relationship between  fractional Dehn twist coefficients and  essential open book foliations. 
As applications we obtain results in contact $3$-manifolds and topology and geometry of $3$-manifolds.

In \cite{ik1-1,ik1-2} we start with an open book $(S,\phi)$ with certain properties (eg. non-right-veeringness) and construct surfaces $F$ in $M_{(S,\phi)}$ (eg. transverse overtwisted discs and Seifert surfaces of closed braids).
The combinatorial nature of the open book foliation $\F(F)$ enables us to extract various properties of the underlying contact manifold $(M_{(S,\phi)}, \xi_{(S, \phi)})$ or the closed braid $\partial F$, such as overtwistedness or the self-linking number.

In this paper we explore the converse: 
We start with a surface admitting an open book foliation $\F(F)$ and read properties of the monodromy $\phi$ like the fractional Dehn twist coefficient $c(\phi, C)$.
For this purpose, a general open book foliation often contains excessive information.  
We define a (strongly) essential open book foliation as an optimal  foliation for sharper estimates of $c(\phi, C)$.
It is similar to the convention that we use  incompressible surfaces rather than generic surfaces in order to study topology/geometry of the ambient $3$-manifold. 
Our procedure breaks into three steps:
\begin{enumerate}
\item 
Replace a given open book foliation $\F$ with another one $\F'$ that reflects more properties of the monodromy $\phi \in \Aut(S, \partial S)$. 
\item 
Read properties of $\phi$ from $\F'$. 
\item
Find applications to contact geometry and topology/geometry of $3$-manifolds.
\end{enumerate}
In this paper we mainly treat Steps (2) and (3).
We discuss Step (2) in Section~\ref{sec:estimateFDTC} and Step (3) in Sections~\ref{sec6}-\ref{sec:geometry}. 

As for Step (1), typically, $\F$ is a generic open book foliation and $\F'$ is an essential one. In fact, in some cases even an essential open book foliation is insufficient for our purposes. 
To deal with the insufficiency we further introduce notions of {\em strongly} essential b-arcs and elliptic points.
In \cite{ik3} we study various techniques (some foliation moves and braid moves) to convert a given $\F$ to a better $\F'$.

The paper is organized as follows: 

In Section~\ref{sec:review} we review basic notions and facts about open book foliations.

In Section \ref{sec: essential ob foliation}  we introduce essential open book foliations, and show that every incompressible surface, modulo de-summing essential spheres, can admits an essential open book foliation (Theorem \ref{theorem:weak}). In particular, if $M$ is irreducible, up to isotopy every incompressible surface can admit an essential open book foliation.
We also introduce the notion of a {\em strongly essential} b-arc which plays a crucial role in large part of the paper.

In Section \ref{sec:FDTC} we study the fractional Dehn twist coefficient $c(\phi, C)$. 
We give a practical and efficient method to compute $c(\phi, C)$ which does not require Nielsen-Thurston classification (Theorem~\ref{theorem:computation}), and interpret $c(\phi, C)$ as a translation number of certain natural dynamics of $\Aut(S,\partial S)$ (Theorem~\ref{theorem:translation}). We also introduce a new notion, the fractional Dehn twist coefficient $c(\phi, L,C)$ for a {\em closed braid} $L$ in $(S,\phi)$.

Section \ref{sec:estimateFDTC} is the core of the paper.
In Theorem~\ref{theorem:estimate} we obtain lower and upper bounds of $c(\phi, C)$ by counting the number of singularities of a strongly essential open book foliation. 
We observe similar estimates for $c(\phi, L, C)$.

In Section 6 we characterize non-right-veering monodromies:  
We show that $\phi$ is non-right-veering if and only if there exists a special ``simplest'' transverse overtwisted disc (Theorem \ref{theorem:nrv}).
This highlights the difference between non-right-veering monodromy and overtwistedness: 

\vspace{2mm}

\noindent
\textbf{Corollary~\ref{cor:n(S,phi)}.}
{\em
There exists an invariant $n(S, \phi) \in \Z_{\geq 0}$ such that 
\begin{enumerate}
\item
$n(S, \phi) = 0$ if and only if $\xi_{(S, \phi)}$ is tight (and hence $\phi$ is right veering). 
\item 
$n(S, \phi) = 1$ if and only if $\xi_{(S, \phi)}$ is overtwisted and $\phi$ is not right veering. 
\item 
$n(S, \phi) \geq 2$ if and only if $\xi_{(S, \phi)}$ is overtwisted and $\phi$ is right veering. 
\end{enumerate}
}

\vspace{2mm}

In Section \ref{sec:topol} we study topology of $3$-manifolds. 
We obtain bounds of $|c(\phi, C)|$ from data on a closed incompressible surface embedded in $M_{(S,\phi)}$ (Theorems~\ref{theorem:surface} and \ref{theorem:surface_connected}). 
As a consequence, we find sufficient conditions on the monodromy $\phi$ that $M_{(S,\phi)}$ is irreducible and/or atoroidal (Corollaries~\ref{irreducible-theorem} and \ref{atoroidal-theorem}). 
Our atoroidality criterion can be refined if we assume $\xi_{(S, \phi)}$ is tight (Theorem~\ref{thm:tight-atoroidal}).

We also apply our technique to study the topology of null-homologous braids in open books. 
We relate the genus of a braid $L$ and $c(\phi, L, C)$ (Corollary~\ref{cor:lower bound of g}), which plays an important role in proving our main result Theorem~\ref{theorem:geometry}.

In Section \ref{sec:geometry} 
we prove the following result parallel to Thurston's classification of geometry of mapping tori \cite{T}. 
The same result holds for closed braid complements (Theorem~\ref{theorem:geometry-braid}).
\vspace{2mm}

\noindent
\textbf{Theorem~\ref{theorem:geometry}.}
{\em 
Let $(S,\phi)$ be an open book decomposition of 3-manifold $M$.
Assume that
\begin{itemize}
\item $\partial S$ is connected and $|c(\phi,\partial S )|>1$, or 
\item $|c(\phi, C)|>4$ for every boundary component $C$ of $S$.
\end{itemize}
Then,
\begin{enumerate}
\item $M$ is toroidal if and only if $\phi$ is reducible.
\item $M$ is hyperbolic if and only if $\phi$ is pseudo-Anosov.
\item $M$ is Seifert fibered if and only if $\phi$ is periodic. 
\end{enumerate}
}

\vspace{2mm}

Since every 3-manifold admits an open book decomposition with a pseudo-Anosov monodromy \cite{ga0} (in fact, more strongly, every contact structure is supported by an open book with pseudo-Anosov monodromy \cite{ch0}) our restriction on the FDTC is necessary.

The theory of essential lamination and taut foliation has been successful in topology/geometry of $3$-manifolds. 
Here, let us compare laminations/foliations and open book foliations. 
Interestingly, they tend to play complementary roles in the following sense: The existence of a taut foliation, an essential lamination or a genuine lamination in $M_{(S, \phi)}$ implies that the fundamental group of $M_{(S, \phi)}$ has various nice properties strongly reflecting the geometric structure (cf. \cite{gk}, \cite{go}). 
Constructing such `good' foliation/lamination in $M_{(S,\phi)}$ can often be done when $\phi$ is pseudo-Anosov. 
However, many toroidal manifolds are also equipped with `good' foliations/laminations.

On the other hand, open book foliations can be used to determine $M_{(S,\phi)}$ is toroidal or atoroidal as shown in Corollary~\ref{atoroidal-theorem}, Theorem~\ref{thm:tight-atoroidal} and Remark~\ref{3rd-atoridal-criterion}. 
While, at this writing, open book foliations say
 little about the fundamental group of $M_{(S,\phi)}$. We use Gabai and Oertel's work on essential laminations \cite{go} and open book foliations to prove Theorems~\ref{theorem:geometry} and \ref{theorem:geometry-braid}. 
One may consider that these theorems are fine examples of the above point of view. 

Thruston's hyperbolic Dehn surgery theorem states that, beside finitely many exceptions, called exceptional surgeries, a Dehn filling of a hyperbolic $3$-manifold yields a hyperbolic $3$-manifold. 
Our Theorem~\ref{theorem:geometry}-(2)
implies that Dehn fillings of the mapping torus of $\phi$ of slope $\frac{1}{k}$ are not exceptional if $|k|>1$ for connected boundary case and $|k|>4$ for general case. 
Note that the conditions of our Theorem~\ref{theorem:geometry} are given in terms of FDTCs, whereas usually exceptional surgery concerns the distance of slopes of two exceptional surgeries (see Gordon's lecture note \cite{gor} for exceptional surgeries).

Open book foliations give direct and elementary ways to analyze the contact structure $\xi_{(S, \phi)}$ and the manifold $M_{(S,\phi)}$ including reducible and toroidal ones, with or without essential laminations, and even without knowing the Nielsen-Thurston type of $\phi$.

\section{Quick review of open book foliation}
\label{sec:review}

In this section we review basic definitions and properties of  open book foliations. For details see \cite[\S 2]{ik1-1}.

Let $(S,\phi)$ be an open book decomposition of a closed oriented $3$-manifold $M$, where $S=S_{g, r}$ is a genus $g$ surface with $r$ boundary components, and $\phi \in {\rm Diff}^+(S, \partial S)$ is an orientation preserving diffeomorphism of $S$ fixing the boundary point-wise.   
The manifold $M$ is often denoted by $M_{(S, \phi)}$. 
Let $B$ denote the {\em binding} of the open book and $\pi:M \setminus B \rightarrow S^{1}$ the fibration whose fiber $S_{t} := \pi^{-1}(t)$ is called a {\em page} where the parameter $t \in [0, 1]/\sim \ = S^1$.

Let $\xi= \xi_{(S,\phi)} = \textrm{ker}\,\alpha$ be the contact structure on $M$ supported by $(S,\phi)$ through the Giroux-correspondence \cite{g}: That is, $\alpha>0$ on the binding $B$ and $d\alpha$ is a positive area form on each page $S_t$. See \cite{TW} for an explicit construction of such $\xi$. 
We say that an oriented link $L$ in $M_{(S,\phi)}$ is in {\em braid position} with respect to the open book $(S,\phi)$ if $L$ is disjoint from the binding and positively transverse to  every page $S_{t}$. 
The algebraic intersection number of $L$ and the page $S_0$ is called the {\em braid index} of $L$. 
Thanks to Bennequin~\cite{Ben}, Mitsumatsu and Mori~\cite{MM}, and Pavelescu~\cite{Pav, P2}, any transverse link in a contact 3-manifold $(M, \xi)_{(S, \phi)}$ can be transversely isotoped to a closed braid in $(S,\phi)$. Conversely, a closed braid in $(S,\phi)$ is naturally regarded as a transverse link in $(M,\xi)$. Hence from now on, we always assume that every (transverse) link is in braid position.

Throughout this paper, $L$ is a closed (possibly empty) braid in $(S, \phi)$ and 
$F$ is an oriented connected compact surface embedded in $M$ such that
\begin{itemize}
\item $F$ is a closed surface in $M \setminus L$, or  
\item $F$ is a Seifert surface of $L$, i.e., $\partial F=L$.
\end{itemize}
Consider the singular foliation $\mF=\mF(F)$ of $F$ induced by the intersections of the pages $\{S_t\}_{t\in [0,1]}$ and $F$. 
We call each connected component of $F \cap S_t$ a {\em leaf}.  

\begin{definition}\label{def of OBF}
We say that $\mF$ is an {\em open book foliation}, denoted by $\F(F)$, if the following four conditions are satisfied.
\begin{description}

\item[($\mathcal{F}$ i)]

The binding $B$ pierces the surface $F$ transversely in finitely many points. 
Moreover, $p \in B \cap F$ if and only if there exists a disc neighborhood $N_{p} \subset \Int(F)$ of $p$ on which the foliation $\mF(N_p)$ is radial with the node $p$ (see the top sketches in Figure~\ref{fig:sign}). 
We call the singularity $p$ an {\em elliptic} point.

\item[($\mathcal{F}$ ii)] 
The leaves of $\mF$ along $\partial F$ are transverse to $\partial F$. 

\item[($\mathcal{F}$ iii)] 
All but finitely many fibers $S_{t}$  intersect $F$ transversely.
Each exceptional fiber is tangent to $\Int(F)$ at a single point. 
In particular, $\mF$ has no saddle-saddle connections.

\item[($\mathcal{F}$ iv)] 
All the tangencies of $F$ and fibers are Morse type saddles (see the bottom sketches of Figure~\ref{fig:sign}). 
We call them {\em hyperbolic} points.
\end{description}
\end{definition}

\begin{definition}

We say that a leaf $l$ of $\F(F)$ is {\it regular} if $l$ does not contain a tangency point and is {\it singular} otherwise. 
The regular leaves are classified into the following three types:
\begin{enumerate}
\item[a-arc]: An arc one of whose endpoints lies on $B$ and the other lies on $L$.
\item[b-arc]: An arc whose endpoints both lie on $B$.
\item[c-circle]: A simple closed curve.
\end{enumerate} 
\end{definition}

\begin{theorem}\cite[Theorem 2.5, Proposition 2.6]{ik1-1}
\label{theorem:existence}
By isotopy that fixes the transverse link type of the boundary $\partial F$ $($if $\partial F$ exists$)$, every surface $F$ admits an open book foliation $\F(F)$. 
Moreover, given an open book foliation $\F(F)$ we can perturb $F$ (by fixing $\partial F$ if it exists) so that the new $\F(F)$ contains no c-circles.
\end{theorem}

We say that an elliptic point $p$ is {\em positive} (resp. {\em negative}) if the binding $B$ is positively (resp. negatively) transverse to $F$ at $p$.
The sign of a hyperbolic point $p$ is {\em positive} (resp. {\em negative}) if the orientation of the tangent plane $T_p F$ does (does not) agree with the orientation of $T_p S_t$. The sign of a singular point $x$ is denoted by $\sgn(x)$.
See Figure \ref{fig:sign}, where we describe an elliptic point by a hollowed circle with its sign inside, a hyperbolic point by a dot with the sign nearby, and positive normals $\vec n_F$ to $F$ by dashed arrows.
\begin{figure}[htbp]
 \begin{center}
\includegraphics[width=130mm]{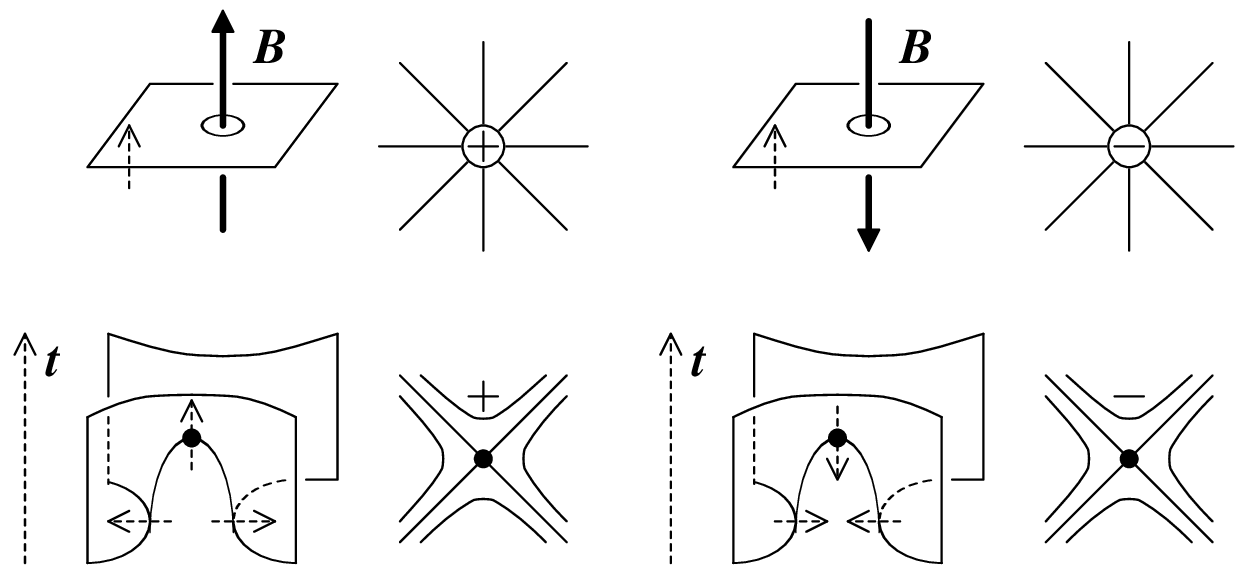}
 \end{center}
 \caption{(Top) $\pm$ elliptic points. (Bottom) $\pm$ hyperbolic points.}
  \label{fig:sign} 
\end{figure}

\begin{definition}\label{def of e}
We denote the number of positive (resp. negative) elliptic points of $\F(F)$ by $e_{+}(\F(F))$ (resp. $e_{-}(\F(F))$). Similarly, the number of positive (resp. negative) hyperbolic points is denoted by $h_{+}(\F(F))$ (resp. $h_{-}(\F(F))$). 
\end{definition}

Recall that the {\em characteristic foliation} $\cF(F)$ of an embedded surface $F \subset (M,\xi)$ is a singular foliation obtained by integrating the singular line filed $\xi \cap TF$ on $F$. 
The next theorem shows a close relation between $\F(F)$ and $\cF(F)$. 
More comparisons of $\F(F)$ and $\cF(F)$ can be found in \cite[Remark 2.22]{ik1-1}.

\begin{theorem}[Structural stability {\cite[Theorem 2.21]{ik1-1}}] \label{identity theorem}
Assume that a surface $F$ in $M_{(S, \phi)}$ admits an open book foliation $\F(F)$. 
There exists a contact structure $\xi$ on $M_{(S, \phi)}$ supported by the open book $(S, \phi)$ such that 
$e_{\pm}(\F(F))=e_{\pm}(\cF(F))$ and $h_{\pm}(\F(F))=h_{\pm}(\cF(F))$.

Moreover, if $\F(F)$ contains no $c$-circles then $\F(F)$ and $\cF(F)$ are topologically conjugate, namely there exists a homeomorphism of $F$ that takes $\F(F)$ to $\cF(F)$. 
\end{theorem}

The above result yields the following:

\begin{proposition}\label{sl-formula-1}
\cite[Propositions 2.11 and 3.2]{ik1-1}
Suppose that $F \subset M_{(S, \phi)}$ is a surface admitting an open book foliation. 
\begin{enumerate}
\item 
If $\partial F$ is non-empty then the self linking number is $$sl(\partial F, [F]) = -\langle e(\xi),[F] \rangle = -(e_{+}-e_{-})+(h_{+}-h_{-}).$$
\item 
The Euler characteristic is $\chi(F) = (e_+ + e_-) - (h_+ + h_-).$
\end{enumerate}
\end{proposition}

\begin{definition}
If $p$ is a hyperbolic point whose nearby regular leaves are a-arcs then we say that $p$ is of {\em type aa}. 
Similarly, we can define {\em types ab, bb, ac, bc}, and {\em cc} for hyperbolic points. 

An {\em aa-tile} $($resp. {\em ab-tile, bb-tile}$)$ is the closure of regular leaves near a type aa $($resp. ab, bb$)$ hyperbolic point and its topological type is a disc. 
An {\em ac-annulus} $($resp. {\em bc-annulus}$)$ 
is the closure of regular leaves near a type ac $($resp. bc$)$  hyperbolic point and its topological type is an annulus. 
A {\em cc-pants} is the closure of regular leaves near a type cc hyperbolic point and its topological type is a pair of pants. 
See Figure ~\ref{region}. 
\end{definition}
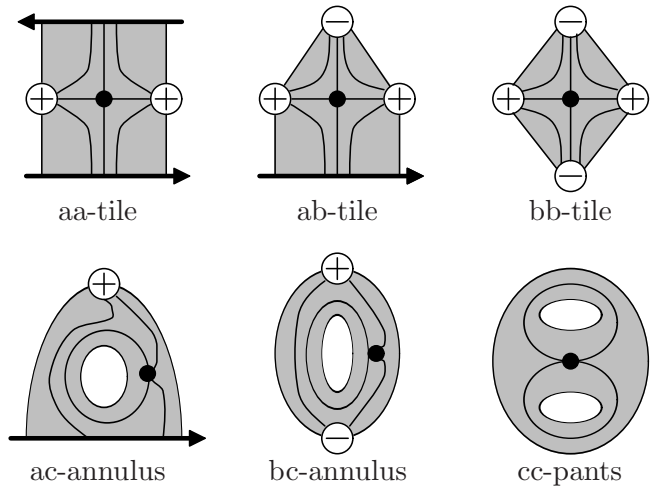
\begin{figure}[htbp]
\begin{center}
\SetLabels
(0.15*0.55)  aa-tile\\
(0.5*0.55)    ab-tile\\
(0.84*0.55)  bb-tile\\
(0.15*0.04)  ac-annulus\\
(0.5*0.04)    bc-annulus\\
(0.84*0.04)  cc-pants\\
\endSetLabels
\strut\AffixLabels{\includegraphics*[scale=0.5, width=90mm]{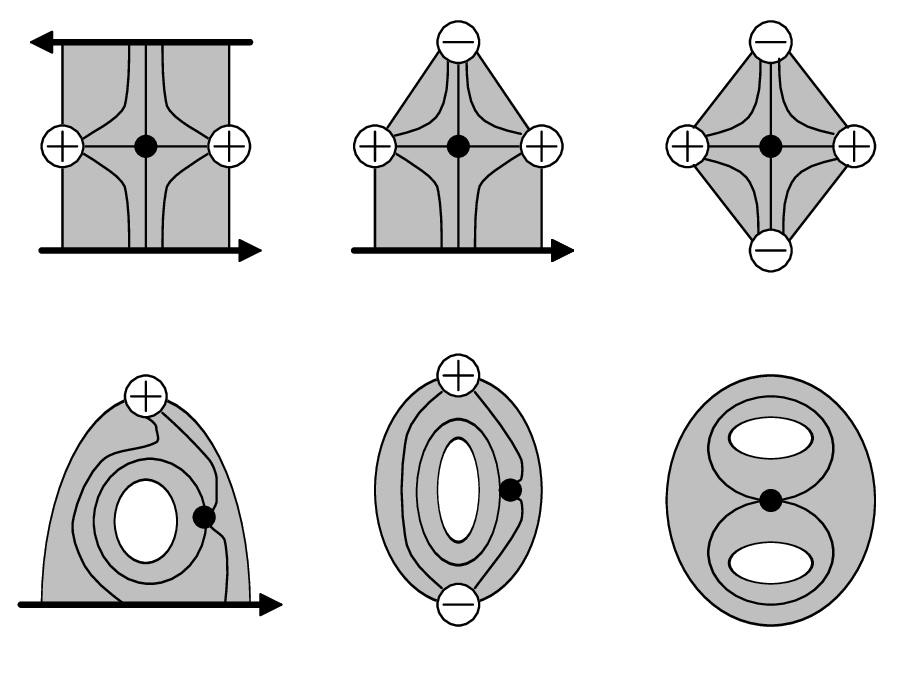}}
\caption{Neighborhood regions near hyperbolic points.}\label{region}
\end{center}
\end{figure}

\begin{proposition}[Region decomposition {\cite[p.441]{bm1} \cite[Proposition 2.15]{ik1-1}}]
\label{prop:region}

If $\F(F)$ contains a hyperbolic point then the surface $F$ 
can be decomposed into a union of aa-tiles, ab-tiles, bb-tiles, ac-annuli, bc-annuli and cc-pants. 
Such a decomposition is called a {\em region decomposition}. 
\end{proposition}

Suppose that $R$ is a component of the region decomposition of $F$.  
If $R$ is an aa-tile, ac-annulus, bc-annuls or cc-pants then  some parts of the boundary $\partial R$ are possibly identified in $F$.
In such a case we say that $R$ is {\em degenerate}, see Figure~\ref{fig:degeneratebc}.

Both the surface $F$ and the ambient manifold $M$ are oriented. 
Let $\vec n_F$ be a positive normal to $F$. 
We orient each leaf of $\F(F)$, for both regular and singular, so that if we stand up on the positive side of $F$ and walk along a leaf in the positive direction then the positive side of the corresponding page $S_t$ of the open book is on our left. 
In other words, at a non-singular point $p$ on a leaf $l \subset (S_t \cap F)$ let $\vec n_S$ be a positive normal to $S_t$ then $X_{ob}= \vec n_S \times \vec n_F$ is a positive tangent to $l$. 
As a consequence, positive/negative elliptic points become sources/sinks of the vector field $X_{ob}$.

To illustrate open book foliations it is convenient to use a {\em movie presentation}, see \cite[Section 2.1.5]{ik1-1}), that is a set of slices $(S_t \cap F, S_t)$. 
In a movie presentation, a hyperbolic singularity appears when the configuration of leaves changes. See the passage in Figure~\ref{fig:hyperbolic}.  As $t$ increases, two regular leaves $l_{1}$ and $l_{2}$ approach along a properly embedded arc $\gamma$ in $S_t \setminus (S_t \cap F)$ (the dashed arc) connecting $l_{1}$ and $l_{2}$. At a critical moment $l_{1}$ and $l_{2}$ form a singular leaf containing a hyperbolic point, and then the configuration changes. Up to isotopy the arc $\gamma$ is uniquely determined.
We call $\gamma$ a {\em describing arc} of the hyperbolic singularity. 
\begin{figure}[htbp]
\SetLabels
(.1*.78) $\gamma$\\
(.05*.86) $l_1$\\
(.18*.86)   $l_2$\\
(.46*.59) $\gamma$\\
(.36*.55) $l_1$\\
(.58*.55)   $l_2$\\
\endSetLabels
\strut\AffixLabels{\includegraphics*[width=100mm]{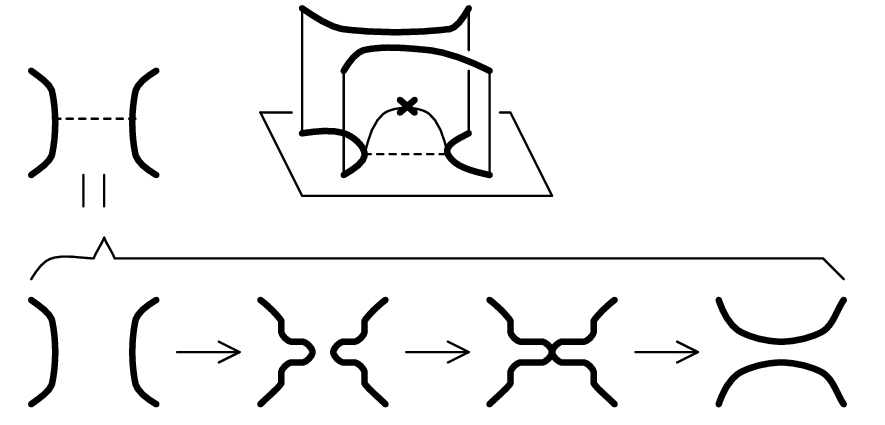}}
 \caption{A movie presentation of a hyperbolic singularity.}
 \label{fig:hyperbolic}
\end{figure}
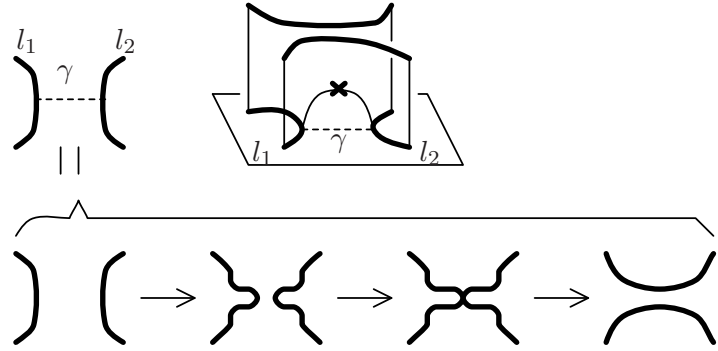

\begin{remark}\label{remark sign describing arc}
The sign of a hyperbolic point can be understood in terms of a describing arc and orientations of nearby leaves: 
Recall that each leaf is both oriented and transversely oriented. 
Figure~\ref{fig:sign} shows that: 
{\em A hyperbolic point is positive (resp. negative) if and only if the positive normals $\vec n_F$ point out of (resp. into) the describing arc. }
This observation will be important in many places in this paper. See Figure~\ref{fig:describing-arc}. 
\begin{figure}[htbp]
\SetLabels
\endSetLabels
\strut\AffixLabels{\includegraphics*[width=90mm]{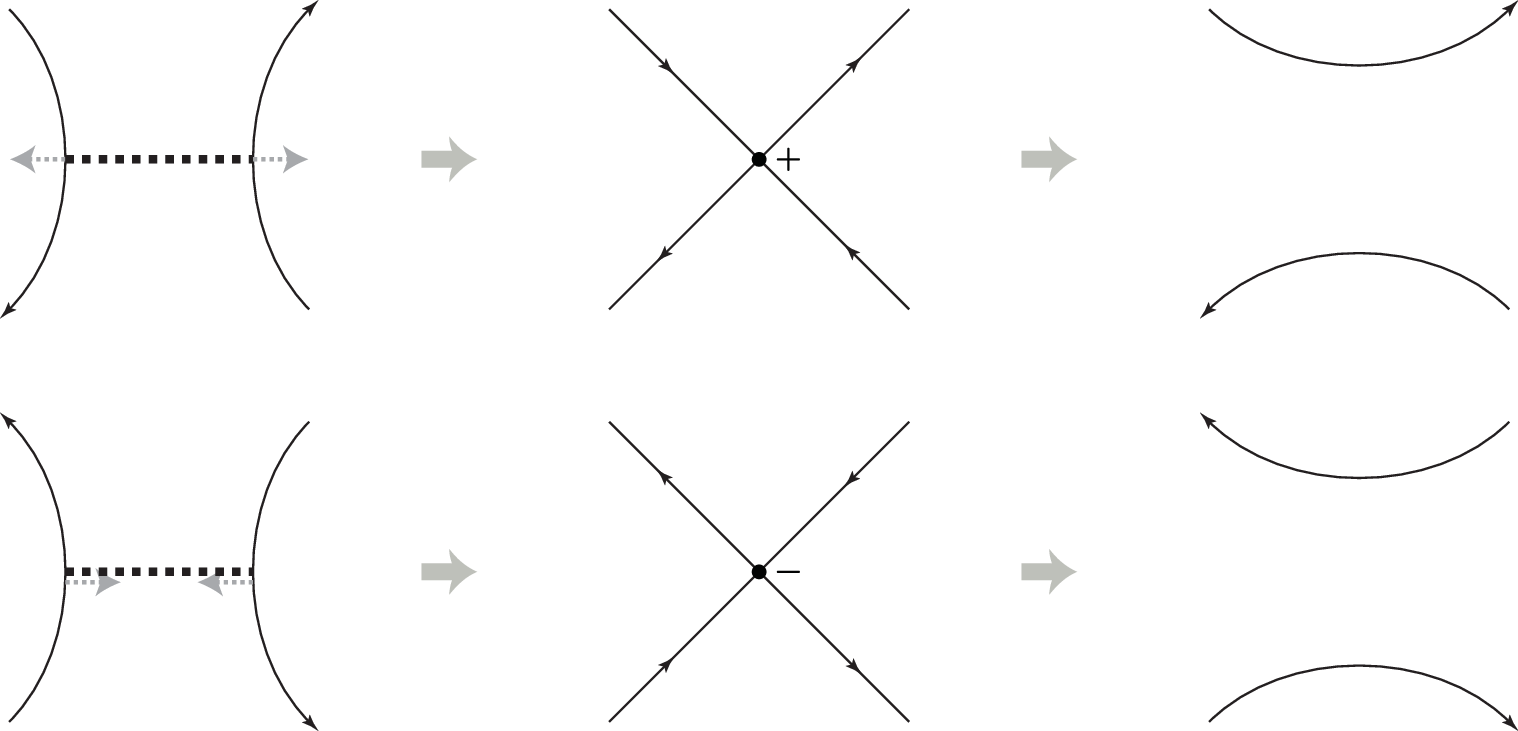}}
\caption{(Remark~\ref{remark sign describing arc}) 
Relation among the signs of hyperbolic points,  describing arcs (thick dashed) and normal vectors $\vec n_F$ (dashed gray arrows).}
\label{fig:describing-arc}
\end{figure}
\end{remark}

\begin{definition}\label{def:negativity-graph}
The {\em graph $G_{--}$} is a graph embedded in $F$. 
The edges of $G_{--}$ are the unstable separatrices for negative hyperbolic points in $aa$-, $ab$- and $bb$-tiles. 
See Figure~\ref{fig:neggraph}. 
We regard negative hyperbolic points as part of the edges. 
The vertices of $G_{--}$ are the negative elliptic points in $ab$- and $bb$-tiles and the end points of the edges of $G_{--}$ that lie on $\partial F$, called the {\em fake} vertices. 

By the same way we can define $G_{++}$ the graph consists of positive elliptic points and stable separatrices of positive hyperbolic points. 
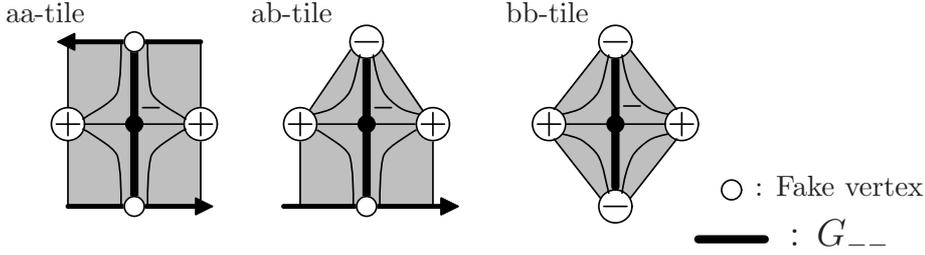
\begin{figure}[htbp]
 \begin{center}
\SetLabels
(0.03*0.95) aa-tile\\
(0.3*0.95) ab-tile\\
(0.58*0.95) bb-tile\\
(0.90*0.1) \Large : $G_{--}$\\
(0.9*0.29) : Fake vertex\\
  \endSetLabels
\strut\AffixLabels{\includegraphics*[scale=0.5, width=120mm]{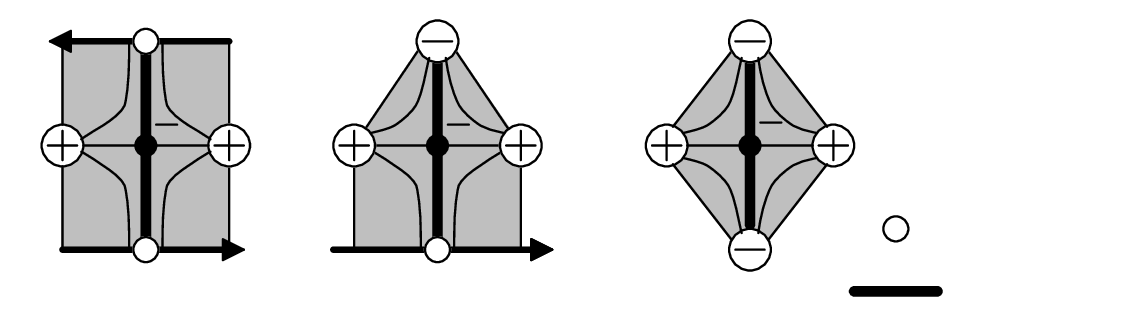}}
 \caption{The graph $G_{--}$.}\label{fig:neggraph}
  \end{center}
\end{figure}
\end{definition}

Finally, we recall the definition of transverse overtwisted discs. 

\begin{definition}\label{def:trans-ot-disc}
Let $D \subset M_{(S, \phi)}$ be an oriented disc whose boundary is a positively braided unknot.
If the following are satisfied $D$ is called a {\em transverse overtwisted disc}: 
\begin{enumerate}
\item $G_{--}$ (Definition~\ref{def:negativity-graph}) is a connected tree with no fake vertices.
\item $G_{++}$ is homeomorphic to $S^1$.
\item $\F(D)$ contains no c-circles. 
\end{enumerate}
\end{definition}
 
\begin{theorem}{\cite[Proposition 4.2 and Corollary 4.6]{ik1-1}}
The contact structure $\xi_{(S, \phi)}$ is overtwisted if and only if there exists a transverse overtwisted disc.
\end{theorem}

Definition~\ref{def:trans-ot-disc} implies that the region decomposition of a transverse overtwisted disc consists of ab- and bb-tiles with $\sgn({\textrm{ab-tile}})=+1$ and $\sgn({\textrm{bb-tile}})=-1$. 
Hence $G_{++}$ lives in the ab-tiles and $G_{--}$ lives in the bb-tiles.

\section{Essential open book foliation }\label{sec: essential ob foliation} 

In this section we introduce {\em essential} open book foliations.
As before $L$ is a closed (possibly empty) braid in the open book $(S, \phi)$ and $F$ is either a closed surface in $M-L$ or a Seifert surface of $L$.

\begin{definition}\label{def of essential}
\begin{enumerate}
\item
A $b$-arc $l$ in a page $S_{t}$ is called {\em essential} (resp. {\em strongly essential}) if $l$ is an essential arc in $S_{t}\setminus(S_{t} \cap L)$ (resp. in $S_t$). 
\item
A $c$-circle $l$ in $S_{t}$ is called {\em essential} if $l$ is an essential simple closed curve in $S_{t}\setminus(S_{t} \cap L)$. 
\item
An open book foliation $\F(F)$ is called {\em essential} if all the $b$-arcs are essential ($c$-circles need not be essential).
\item
An elliptic point $v\in\F(F)$ is called {\em essential} (resp. {\em  strongly essential}) if every $b$-arc that ends at $v$ is essential (resp. strongly essential).
\end{enumerate}
\end{definition}

Our first result is an improvement of Theorem \ref{theorem:existence}.

\begin{theorem}[Non-sphere case] 
\label{theorem:weak}
Suppose that $F$ is incompressible and not a sphere. 
There exists a surface $F'$ admitting an essential open book foliation and essential spheres $\mathcal S_1, \dots, \mathcal S_k$ $(k \geq 0)$ such that
$$F = F' \# \mathcal S_1 \# \dots \# \mathcal S_k$$
up to isotopy that fixes $L$. Moreover if $F$ does not intersect a binding component $C$ then nor does $F'$. 
\end{theorem}

\begin{theorem}[Sphere case]
\label{theorem:weak-sphere}
If $F$ is an essential sphere then there exists an essential sphere $F'$ that admits an essential open book foliation. 
Moreover, if $F$ does not intersect a binding component $C$ then nor does $F'$.
\end{theorem}

\begin{proof}[Proof of Theorems \ref{theorem:weak} and \ref{theorem:weak-sphere}]

We put the surface $F$ in general position so that it admits a singular foliation $\mF$ satisfying the properties {\bf ($\mathcal{F}$ i)}, {\bf ($\mathcal{F}$ ii)}, {\bf ($\mathcal{F}$ iii)} in Definition~\ref{def of OBF} and\\

\noindent
{\bf($\mathcal{F}$ iv$'$)}: The type of each tangency in {\bf ($\mathcal{F}$ iii)} is either a saddle or a local extremum.\\

\noindent
Note that {\bf ($\mathcal{F}$  iv$'$)} is weaker than {\bf ($\mathcal{F}$  iv)}. 

First we remove all the inessential b-arcs (cf. \cite[Lemma 1.2]{bf}): 
Let $l$ be an innermost inessential b-arc of $\mF$ in a regular page $S_{t}$ that cobounds a disc $\Delta \subset S_{t}$ with the binding, and $\Delta \cap L$ is empty. 
Since $l$ is innermost $\Delta$ contains no $b$-arcs.
If $\Delta$ contains $c$-circles $c_{1},\ldots,c_{k}$, 
let $A_{i} \subset F$ be a small annular neighborhood of $c_{i}$ with no singularities. We push each annulus $A_{i}$ out of $\Delta$ across the binding $B$ as shown in Figure~\ref{fig:removeb1}. 
This does not create new b-arcs but new local extrema appear. 
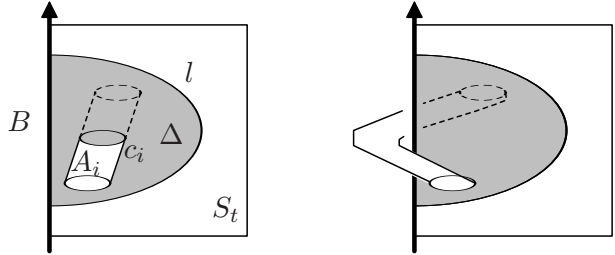
\begin{figure}[htbp]
 \begin{center}
\SetLabels
(0.25*0.45) $\Delta$\\
(0*.5) $B$\\
(0.34*0.18) $S_{t}$\\
(0.28*0.68) $l$\\
(.19*.4) $c_i$\\
(.11*.35) $A_i$\\
  \endSetLabels
\strut\AffixLabels{\includegraphics*[scale=0.5, width=80mm]{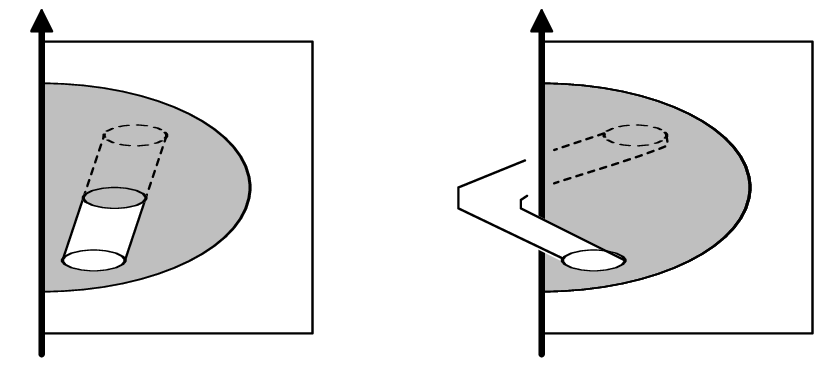}}
 \caption{Removing a c-circle in $\Delta$.}
 \label{fig:removeb1}
  \end{center}
\end{figure}

Now $l$ is boundary parallel in $S_{t} \setminus (S_{t} \cap F)$. 
We push the surface $F$ along $\Delta$ as shown in Figure~\ref{fig:removeb}.
\begin{figure}[htbp]
\begin{center}
\SetLabels
(0.25*0.58) $\Delta$\\ 
(.25*.68) $l$\\
\endSetLabels
\strut\AffixLabels{\includegraphics*[scale=0.5, width=80mm]{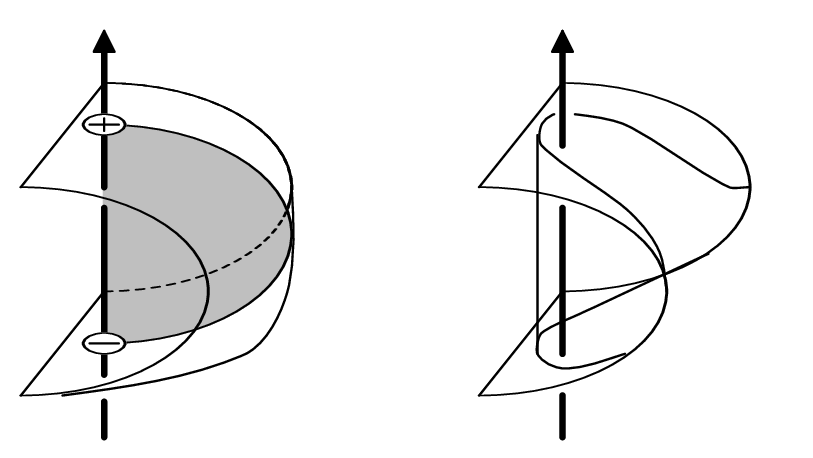}}
 \caption{Removing an inessential b-arc.}
 \label{fig:removeb}
  \end{center}
\end{figure}
As a consequence two elliptic points that are the ends of $l$ disappear and new hyperbolic points and local extremal points appear. 
Since the number of elliptic points in the original $\mF$ is finite, applying the operation finitely many times $\mF$ has no inessential b-arcs.

The next step is to remove local extremal points.
The idea and method are similar to Roussarie-Thurston's general position argument \cite{r,t} for surfaces in taut foliation.

Let $p$ be a local minimum. 
(Parallel arguments hold for local maxima.)  
Suppose that $p$ is contained in the page $S_{t_0}$. 
The singular foliation $\mF$ near $p$ consists of a family of concentric c-circles each of which bounds a disc $X_t \subset S_t$ for $t_0 < t < t_0 +\e$. 
Let $q_1$ denote the first saddle point the c-circles $\partial X_t$ encounter as $t$ ($\geq t_0+\e$)  increases.   
The saddle point $q_1$ can be formed in one of the following three ways, see also Figure \ref{fig:saddletypes0}:
\begin{itemize}
\item[(Type I)] 
The describing arc is contained in $S_t \setminus X_t$ and joining a point on $\partial X_t$ and a point on another leaf of $\mF$. 
\item[(Type II)] 
The describing arc is contained in $S_t \setminus X_t$ and joining two points of $\partial X_t$. 
\item[(Type III)] 
The describing arc is contained in $X_{t}$ and joining two points of $\partial X_t$.
\end{itemize}
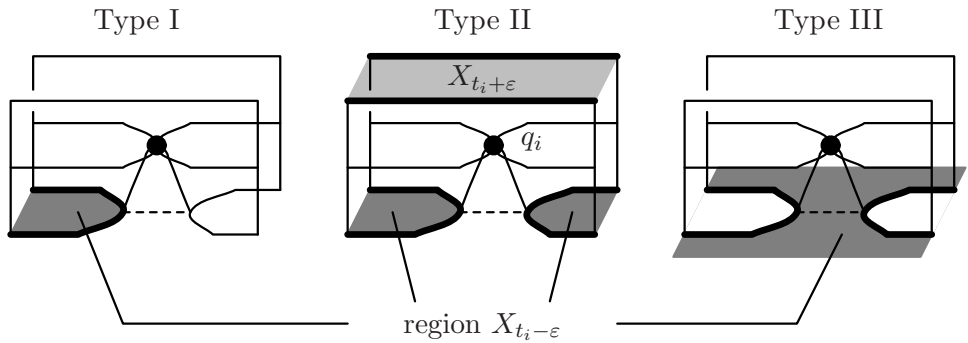
\begin{figure}[htbp]
\begin{center}
\SetLabels
(0.15*0.95) Type I\\
(.5*0.95) Type II\\
(.85*0.95) Type III\\
(.5*.04) region $X_{t_i -\e}$\\
(.5*.78) $X_{t_i + \e}$\\
(.55*.6) $q_i$\\
\endSetLabels
\strut\AffixLabels{\includegraphics*[width=130mm]{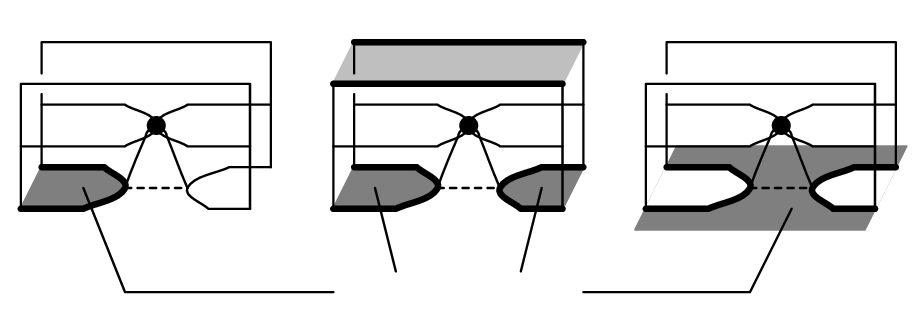}}
\caption{Three types of saddles. Dashed arcs are describing arcs.} \label{fig:saddletypes0}
\end{center}
\end{figure}
We say that a local minimum $p$ is of \emph{Type I, II, III} if the saddle $q_1$ is of Type I, II, III, respectively. See Figure \ref{fig:saddletypes}.
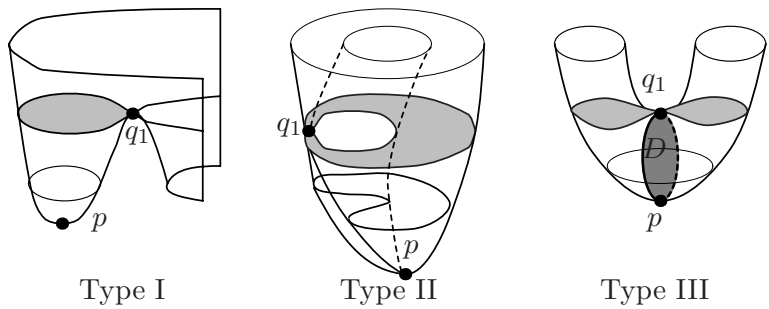
\begin{figure}[htbp]
 \begin{center}
\SetLabels
(0.15*-.07) Type I\\
(.5*-.07) Type II\\
(.85*-.07) Type III\\
(.85*.45) $D$\\
(.17*.52) $q_1$\\
(.37*.55) $q_1$\\
(.85*.71) $q_1$\\
(.12*.2) $p$\\
(.53*.1) $p$\\
(.85*.2) $p$\\
\endSetLabels
\strut\AffixLabels{\includegraphics*[width=100mm]{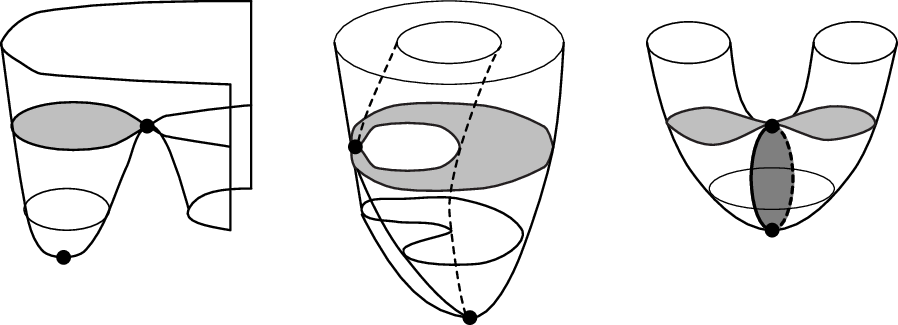}}
\caption{Three types of local minima $p$.}
 \label{fig:saddletypes}
\end{center}
\end{figure}

Suppose that $q_1$ is sitting on the page $S_{t_1}$.

Assume that the first saddle $q_1$ is of Type II. 
Inductively, we can define
a family of connected regions $\{ X_t \ | \ t_0 \leq t \leq t_n \}$ for some $n = n(p) \in \mathbb N$ and $t_0< t_1 < \dots < t_n$ (with the cyclic order of $S^1\cong \R/\Z$), such that:
\begin{itemize}
\item
$X_{t_0} = \{p\}$. 
\item 
For $i=1,\dots,n-1$ the region $X_{t_i}$ contains a Type II saddle $q_i \in S_{t_i}$ and the region $X_{t_i+\e}$ is obtained from $X_{t_i - \e}$ by attaching a $1$-handle corresponding to the describing arc as shown in the middle sketch of Figure~\ref{fig:saddletypes0}.
\item
For $t \in (t_i, t_{i+1})$ the region $X_t$ is isotopic to $X_{t_i + \e}$ and bounded by c-circles. 
\item
For $i=n$ the limit region $X_{t_n} := \lim_{t\to t_n} X_t$ contains a saddle $q_n \in S_{t_n}$ of either Type I or Type III. (For a saddle $q_i$, Types I, II, III can be define in the same way as $q_1$, according to the relation between the describing arc of $q_i$ and the region $X_t$.)
\item
If $q_1$ is of Type I or Type III then $n=n(p)=1$ and the family $\{ X_t \ | \ t_0 < t < t_{n=1} \}$ consists of discs.  
\end{itemize}
The first property shows that there is an inclusion map $\iota: X_{t_i - \e} \hookrightarrow X_{t_i+\e}$.
Using $\iota$ we may regard $X_{t_i - \e} \subset X_{t_i+\e}$.   
Also by the second property we can identify $X_{t_i+\e}$ with $X_{t_{i+1}-\e}$. 
Thus we may view $\{X_t\}$ as an increasing family of regions; 
\begin{equation}\label{inclusion}
X_{t_1-\e} \subset X_{t_2-\e} \subset\dots\subset X_{t_n-\e}. 
\end{equation}
For $t \in (t_0,t_n)$ let $$F_{[t_0,t]} := \bigcup_{t_0 \leq s \leq t } (\partial X_s)$$ be a subsurface of $F$.

\begin{remark}\label{rem:isotopic}
If the family of regions $\{X_s \ | \ t_0 \leq s \leq t\}$ contains no local minima then $F$ is ambient isotopic to $(F \setminus F_{[t_0, t]}) \cup X_t$.  
\end{remark}

We say that a local minimum $p$ is {\em innermost} if the family of regions $\{ X_t \ | \ t_0 < t < t_n \}$ contains no local minima, see Figure~\ref{fig:not-innermost}.
\begin{figure}[htbp]
\begin{center}
\SetLabels
(.22*0) $p$\\
(.73*.65) $q_1$\\
(.03*.2) $X_t$\\
\endSetLabels
\strut\AffixLabels{\includegraphics*[width=30mm]{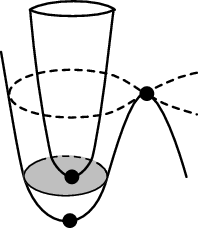}}
\caption{The local minimum $p$ is not innermost.}
\label{fig:not-innermost}
\end{center}
\end{figure}

Let
\begin{eqnarray*}
m(F) &=& \mbox{the number of local minima of } F,\\
m_{\rm II}(F) &=& \mbox{the number of Type II local minima of } F,\\
h(F) &=& \mbox{the number of saddle tangencies of } F \mbox{ with the pages.}\\
\end{eqnarray*}
Let $\{S_{\tau_i}\}_{i=0,\dots,k}$ ($0<\tau_0<\tau_1<\dots<\tau_k <1$) be the set of singular pages for $F$, namely the pages that are tangent to $F$ in a single point. 
(Recall that each page contains at most one singularity.) 
Choose $s_0,\dots,s_k$ such that $0<s_0 < \tau_0$ and $\tau_{i-1}<s_i < \tau_i$. 
We define a {\em width} of the surface $F$ by 
$$
w(F)= \sum_{i=0}^k \#(S_{s_i}\cap F)
$$
where $\#(S_{s_i}\cap F)$ denotes the number of components of $S_{s_i}\cap F$ (i.e., the number of the leaves in the page $S_{s_i}$).

Let $p \in S_{\tau_i}$ be a local minimum and $\gamma \subset S_{\tau_i}$ be a path connecting $p$ and a point on the boundary $\partial S_{\tau_i}$. 
Let $\mathsf i (\gamma, F \cap S_{\tau_i})$ be the geometric intersection number of the path $\gamma$ and the leaves $F \cap S_{\tau_i}$.
We define the \emph{complexity} of the local minimum $p$ by 
\[ c(p) = \min \{ \mathsf{i}(\gamma, F \cap S_{\tau_i})\: | \: \gamma \text{ is a path  connecting } p \text{ and a point on } \partial S_{\tau_i} \} \]

We consider the complexity of $F$ measured by a four-tuple $$
(m(F), h(F), m_{\rm II}(F), w(F)),
$$
ordered by the lexicographic order.  
We prove that, by isotopy and desumming spheres, we can reduce the complexity until we remove all the local minima of $F$.  
Let $p$ be a local minimum realizing the maximal $c(p)$ among all the local minima of $F$. 
We have four cases to consider. 

\noindent
{\bf Case I}: $p$ is not innermost.

Suppose that $p$ lies in the page $S_{t_0}$. Let $\{X_t \ | \ t_0 \leq t \leq t_n\}$ be the family of regions emerging from $p$ as defined above. 
Let $p'$ be a local minimum contained in the region $\{X_t \ | \ t_0 < t < t_n\}$.
We may assume that $p' \in S_{t_*}$ for $t_* \in (t_0, t_n)$ and there are no local minima in the region $\{X_t \ | \  t_0 < t< t_*\}$.

By Remark~\ref{rem:isotopic} we can compress (rescale) the subsurface $F_{[t_0,t_*]}$ into a small interval $[t_*-\e, t_*]$ to get a surface, $F'$. 
Let $p''$ denote the local minimum of $F'$ corresponding to the original local minimum $p \in F$.  
See Figure~\ref{fig:compress-1}. 
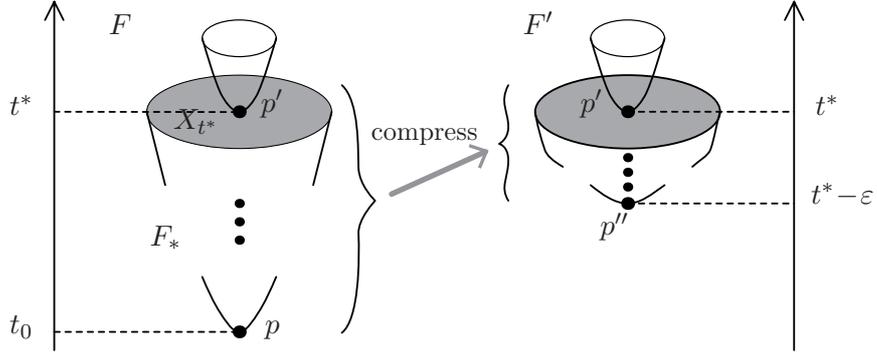
\begin{figure}[htbp]
\begin{center}
\SetLabels
(0.16*0.3) $F_*$\\
(.1*.9) $F$\\
(0.65*.9) $F'$\\
(0.5*0.6) {\small compress}\\
(-.03*0.05) $t_{0}$\\
(-.03*0.67) $t^{*}$\\
(1.05*0.42) $t^{*}\!-\!\varepsilon$\\
(1.03*0.67) $t^{*}$\\
(0.2*0.63) $X_{t^{*}}$\\
(0.3*0.05) $p$\\
(0.3*0.68) $p'$\\
(0.75*0.33) $p''$\\
(0.72*0.68) $p'$\\
\endSetLabels
\strut
\AffixLabels{\includegraphics*[width=100mm]{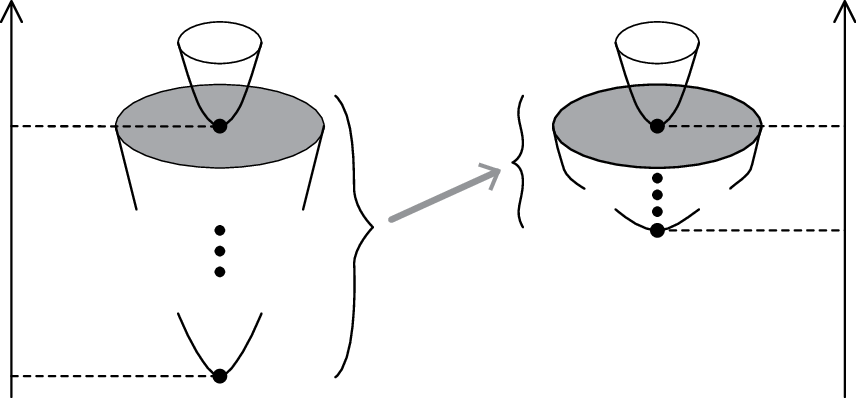}}
\caption{Compressing the subsurface $F_{[t_0, t_*]}$.}
\label{fig:compress-1}
\end{center}
\end{figure}

Note that $m(F)=m(F')$, $h(F)=h(F')$ and $m_{\rm II}(F)=m_{\rm II}(F')$ and $w(F)\leq w(F')$. 

We claim that $w(F') < w(F)$.

Assume that $w(F')=w(F)$. 
This means $0<t_* - t_0 <1$ and in the interval $(t_0, t_*)$ all the tangencies (if exist) are exactly the Type II tangencies $q_1,\dots,q_m$ ($m \leq n-1$) associated to $p$. 
Thus we have $c(p)=c(p'').$
Since a component of $\partial X_{t_*}$ must contribute to $c(p')$ we have $$c(p)=c(p'')=c(p')-1.$$
This contradicts the maximality assumption of $c(p)$. 

\noindent
{\bf Case II}: $p$ is innermost and of Type I.

In this case $p$ and the Type I saddle $q_1$ can be removed by flattening the bump. 
See Figure~\ref{fig:bump}. 
This can be done without affecting other part of $F$.  
As a consequence, we obtain a surface $F'$ with $m(F')<m(F)$.
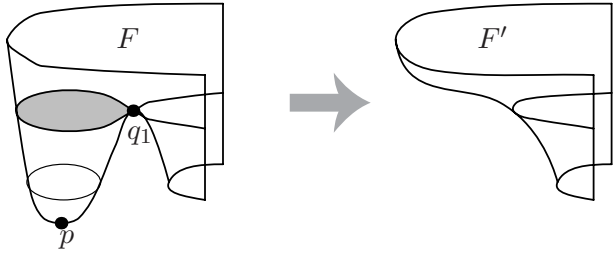
\begin{figure}[htbp]
\begin{center}
\SetLabels
(.2*.8) $F$\\
(.8*.8) $F'$\\
(.1*-.06) $p$\\
(.22*.4) $q_1$\\
\endSetLabels
\strut\AffixLabels{\includegraphics*[width=80mm]{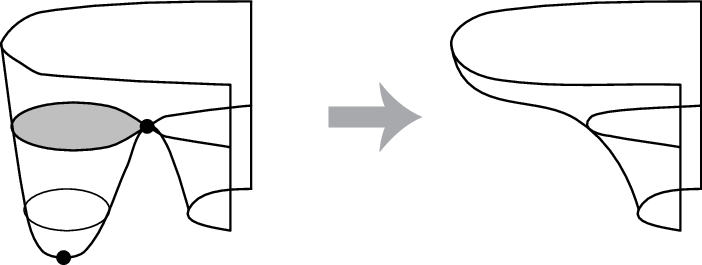}}
\caption{(Case II) Flattening a bump.}
\label{fig:bump}
\end{center}
\end{figure}

\noindent
{\bf Case III}: $p$ is innermost and of Type III.

Let $D$ (see Figure~\ref{fig:saddletypes}) be a disc bounded by the two trajectories from the Type III saddle $q_1$ to $p$ and intersecting each $\partial X_t$ ($t_0 < t < t_1$) transversely. 
The set $({\rm{Int}} D) \cap F$ is empty because of the innermost assumption. 

We cut the surface $F$ along the circle $\partial D$.  
Since $F$ is incompressible $\partial D$ bounds a disc $D' \subset F$ and we obtain a sphere $\mathcal{S} = D \cup D'$. 
Thus, the surface $F$ is the connected sum of a surface $F':= (F\setminus D') \cup D$ and $\mathcal S$, 
$$F = F' \# \mathcal S.$$
See Figure~\ref{fig:saddleremoveI-III}.  
If $\mathcal S$ bounds a ball then $F$ and $F'$ are isotopic. Otherwise $F$ and $F'$ are just homeomorphic. 
If $F$ is an essential sphere then both $F'$ and $\mathcal S$ are spheres. Since $F$ is essential at least one of them must be essential. 
We may assume that $F'$ is essential. 
\begin{figure}[htbp]
\begin{center}
\SetLabels
(.28*.6) $F$\\
(.7*.75)  $\mathcal{S}$\\
(.95*0.6)  $F'$\\
\endSetLabels
\strut\AffixLabels{\includegraphics*[width=90mm]{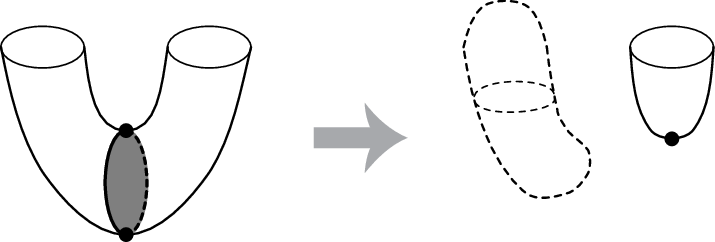}}
\caption{(Case III) Treatment of an innermost Type III local minimum.}
\label{fig:saddleremoveI-III}
\end{center}
\end{figure}

With a small perturbation we can put $F'$ so that its foliation $\mathcal F'$ satisfies ($\mathcal F$ i)-($\mathcal F$ iii) and ($\mathcal F$ iv$'$) and all the b-arcs are kept essential. 
We have $m(F')\leq m(F)$ and $h(F')<h(F)$.

\noindent
{\bf Case IV}: $p$ is innermost and of Type II.

Recall the saddles $q_1,\dots, q_n$ associated to $p$ such that $q_1,\dots, q_{n-1}$ are of Type II and $q_n$ is of Type I or III. 

\noindent
{\bf Case IV-a}: 
Assume that $q_n$ is of Type I.

Let $\iota: X_{t_{n-1} - \e} \hookrightarrow X_{t_n - \e}$ be the inclusion mentioned in (\ref{inclusion}). 
By the definitions of Type I and Type II saddles, 
a describing arc $\delta_n$ of the Type I saddle $q_n$ lies on $S_{t_n-\e}\setminus X_{t_n-\e}$, whereas a describing arc, $\delta_{n-1}$, of the Type II saddle $q_{n-1}$ lies on $S_{t_{n-1}-\e}\setminus X_{t_{n-1}-\e}$, i.e., $\iota(\delta_{n-1}) \subset X_{t_n - \e}$. See the left column of Figure~\ref{fig:typeII}.
This shows that by ambient isotopy we can put the saddles $q_{n-1}$ and $q_n$ in the same page, say $S_{t_n}$. 
With further isotopy we may ``move up (in the $t$-direction)'' the saddle $q_{n-1}$ to the level $t=t_n+\e$, see the right column of Figure~\ref{fig:typeII}. 

\begin{figure}[htbp]
\begin{center}
\SetLabels
(-.1*.92) {\small Type II}\\
(.15*.91) $\delta_{n-1}$\\
(-.1*.42) {\small Type I}\\
(.15*.42) $\delta_n$\\
(.9*.42) $\delta_n$\\
(.43*.9) \small $t_{n-1} -\e$\\
(.45*.4) \small $t_n - \e$\\
(.45*.24) \small $t_n+\e$\\
(.3*.94) $X_t$\\
(.48*.05) \large $t$\\
(1.1*.42) {\small Type I}\\
(.15*1.01) {\tiny some other leaf}\\
\endSetLabels
\strut\AffixLabels{\includegraphics*[width=75mm]{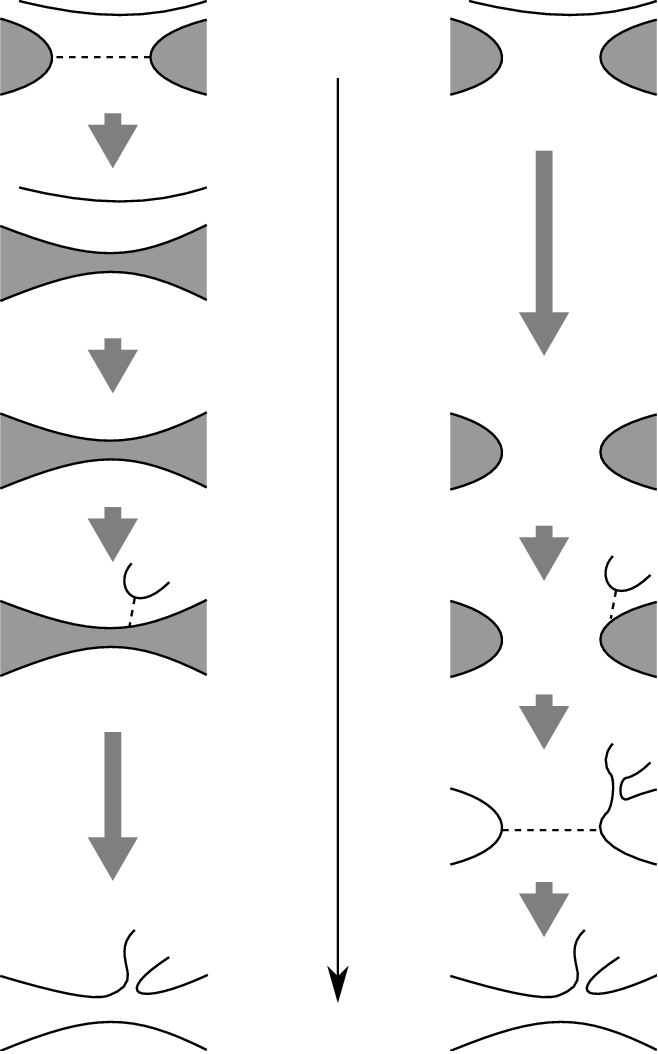}}
\caption{(Case IV-a) Movie presentations of $F$ (left) and $F'$ (right). The order of consecutive Type II and Type I saddles is exchangeable. (The shaded regions are $X_t$.)}
\label{fig:typeII}
\end{center}
\end{figure}
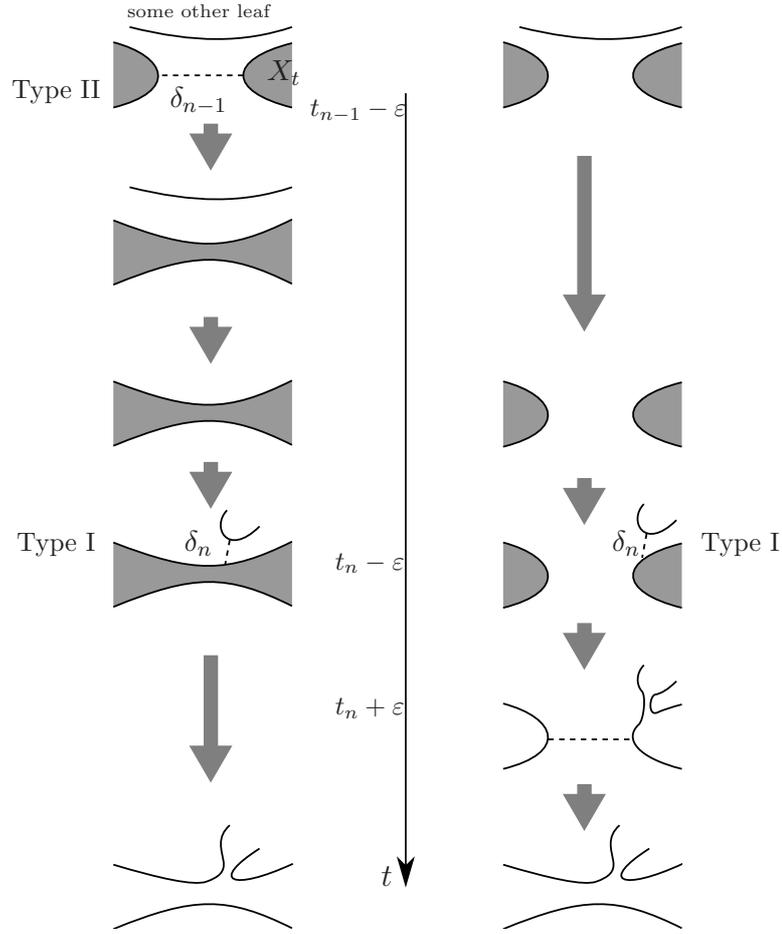

Repeating this for each of $q_{n-1}, q_{n-2},\dots, q_1$, we may assume that $q_1$ is of Type I, hence so is $p$. 
This operation preserves the number of local minima and the number of saddles. 
Thus the resulting surface, $F'$, satisfies 
$m(F')=m(F)$, $h(F')=h(F)$ and $m_{\rm II}(F') <m_{\rm II}(F)$.

\noindent
{\bf Case IV-b}: 
Assume next that $q_n$ is of Type III. 

We first note that, unlike Case IV-a, the order of subsequent Type II and Type III saddles may not be changeable.  
Such an example is given in the column (a) of Figure~\ref{fig:exchangeIII}, where the describing arcs for the Type II and Type III saddles, $q_3$ and $q_4$, are parallel.  
\begin{figure}[htbp]
\begin{center}
\SetLabels
(0*.95) (a)\\
(.6*.75) (b)\\
(-.1*.3) $t=t_4-\e$\\
(-.1*.27) ($4=n$)\\
(-.1*.1) $t=t_4+\e$\\
(-.1*.6) $t=t_3-\e$\\
(1.08*.65) $t=t_4-3\e$\\
(1.1*.1) $t=t_4+\e$\\
(.14*.92) $\delta_1$\\
(.14*.8) $\delta_2$\\
(.5*.6) $\delta_3$\\ 
(.25*.35) $\delta_4$\\
(.94*.65) $\delta'_4$\\
(.62*.34) $\delta'_1$\\
(.73*.3) $\delta'_2$\\
(.62*.23) $\delta'_3$\\
(1.1*.35) $t=t_4-2\e$\\
(1.1*.32) $t=t_4-\e$\\
(1.1*.29) $t= t_4$\\
\endSetLabels
\strut\AffixLabels{\includegraphics*[width=90mm]{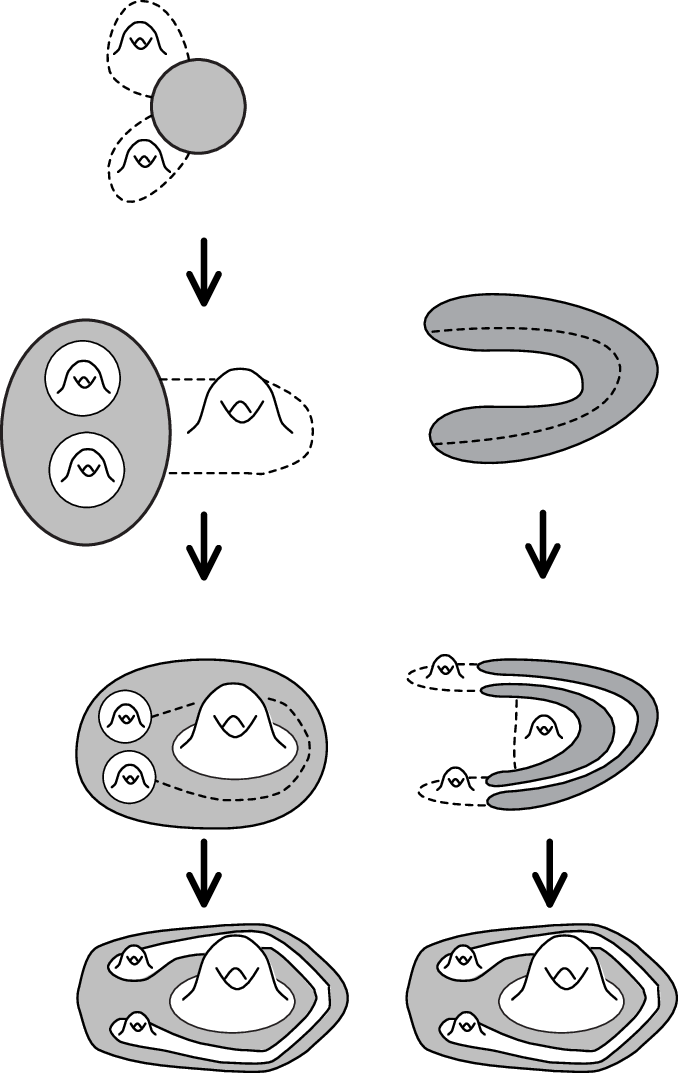}}
\caption{
(Case IV-b) Movie presentations of (a) $F$ and (b) $F'$, where $n=4$.}
\label{fig:exchangeIII}
\end{center}
\end{figure}

To deal with such examples we retake a new set of Type II saddles $q'_1,\dots,q'_{n-1}$ that can be exchangeable with the Type III saddle $q_n$. 

By the definition of $\{X_t \ | \ t_0 \leq t \leq t_n\}$, the region $X_{t_n - \e} = X_{t_{n-1}+\e}$ can be written as 
\[
X_{t_n-\e} = D \cup H_1 \cup \cdots \cup H_{n-1} 
\]
where $D$ is a disc and $H_1,\dots,H_{n-1}$ are $1$-handles corresponding to the Type II saddles $q_1,\dots,q_{n-1}$. (See the left sketch of Figure~\ref{fig:handle-decomp}.)

We consider another handle decomposition of $X_{t_n-\e}$. 
Let $\nu(\delta_n) \subset S_{t_n-\e}$ be a neighborhood of a describing arc, $\delta_n$, of the Type III saddle $q_n$. 
There exist $1$-handles $H'_1,\dots,H'_{n-1}\subset S_{t_n-\e}$ such that 
$$
X_{t_n-\e} = \nu(\delta_n) \cup H'_1\cup\dots\cup H'_{n-1}.
$$
Let $\delta'_1,\dots,\delta'_{n-1}$ be cores of the $1$-handles $H'_1,\dots,H'_{n-1}$. 
We observe that:  
\begin{equation}\label{observation-handle}
\mbox{The arcs } \delta'_1,\dots,\delta'_{n-1} \mbox { and } \delta_n \mbox{ are mutually disjoint. }
\end{equation}
See the right sketch of Figure~\ref{fig:handle-decomp}.
\begin{figure}[htbp]
\begin{center}
\SetLabels
(.12*.4) $D$\\
(.18*.56) $\delta_4$\\
(0*.9) $H_1$\\
(.05*.8) $\delta_1$\\
(0*0) $H_2$\\
(.05*.1) $\delta_2$\\
(.4*.5) $H_3$\\
(.3*.6) $\delta_3$\\
(1*.6) $\nu(\delta_4)$\\
(.9*.57) $\delta_4$\\
(.62*.8) $H'_1$\\
(.62*.2) $H'_3$\\
(.77*.5) $H'_2$\\
(.7*.85) $\delta'_1$\\
(.67*.1) $\delta'_3$\\
(.83*.45) $\delta'_2$\\
\endSetLabels
\strut\AffixLabels{\includegraphics*[width=120mm]{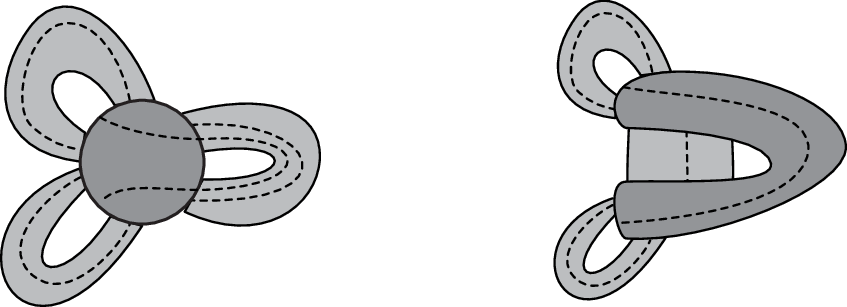}}
\caption{
Two handle decompositions of $X_{t_n -\e}$ where $n=4$.}
\label{fig:handle-decomp}
\end{center}
\end{figure}

With this new handle decomposition we construct a surface $F''$ isotopic to $F$. 
We may take $\e>0$ small enough to make sure that there are no tangencies in the interval $[t_n - n \e, t_n - \e]$ in the rest of the surface $F \setminus F_{[t_0,t_n -\e]}$.

Let $F''$ be a surface starting with a local minimum, $p'' \in S_{t_n - n\e}$, and adding Type II $1$-handles to an upward disc arising from $p''$ prescribed by
\begin{itemize}
\item
the describing arc $\delta'_1$ copied in $S_{t_n - n\e}$,  
\item[] $\qquad \qquad\qquad \vdots$
\item the describing arc $\delta'_{n-2}$ copied in $S_{t_n - 3\e}$,  
\item
the describing arc $\delta'_{n-1}$ copied in $S_{t_n - 2\e}$,  
\end{itemize}
then taking the union with $F\setminus F_{[t_0,t_n-\e]}$. 
By (\ref{inclusion}) the both surfaces $F$ and $F''$ are isotopic to $(F\setminus F_{[t_0, t_n-\e]}) \cup X_{t_n-\e}$.

The observation (\ref{observation-handle}) guarantees that reordering the Type III saddle $q_n$ and the Type II saddles $q'_1,\dots,q'_{n-1}$ yields a surface $F'$ that is isotopic to $F''$. 
In particular, the Type II local minimum $p''$ of $F''$ becomes a Type III local minimum in the surface $F'$. 
Hence we have $m(F')=m(F)$, $h(F')=h(F)$ and $m_{\rm II}(F') <  m_{\rm II}(F)$.

For all the cases (Cases I, \dots, IV) there exists a surface $F'$ with 
$$(m(F'), h(F'), m_{\rm II}(F'), w(F')) < (m(F), h(F), m_{\rm II}(F), w(F))$$ if $m(F)\geq 1$. 
Eventually we can find a surface isotopic to $F$ and with no local minima.

Finally, we note that in the above argument we did not create any new intersections (elliptic points) of $F$ and the binding. This shows that if the original surface $F$ does not intersect a binding component, $C$, then the surface $F'$ admitting an essential open book foliation does not intersect $C$ either.
\end{proof}


\section{Fractional Dehn twist coefficients}
\label{sec:FDTC}

In this section we review the notion of right-veering diffeomorphisms and the fractional Dehn twist coefficient (FDTC) defined by Honda, Kazez and Mati\'c \cite{hkm1} and study its basic properties.
Then we show that the FDTC is effectively computable and give an alternative description that does not require the Nielsen-Thurston classification. 

\subsection{Definitions of $c(\phi, C)$ and $c(\phi, L, C)$} 

\begin{definition}\cite{hkm1}
Let $C$ be a boundary component of $S$, and let $\gamma, \gamma'$ be isotopy classes (rel. to the endpoints) of oriented properly embedded arcs in $S$ which start at the same base point $* \in C \subset \partial S$. 
We say that $\gamma'$ lies {\em strictly on the right side} of $\gamma$ if there exist curves representing $\gamma$ and $\gamma'$ realizing the minimal geometric intersection number, and $\gamma'$ strictly lies on the right side of $\gamma$ near $*$. In such case, we denote $\gamma > \gamma'$. 
If $\gamma>\gamma'$ or $\gamma=\gamma'$ we denote $\gamma\geq\gamma'$. 
 
\end{definition}

\begin{definition}\cite[Definition 2.1]{hkm1}
Let $C$ be a boundary component of $S$.
Let $\Aut(S,C)$ denote the group of isotopy classes of diffeomorphisms of $S$ fixing $C$ point-wise.
We say that $\phi \in \Aut(S,C)$ is {\it right-veering} (resp. {\em strictly right-veering}) with respect to $C$ if $\gamma \geq \phi(\gamma)$ (resp. $\gamma> \phi(\gamma)$) for any isotopy classes $\gamma$ of essential arcs in $S$ starting at $C$. 
For $\phi \in \Aut(S,\partial S)$,
we say that $\phi$ is {\em $($strictly$)$ right-veering} if $\phi$ is (strictly) right-veering with respect to all the boundary components of $S$. 
In particular, the identity map is right-veering. 
\end{definition}

\begin{convention}
Assume that $\chi(S) < 0$, i.e., $S$ admits a complete hyperbolic metric with finite area and geodesic boundary.
By the Nielsen-Thurston classification \cite{fm}, any $\phi \in \Aut(S, \partial S)$ is freely isotopic to a homeomorphism of $S$ of type either periodic, reducible or pseudo-Anosov.  
For each case, we say that $\phi \in \Aut(S, \partial S)$ is {\em periodic, reducible, pseudo-Anosov}, respectively.
If $S$ is an annulus or a disc, every element $\phi \in \Aut(S,\partial S)$ is regarded as periodic. 
\end{convention}

We review the definition the {\em fractional Dehn twist coefficient} due to Honda, Kazez and Mati\'c \cite{hkm1}. The origin of the FDTC is Gabai and Oertel's {\em degeneracy slope}, which is used to study taut foliations and essential laminations \cite{go}.

\begin{definition}
\label{def of c(phi,C)}

(1) 
Assume that $\phi \in \Aut(S, \partial S)$ is periodic. 
Let $C_{1},\ldots,C_d$ be the boundary components of $S$. 
Let $T_C$ denote the right-handed Dehn twist along (a curve parallel to) $C$. 
There exist numbers $N \in \mathbb{N}$ and $M_1, \ldots, M_d \in \Z$ such that $\phi^{N}=T_{C_{1}}^{M_{1}}\cdots T_{C_{n}}^{M_d}$, hence $\phi^N$ is freely isotopic to the identity map. 
We define the fractional Dehn twist coefficient $c(\phi,C_i)=\frac{M_i}{N}$. 

(2) 
Assume that $\phi\in \Aut(S, \partial S)$ is freely isotopic to a pseudo-Anosov homeomorphism, $\Phi$.
Fix a boundary component $C$. 
We review the description in \cite[p.432]{hkm1}. 
Let $L$ be the stable (or unstable) geodesic measured lamination for $\Phi$ and $W$ be the connected component of $S\setminus L$ containing $C$. It is known that $W$ is homeomorphic to $(S^{1} \times [0,1]) \setminus A$ where $C$ is identified with $S^1 \times \{0\}$ and $A =\{p_0,\ldots,p_{m-1}\}$ a finite set of points in $S^{1} \times \{1\}$. 
Take $m$ semi-infinite geodesics $\lambda_{i} \subset W$ which start at a point $q_i \in C$ and approach $p_{i}$ (see Figure~\ref{fig:fracDehn} left). 
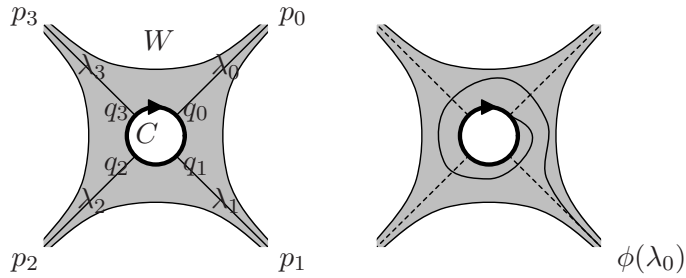
\begin{figure}[htbp]
 \begin{center}
\SetLabels
(0.2*.5)  $C$\\
(0.22*.84)  $W$\\
(0.44*.95)  $p_0$\\ 
(0.44*.05)  $p_1$\\
(0*.05)  $p_2$\\
(0*.95)  $p_3$\\
(.33*.75) $\lambda_{0}$\\
(.33*.26) $\lambda_1$\\
(.11*.75) $\lambda_3$\\
(.11*.26) $\lambda_2$\\
(.28*.6) $q_0$\\
(.28*.4) $q_1$\\
(.15*.6) $q_3$\\
(.15*.4) $q_2$\\
(1.03*.05)  $\phi(\lambda_{0})$\\
  \endSetLabels
\strut\AffixLabels{\includegraphics*[scale=0.5, width=80mm]{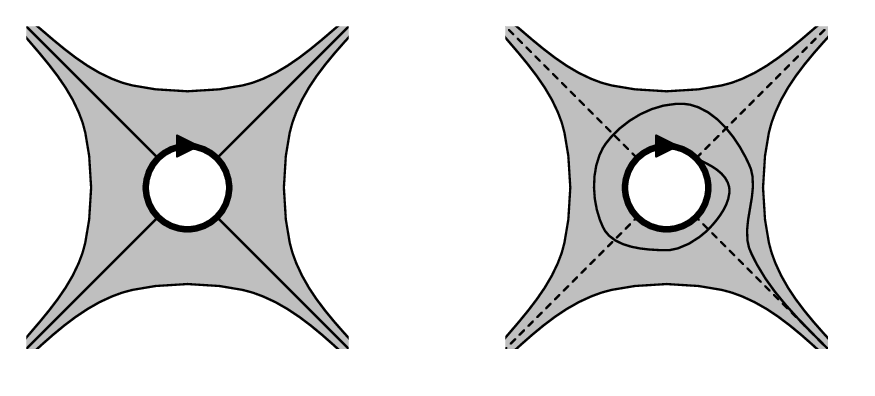}}
 \caption{Pseudo-Anosov case. $c(\phi, C)=5/4$.}
 \label{fig:fracDehn}
  \end{center}
\end{figure}
There exists $j \in \{0, \ldots, m-1\}$ such that  $\Phi(\lambda_i)=\lambda_{i+j}$.
Let $H: C \times [0,1] \to C$ be the free isotopy from $\phi$ to $\Phi$ restricted to $C$.  
We have $H(q_i, 0)=\phi(q_i)=q_i$ and $H(q_i, 1)=\Phi(q_i)=q_{i+j}$.
There exists $n\in \Z$ such that the arc $\alpha_t=H(q_0, t)$ starts from $q_0$ and winds around $C$ in the positive direction for $n+\frac{j}{m}$ times and then stops at $q_{j}$. 
We define $c(\phi,C) = n+\frac{j}{m}$. 
For example, $\phi$ depicted in Figure~\ref{fig:fracDehn} has $c(\phi, C)= 1+ \frac{1}{4} = \frac{5}{4}$.

(3) Suppose that $\phi\in\Aut(S, \partial S)$ is reducible. For each boundary component $C$ of $S$ there exists subsurface $S'$ of $S$ containing $C$ such that $\phi(S')=S'$ and $\phi|_{S'}$ is either periodic or pseudo-Anosov.
We define $c(\phi,C)=c(\phi|_{S'},C)$.
\end{definition}

Next we introduce a new notion, the FDTC for {\em closed braids}.

\begin{definition}\label{def of FDTC braid}
Let $L$ be a closed $n$-braid in an open book $(S,\phi)$ and $C$ be a boundary component of $S$.
We move $L$ by a braid-isotopy so that the $n$ points $L \cap S_0$ in the page $S_0$ stay in a collar neighborhood $\nu(C) \subset S$ of $C$. 
$$\{x_1,\dots,x_n\}: = L \cap S_0 \subset \nu(C)$$
Cutting the complement of the binding along the page $S_0$ we get a product region $S \times [0,1]$ and $L$ gives rises to an element, $\beta_L$, of the $n$-stranded surface braid group $B_n(S)$ of $S$. 
Let $i(\beta_L) \in \MCG(S, \{x_,\dots,x_n\})$ denote the image of $\beta_L$ under the {\em generalized  Birman exact sequence} \cite[Theorem 9.1]{fm}
$$
1 \to B_n(S) \stackrel{i}{\to} \MCG(S, \{x_,\dots,x_n\}) \to \MCG(S) \to 1
$$
where $\MCG(S, \{x_,\dots,x_n\})$ is the mapping class group of the surface $S$ with the $n$-marked points $x_1,\dots,x_n \in \nu(C)$.

We may assume that $\phi=\id$ in the collar neighborhood $\nu(C)$ so that $\phi$ can be regarded as an element of $\MCG(S, \{x_,\dots,x_n\})$. 
We define the fractional Dehn twist coefficient $c(\phi, L, C)$ of $L$ by 
\begin{equation}\label{def of c(phi,L,C)}
c(\phi, L, C) := c( i(\beta_L) \circ \phi, C).
\end{equation}

\end{definition}

Note that the map $i(\beta_{L}) \circ \phi$ can be understood as a monodromy of the fibration $M_{(S,\phi)} \setminus (B \cup L) \rightarrow S^{1}$ induced from the open book fibration $M_{(S,\phi)} \setminus B \rightarrow S^{1}$, where $B$ is the binding of the open book.

\begin{remark}

Given a closed braid $L$, the element $\beta_L \in B_n(S)$ is uniquely determined up to ``twisted'' conjugation in the following sense (see also \cite{Sk}):
Suppose that closed braids $L_1$ and $L_2$ are braid-isotopic with the same $n$-marked points $L_1 \cap S_0 = L_2 \cap S_0 \subset \nu(C)$. 
Then there exists a braid $\gamma \in B_n(S)$ such that $$\beta_{L_1} = \gamma^{-1} \beta_{L_2} \gamma^\phi$$ (see \cite[p.184]{Sk} for the definition of $\gamma^\phi$). 
In $\MCG(S, \{x_,\dots,x_n\})$ this implies:
\begin{eqnarray*}
i(\beta_{L_1}) \circ \phi &=&
(i(\gamma))^{-1} \circ i(\beta_{L_2}) \circ (\phi \circ i(\gamma) \circ \phi^{-1}) \circ \phi \\
&=&
(i(\gamma))^{-1} \circ (i(\beta_{L_2}) \circ \phi) \circ i(\gamma).
\end{eqnarray*}
Therefore, $i(\beta_{L_1}) \circ \phi$ and $i(\beta_{L_2}) \circ \phi$ are conjugate in $\MCG(S, \{x_,\dots,x_n\})$ and (\ref{def of c(phi,L,C)}) is well-defined. 
\end{remark}

\begin{example}
The difference $c(\phi,L,C)-c(\phi,C)$ can be arbitrary large.
For example, the open book decomposition $(S, \phi)=(D^{2},id)$ of $S^{3}$ has $c(\phi,\partial S) = 0$. If $L$ is a closed $2$-braid representing the $(2,k)$-torus knot then $c(\phi,L,\partial S)=\frac{k}{2}$. 
Another example is: if $L$ is a closed $1$-braid in a general open book $(S, \phi)$ and runs around a binding component $C$ for $k$ times in the direction of $C$ then  $c(\phi, L, C)=k$. 
\end{example}

\begin{example}\label{ex:unknot}
If $L$ is a meridian (i.e., $1$-braid) of a binding component $C$ then $c(\phi, L, C)=0$.  Note in general $c(\phi, L, C') \neq 0$ for other binding components $C'$.  
For example, consider an annulus open book with $k$-th power of the positive Dehn twist, $(S, \phi)=(A, D^k)$. Let $C, C'$ be the boundary circles of $A$ and $L$ be a meridian of $C$. Then $c(\phi, L, C) = 0$ and $c(\phi, L, C') = k$. 
This is because after isotopy $L \cap S_0 \subset \nu(C')$ and $L$ runs around $C'$ for $k$ times in the direction of $C'$. 
\end{example}

\begin{remark}\label{remark:cable}
From Definition \ref{def of FDTC braid} (cf. Definition \ref{def of c(phi,C)}-(3), the reducible case) it is clear that if $L'$ is a cable of a braid $L$ then $c(\phi, L, C)=c(\phi, L', C)$, except for the case when the surface $S$ is a disc, where the map $i:B_{n}(S) \rightarrow \MCG(S, \{x_1,\dots,x_n\})$ is an isomorphism.
\end{remark}

\subsection{Properties of $c(\phi, C)$ }

We develop computational techniques for $c(\phi, C)$, including the {\em key lemma} (Lemma~\ref{lemma:fracDehn}). 
In later sections we will obtain more estimates by using open book foliations.  
We start by listing basic properties of $c(\phi,C)$ which follow easily from the definition. 

\begin{proposition}\label{prop:property_of_c}

Let $C$ be a boundary component of $S$ and $\phi \in \Aut(S,\partial S)$.
\begin{enumerate}
\item $c(\phi^{N},C)=N c(\phi,C)$ for $N \in \Z$.
\item $c(T_{C},C)=1$ and $c(T_{C} \circ \phi,C) =c(\phi \circ T_{C}, C) = 1 + c(\phi, C)$.
\item $c(\phi, C) = c(\psi \circ \phi \circ \psi^{-1}, C)$ for any $\psi \in \Aut(S,\partial S)$. 

\end{enumerate}
\end{proposition}

\begin{proposition}[Honda, Kazez and Mati\'c \cite{hkm1}]
If $\phi$ is periodic then $\phi$ is right-veering if and only if $c(\phi,C) \geq 0$ for all $C$.
If $\phi$ is pseudo-Anosov then $\phi$ is right-veering with respect to $C$ if and only if $c(\phi,C)>0$.
\end{proposition}

The following shows that the topology of $S$ governs $c(\phi, C)$. 
\begin{proposition}
\label{prop:image_c}
Let $S=S_{g, d}$ be an oriented genus $g$ surface with $d>0$ boundary components. 
Let $C$ be a boundary component of $S$.
\begin{enumerate}
\item 
If $\phi$ is periodic then $c(\phi,C) \in \left\{  \left. \frac{p}{q} \: \right|   \: p \in \Z, q \in \{1,2,\ldots,4g+2\} \right\}$.
\item
If $\phi$ is pseudo-Anosov then $
c(\phi,C) \in \left\{ \left.\frac{p}{q} \: \right| \: p \in \Z, q \in \{1,2,\ldots,4g+d-3 \} \right\}$. 
\end{enumerate}
\end{proposition}

\begin{proof}
If $S$ is an annulus (or a disc), then $c(\phi, C) \in \Z$ (or $c(\phi, C)=0$) so the statement holds. 
Hence in the rest of the proof we assume $\chi(S)<0$. 
By fixing a hyperbolic metric on $S$, we regard $S$ as a complete hyperbolic surface with geodesic boundary and finite area.

Assume that $\phi$ is periodic with periodicity $M$. 
Let $\widehat S$ be the genus $g$ surface obtained by capping off the $d$ boundary circles.
Extend $\phi$ to $\widehat \phi \in \Aut(\widehat S)$ by setting $\widehat \phi = id$ on $\widehat S - S$. 
Clearly $\widehat\phi$ has period $M$. 
The ``$4g+2$ theorem'' \cite[Theorem 7.5]{fm} implies that $M \leq 4g+2$. 
Since $M c(\phi, C) \in \Z$ we get (1).

Next assume that $\phi$ is pseudo-Anosov. 
Let $L, W$ and $m$ be as in Definition~\ref{def of c(phi,C)}-(2). 
By Definition~\ref{def of c(phi,C)} we have:
\[ c(\phi,C) \in \left\{ \left. \frac{p}{m} \: \right| \: p \in \Z \right\}. \]
By the Gauss-Bonnet theorem $\textrm{Area}(W)=m\pi$ and $\textrm{Area}(S-W)\geq (d-1)\pi$, and we have: 
\[ m \pi + (d-1)\pi \leq \textrm{Area}(W) + \textrm{Area}(S- W) = \textrm{Area}(S) = (4g-4+2d) \pi, \]
i.e., $m \leq 4g-3+d$, thus we obtain (2).
\end{proof}

The following is a key estimate of $c(\phi, C)$ that will be used repeatedly in this paper. One can also find a similar result in \cite[Corollary 2.6]{kr}.

\begin{lemma}[Key lemma]
\label{lemma:fracDehn}
Let $C$ be a boundary component of $S$ and $\phi \in \Aut(S,\partial S)$. 
If there exists an essential arc $\gamma \subset S$ that starts on $C$ and satisfies 
$T_{C}^{m}(\gamma) \geq \phi(\gamma) \geq T_{C}^{M}(\gamma)$ 
for some $m, M \in \Z$ then $m \leq c(\phi,C) \leq M$. 
\end{lemma}

\begin{proof}
Assume contrary that $M < c(\phi,C)$ 
then $c(T_{C}^{-M}\phi, C) > 0$.
Propositions 3.1 and 3.2 of \cite{hkm1} imply that $T_C^{-M}\phi$ is strictly right-veering with respect to $C$.
Hence for any immersed geodesic arc $\alpha$ which starts from $C$ we have $\alpha > T_C^{-M}\phi(\alpha)$, hence $T_C^M(\alpha) > \phi(\alpha)$. 
This contradicts the assumption. 
The proof of $m\leq c(\phi,C)$ is similar.
\end{proof}

Practically, in order to compute $c(\phi, C)$ one may need to know the Nielsen-Thurston normal form for $\phi$ and its invariant measured lamination. 
Proposition~\ref{prop:image_c} and Lemma~ \ref{lemma:fracDehn} provide effective methods to compute $c(\phi, C)$ without using Nielsen-Thurston theory. 
Recall that in \cite{mo} Mosher proves that the mapping class group of $S$ is automatic and hence each element of $\Aut(S,\partial S)$ admits a normal form called {\em Mosher's normal form}.

\begin{theorem}
\label{theorem:computation}
Let $S=S_{g, d}$ and $D(S) = \max\{4g+2,4g+d-3\}$. Fix an integer $N > D(S)(D(S)-1)$.
Suppose that there exists a geodesic arc $\gamma\subset S$ that starts on $C$ and an integer $M$ satisfying
\[ T_{C}^{M}(\gamma) \geq \phi^{N}(\gamma) > T_{C}^{M+1}(\gamma). \]
Then the fractional Dehn twist coefficient has
\[ c(\phi,C) = \left[\frac{M}{N},\frac{M+1}{N} \right] \cap \left\{ \left. \frac{p}{q} \: \right|\: p \in \Z, q \in \{ 1 , 2 , \ldots, D(S) \} \, \right\}. \]
Moreover, $c(\phi,C)$ can be computed in polynomial time with respect to the length of Mosher's normal form of $\phi$.
\end{theorem}
 
\begin{proof}
By Lemma \ref{lemma:fracDehn} and Proposition \ref{prop:property_of_c}, $M \leq c(\phi^{N},C) = N c(\phi,C) \leq M+1$ so 
$c(\phi,C) \in [\frac{M}{N},\frac{M+1}{N}]$. 
On the other hand, by Proposition~\ref{prop:image_c}, $c(\phi,C) \in \{\frac{p}{q} \: | \: p \in \Z, q = 1 , 2 , \ldots,  D(S) \}$. Since we choose $N$ with $N > D(S)(D(S)-1)$, the intersection 
\[ \left[\frac{M}{N},\frac{M+1}{N}\right] \cap \left\{\frac{p}{q} \: | \: p \in \Z, q \in \{ 1 , 2 , \ldots, D(S) \} \, \right\} \]
consists of one rational number, which must be $c(\phi,C)$.

Next we show that $c(\phi,C)$ is computable in polynomial time. 
We define a partial ordering $<_{\gamma}$ on $\Aut(S,\partial S)$ by $\phi \leq_{\gamma} \psi$ if 
$\phi(\gamma) \geq \psi (\gamma)$. 
As shown in \cite[Theorem~2.1]{rw}, this partial ordering is determined in linear time with respect to the length $l(\phi)$ of Mosher's automatic normal form of $\phi$ (see \cite{mo} for the definition) by using Mosher's automatic structure of $\Aut(S,\partial S)$. 
By definition of Mosher's normal form, each generator $x$ of Mosher's normal form satisfies $c(\phi,C)=0$ so we have an a priori estimate $|c(\phi,C)| \leq  l(\phi)$. This implies that the above integer $M$ can be computable in polynomial time with respect to $l(\phi)$, hence so is $c(\phi,C)$. 
\end{proof}

\begin{corollary} 
If there exists a (possibly immersed) geodesic arc $\gamma \subset S$ that starts on $C \subset \partial S$ with 
$T_{C}^{m}(\gamma) = \phi^{N}(\gamma)$ for some $m,N \in \Z$ $(N \neq 0)$, then $c(\phi,C) = \frac{m}{N}$. 
\end{corollary}

\subsection{Alternative description of $c(\phi, C)$} 

We give an alternative description of the fractional Dehn twist coefficient which appears to be natural from a theoretical point of view and does not require the Nielsen-Thurston classification.
Let $\pi: \widetilde{S} \rightarrow S$ be the universal cover of $S$.
Fix a base point $* \in C \subset \partial S$ and its lift $\widetilde{*} \in \pi^{-1}(C) \subset \pi^{-1}(\partial S)$. Let $\widetilde{C}$ be the connected component of $\pi^{-1}(C)$ that contains $\widetilde{*}$. Since $S$ admits a hyperbolic metric there is an isometric embedding of $\tilde S$ to the Poincar\'e disc  $\mathbb{H}^{2}$. By attaching points at infinity to  $\widetilde{S}$ we obtain a compact disc $\overline{S} \subset \overline{\mathbb{H}^2}$.

For a homeomorphism $f: S \to S$ fixing the boundary pointwise we take the lift $\widetilde{f}: \widetilde{S} \rightarrow \widetilde{S}$ with $\widetilde{f}( \widetilde{*} ) =\widetilde{*}$.  
It uniquely extends to a homeomorphism $\overline{f}: \overline{S} \rightarrow \overline{S}$. 
The restriction $\overline{f}|_{\partial \overline S}$ is an invariant of the mapping class $[f] \in \MCG(S)$. 
Since $f = id$ on $\partial S$ and $\widetilde f(\widetilde*)=\widetilde*$ the map $\overline{f}$ fixes $\widetilde{C}$ pointwise. 
On the boundary at infinity $\partial \overline{S} \setminus \widetilde{C}$, that is homeomorphic to $\R$, the map $\overline f$ induces a homeomorphism of $\R$.

Consider the case: $f = T_C$ the positive Dehn twist along $C \subset \partial S$. 
The map $\overline{T_C}$ has no fixed points in the interior of the interval $\partial \overline{S} \setminus \widetilde{C}$. 
For if $\overline {T_C}(p_0)=p_0$ for some $p_0 \in \Int(\partial \overline{S} \setminus \widetilde{C})$ then $\overline {T_C}$ fixes (set-wise) the geodesic $\widetilde{\gamma} \subset \overline {\mathbb H^2}$ through $p_0$ and $\widetilde{*}$.
This means the Dehn twist $T_{C}$ fixes the geodesic ray $\pi(\widetilde{\gamma}) \subset S$ starting at $*$, which cannot happen.
Therefore, we can find a homeomorphism 
$$\Phi: \partial \overline{S} \setminus \widetilde{C} \to \R$$
such that 
$\Phi(\overline {T_C}(p)) = \Phi(p) + 1$.

Let $\widetilde{\textrm{Homeo}}^{+} (S^{1})$ 
be the group of orientation-preserving homeomorphisms of $\R$ that are lifts of orientation-preserving homeomorphisms of $S^{1}$. 
In other words, $\widetilde{\textrm{Homeo}}^{+} (S^{1})$ consists of elements of $\textrm{Homeo}^{+} (\R)$ that commute with the translation $x \mapsto x +1$.

Since $T_{C}$ is a central element in $\MCG(S)$ we can define a homomorphism 
$$\Theta_{C} : \MCG(S) \rightarrow \widetilde{\textrm{Homeo}}^{+} (S^{1}) \quad \text{by} \quad\Theta_{C}([f]) = \overline{f}|_{\partial \overline{S} \setminus \widetilde{C} }. $$
The map $\Theta_{C}$ is called the {\em Nielsen-Thurston homomorphism} and is intensively studied in \cite{sw} to describe a total left-invariant ordering of $\MCG(S)$. 
It is known that $\Theta_C$ is injective \cite{sw}. We note that $\Theta_{C}$ depends on various choices such as hyperbolic metrics on $S$ and identifications of $\partial \overline{S} \setminus \widetilde{C} $ with $\R$.

Let $\tau: \widetilde{\textrm{Homeo}}^{+} (S^{1}) \rightarrow \R$ be the {\em translation number}
defined by 
\[ \tau( h ) = \lim_{N \to \infty} \frac{h^{N}(x)-x}{N} \;\;\;(x \in \R).\]
It is well-known that the above limit exists and is independent of the choice of $x \in \R$.
The fractional Dehn twist coefficient is related to the Nielsen-Thurston map as follows.

\begin{theorem}
\label{theorem:translation}
$($cf. \cite[p.3]{ch}$)$
For $\phi \in \Aut(S, \partial S)$ we have $c(\phi,C) = \tau(\Theta_{C}(\phi))$.
\end{theorem}
\begin{proof}
Let us take a geodesic $\widetilde{\gamma}$ in $\widetilde{S} \subset \mathbb{H}^{2}$ which joins $\widetilde{*}$ and $x \in \partial \overline{S} \setminus \widetilde{C} = \R$.  
Denote $\gamma = \pi(\widetilde{\gamma})$. 
For $N >0$ there exists an integer $M(N)$ such that
\[ T_{C}^{M(N)} (\gamma) \geq \phi^{N}(\gamma) \geq T_{C}^{M(N)+1}(\gamma). \]
This is equivalent to 
\[ \Theta_{C}(T_{C}^{M(N)})(x) \leq \Theta_{C}(\phi^{N})(x) \leq \Theta_{C}(T_{C}^{M(N)+1})(x). \]
Recall that $\Theta(T_{C})$ translates $x \mapsto x+1$, hence 
\[ x+M(N) \leq  \Theta_{C}(\phi^{N})(x) \leq x+ M(N)+ 1, \]
i.e., 
\[ \frac{M(N)}{N} \leq \frac{ \Theta_{C}(\phi^{N})(x) -x}{N} \leq \frac{M(N)+1}{N}. \]
By Theorem \ref{theorem:computation} as $N\to\infty$ both $\frac{M(N)}{N}$ and $\frac{M(N)+1}{N}$ converge to $c(\phi,C)$ and the middle term converges to $\tau(\Theta_{C}(\phi))$, so we obtain $c(\phi,C) = \tau(\Theta_{C}(\phi)).$
\end{proof}

Since the translation number $\tau: \widetilde{\textrm{Homeo}}^{+} (S^{1}) \rightarrow \R$ is a homogeneous quasi-morphism of defect $1$, we get the following.

\begin{corollary}
The fractional Dehn twist coefficient with respect to $C$ defines a homogeneous quasi-morphism 
\[ c(\cdot, C): \Aut(S,\partial S) \rightarrow \Q  \]
of defect $1$. That is,
\[ | c(\phi \psi, C) - c(\phi,C) - c(\psi,C)| \leq 1 \]
and
\[ c(\phi^{N},C) = N c(\phi,C) \]
hold for all $\phi,\psi \in \Aut(S,\partial S)$ and $N \in \Z$.
\end{corollary}


\section{Estimates of fractional Dehn twist coefficient from open book foliation}
\label{sec:estimateFDTC}
Section~\ref{sec5.1} is devoted to estimates of the fractional Dehn twist coefficient of a monodromy. In Section~\ref{sec:estimate_braid} we extend the results to the FDTC for braids. 

Throughout this section, $L$ is a closed braid (possibly empty) in $(S, \phi)$ and $F$ denotes an oriented, connected compact surface which is either a Seifert surface of $L$ or a surface in $M_{(S,\phi)}-L$, admitting an open book foliation $\F(F)$.

\subsection{Estimates of $c(\phi, C)$}\label{sec5.1}

To estimate the FDTC of a monodromy the notion of strong essentiality (Definition~\ref{def of essential}) plays an important role.
Lemma~\ref{lemma:estimate}, a special case of Theorem~\ref{theorem:estimate}, gives a simple but still useful estimate. The proof shows  how the FDTC and an open book foliation are related. 

\begin{lemma}
\label{lemma:estimate}
Let $v$ be an elliptic point of $\F(F)$ lying on a binding component $C \subset \partial S$. 
Assume that $v$ is strongly essential and there are no a-arcs starting from $v$.
Let $p$ (resp. $n$) be the number of positive (resp. negative) hyperbolic points that are joined with $v$ by a singular leaf. 
\begin{enumerate}
\item If $\sgn(v)= +1$ then $-n \leq c(\phi,C) \leq p.$
\item If $\sgn(v) = -1$ then $-p \leq c(\phi,C) \leq n.$
\end{enumerate} 
\end{lemma}

\begin{proof}
We prove the case $\sgn(v)=-1$. (Similar arguments hold for the positive case.) Note that when $\sgn(v)=-1$ any regular leaf that ends at $v$ is a b-arc, so the assumption that there are no a-arcs starting from $v$ is automatically satisfied.

Let $h_{1},\ldots,h_{n+p}$ be the hyperbolic points connected to $v$ by a singular leaf, and $S_{t_{i}}$ be the page that contains $h_{i}$. Without loss of generality we may assume $0<t_{1}< \cdots < t_{n+p} < 1$.
For $t \neq t_1, \ldots, t_{n+p}$ let $b_{t}$ denote the $b$-arc in $S_{t}$ that ends at $v$. 
Since $v$ is strongly essential, the b-arc $b_t$ is strongly essential.

Let $l = p+n$. By induction on $l$ we prove 
\begin{equation}\label{induction}
T_C^{-n} (b_0) > b_{t_l + \e} > T_C^p (b_0).
\end{equation}

($l=1$ case)
Suppose that $\sgn(h_1)=+1$.
Let $\gamma \subset S_{0}\setminus (S_{0} \cap F)$ be a describing arc for $h_{1}$.
At least one of the endpoints of $\gamma$ lies on $b_0$, which we call $v'$ (if the both endpoints lie on $b_0$, pick the one closer to $v$). 
We isotope $\gamma$ in $S_0 \setminus (S_0 \cap F)$ by sliding $v'$ along $b_0$ until it reaches $v$. 
See the left sketch in Figure \ref{fig:estimate0}.
Since $\sgn(v)=-1$ a positive normal $\vec v_F$ of $b_0$ near $C$ is pointing the  opposite direction to the orientation of $C$. Thus Remark~\ref{remark sign describing arc} and $\sgn(h_1)=+1$ imply that $\gamma$ lies strictly on the right side of $b_0$ near $v$.  
Hence 
$$b_0 > \gamma.$$
Since the interiors of $\gamma$ and $b_0$ are disjoint and $b_0$ is strongly essential, $\gamma$ lies strictly on the left side of $T_{C}(b_0)$ near $v$, that is, 
$$\gamma > T_C(b_0).$$
After passing the critical time $t=t_1$ we may identify $b_{t_1+\e}$  with $\gamma$ for a sufficiently small $\e>0$ locally near $v$. 
Also any two of the three arcs $b_{0}$, $b_{t_1+\e}$ and $T_{C}(b_0)$ realize the minimal geometric intersection. 
Hence we get 
\begin{equation}\label{l=1}
b_0 > b_{t_1+\e} > T_C(b_0).
\end{equation}
Similarly, if $\sgn(h_{1})=-1$ we obtain 
$$T_{C}^{-1}(b_{0}) > b_{t_{1}+\e} > b_0.$$
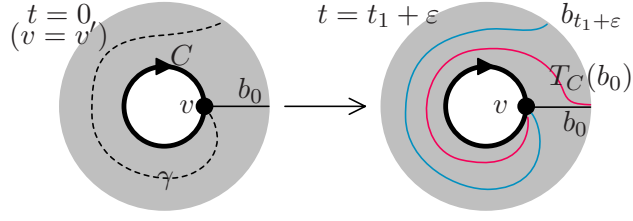
\begin{figure}[htbp]
 \begin{center}
 \SetLabels
(0*0.9) $t=0$\\ 
(0*.8) $(v = v')$\\
(0.6*0.9) $t=t_1 + \e$\\ 
(0.2*0.13)   $\gamma$\\
(0.36*0.54)  $b_{0}$\\
(.97*0.38)  $b_{0}$\\
(0.24*.47)   $v$\\
(0.83*.47)   $v$\\
(1*.9)   $b_{t_{1}+\e}$\\
(1*.6) $T_C(b_0)$\\
(.23*.7) $C$\\
\endSetLabels
\strut\AffixLabels{\includegraphics*[scale=0.5, width=70mm]{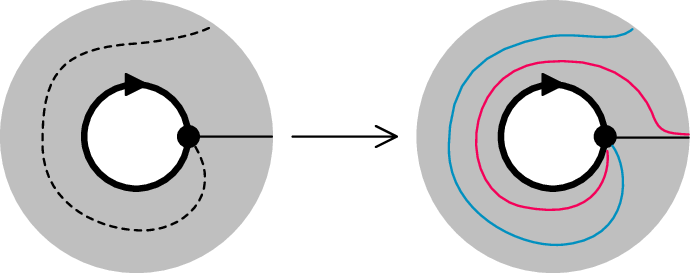}}
   \caption{Before and after the hyperbolic point $h_1$.}
 \label{fig:estimate0}
  \end{center}
\end{figure}

Assume that (\ref{induction}) holds when $l=k=p+n$. 
Let $l = k+1$. 
If $\sgn(h_{k+1})=+1$ by the above argument (\ref{l=1}) we have
$$ b_{t_{k+1} - \e} >  b_{t_{k+1} + \e} > T_C(b_{t_{k+1} - \e}).$$
Therefore, by the induction hypothesis (\ref{induction}) we have:
$$T_C^{-n} (b_0) > b_{t_k + \e} = b_{t_{k+1} - \e} >  b_{t_{k+1} + \e} > T_C(b_{t_{k+1} - \e})=T_C(b_{t_k + \e}) > T_C^{p+1} (b_0).$$
Similarly 
\[T_C^{-(n+1)} (b_0) > b_{t_{k+1} + \e} > T_C^p (b_0) \ \mbox{ if } \sgn(h_{k+1})=-1. \]
This concludes the statement (\ref{induction}).

Finally knowing that $b_{t_{p+k}+\e} = b_1$ near the vertex $v$
\[ T_{C}^{-n} (b_{0}) > b_{1} = \phi^{-1}( b_{0}) > T_{C}^{p}(b_{0}).\]
Since the elliptic point $v$ is strongly essential, the b-arc $b_0$ is strongly essential, i.e., $b_0$ is an essential arc in $S$. 
Lemma~\ref{lemma:fracDehn} implies $-n \leq c(\phi^{-1}, C)\leq p$. With Proposition~\ref{prop:property_of_c}-(1) we obtain the desired estimate.
\end{proof}

\begin{remark}
In \cite{i1}, \cite{i2} the first-named author used similar arguments to relate the valence of a vertex in the braid foliation (which corresponds to the number of hyperbolic singular points connected to an elliptic points by a singular leaf) and the Dehornoy floor, an integer-valued complexity of braids defined by the Dehornoy ordering of the braid groups which roughly corresponds to the absolute value of the fractional Dehn twist coefficient.
Lemma \ref{lemma:estimate} can be seen as a generalization of the arguments in \cite{i1} and \cite{i2}. 
\end{remark}

Below is an immediate consequence of Lemma~\ref{lemma:estimate}.

\begin{corollary}
\label{theorem:estimatesimple}
Let $v_1, \ldots, v_n$ be the strongly essential elliptic points on the same binding component $C$ of the open book $(S, \phi)$ such that all the regular leaves ending on $v_i$ are b-arcs. 
Let $p_i$ (resp. $n_i$) be the number of positive (resp. negative) hyperbolic points connected to $v_i$ by a singular leaf. 
Define the {\em upper bound} function 
$$U(v_i) = \left\{
\begin{array}{lll}
p_i & \mbox{ if } & \sgn(v_i)=+1, \\
n_i & \mbox{ if } & \sgn(v_i)=-1, \\
\end{array}
\right.$$
and the {\em lower bound} function 
$$L(v_i) = \left\{
\begin{array}{lll}
-n_i & \mbox{ if } & \sgn(v_i)=+1, \\
-p_i & \mbox{ if } & \sgn(v_i)=-1. \\
\end{array}
\right.$$
Then the fractional Dehn twist coefficient has
\begin{equation}\label{estimateA}
\max_{i=1,\ldots,n} L(v_i) \leq c(\phi, C) \leq \min_{i=1,\ldots,n} U(v_i).
\end{equation}
\end{corollary}

Lemma~\ref{lemma:estimate} and Corollary~\ref{theorem:estimatesimple} require only local information of an open book foliation, namely, the number of hyperbolic points connected to a {\em single} strongly essential elliptic point by a singular leaf.
In Theorem~\ref{theorem:estimate} below we examine {\em several} strongly essential elliptic points and obtain a sharper estimate of $c(\phi, C)$.

Let $\lceil x \rceil \in \Z$ be the {\em ceiling} of $x \in \R$, that is, the smallest integer grater than or equal to $x$.

\begin{theorem}
\label{theorem:estimate}
Let $v_{1},\ldots,v_{n} \in \F(F)$ be strongly essential elliptic points. 
Assume that all of $v_{1},\ldots,v_{n}$ lie on the same binding component $C$ and that all the regular leaves ending at $v_i$ are b-arcs. 
Let $N$ (resp. $P$) be the total number of negative (resp. positive) hyperbolic points that are connected to at least one of $v_{1},\ldots,v_{n}$ by a singular leaf.
Let $f_{\pm} :\mathbb{N} \rightarrow \Q$ be a map defined by
\[ f_{-}(m) = \left\{
\begin{array}{ll}
 \frac{1}{m} \left\lceil  \frac{Nm}{n} -\frac{(n-1)^{2}}{4 n^{2}} \right\rceil
& ( n: \textrm{odd}) \\
& \\ 
\frac{1}{m} \left\lceil \frac{Nm}{n} -\frac{n-2}{4n}  \right\rceil
& ( n: \textrm{even})
\end{array}
\right.
\]
and
\[ f_{+}(m) = \left\{ 
\begin{array}{ll}
\frac{1}{m} \left\lceil \frac{Pm}{n} - \frac{(n-1)^{2}}{4n^{2}} \right\rceil
& ( n: \textrm{odd}) \\
& \\ 
\frac{1}{m} \left\lceil \frac{Pm}{n} -\frac{n-2}{4n} \right\rceil
& ( n: \textrm{even}).
\end{array}
\right.
\]
\begin{enumerate}
\item If $\sgn(v_{1})=\sgn(v_{2}) = \cdots = \sgn(v_{n}) = -1$, then
\begin{equation}\label{estimateB}
-\inf_{m \in \mathbb{N}} f_{+}(m)\leq c(\phi,C) \leq \inf_{m \in \mathbb{N}} f_{-}(m). 
\end{equation}
\item If $\sgn(v_{1})=\sgn(v_{2}) = \cdots = \sgn(v_{n}) = +1,$ then
\begin{equation}\label{estimateC}
-\inf_{m \in \mathbb{N}} f_{-}(m) \leq c(\phi,C) \leq  \inf_{m \in \mathbb{N}} f_{+}(m).
\end{equation}
\end{enumerate}
\end{theorem}

\begin{remark}
Lemma \ref{lemma:estimate} is a special case ($n=1$ case) of Theorem \ref{theorem:estimate}.   
If $n \geq 2$ it depends on the embedding of $F$ which estimate among (\ref{estimateA}), (\ref{estimateB}) or (\ref{estimateC}) is the sharpest.
\end{remark}

\begin{proof}
We show the upper bound of $c(\phi,C)$ in (\ref{estimateB}).
The rest of the bounds can be obtained similarly.

We construct the $m$-fold cyclic branched covers of the ambient manifold $M_{(S, \phi)}$ and the surface $F$ branched at the binding:

Let $A = F \cap S_{0}$ be the multi-curve on $S=S_{0}$.
We cut $F \subset M$ along $A$ to get a properly embedded oriented surface $\Sigma$ in $\overline{M \setminus S_0} \simeq (S \times [0,1])/_{\sim_\partial}$ where ``$\sim_\partial$'' is an equivalence relation $(x, t) \sim_\partial (x, 0)$ for  $x \in \partial S$ and $t \in [0,1]$. 
We orient $A$ so that 
$$\Sigma \cap S_{0} = -A, \qquad \Sigma \cap S_{1} = \phi^{-1}(A).$$ 
Fix an integer $m\geq 1$. For $i=0,\ldots,m-1$, let 
\[ \Phi_i: (S \times [0,1])/_{\sim_\partial} \rightarrow (S\times[i,i+1])/_{\sim_\partial} \]
be a map defined by 
\[ \Phi_i(x,t) =(\phi^{-i}(x),t+i).\]
Let 
\[ \Sigma_{m} = \Sigma \cup \Phi_1(\Sigma) \cup \cdots \cup \Phi_{m-1}(\Sigma) \ \subset (S \times [0,m])/_{\sim_\partial}, \]
a properly embedded surface. 
Consider a natural quotient map $$\pi_m: (S \times [0,m])/_{\sim_\partial} \to M_{(S, \phi^m)}$$ identifying $(x, m)$ with $(\phi^m(x), 0)$ for $x \in \Int(S)$. 
Note that $$\Sigma_m \cap S_{0} = -A \ \mbox{ and }\  \Sigma_m \cap S_m = \phi^{-m}(A).$$
Hence we obtain a surface $F_{m}:=\pi_m(\Sigma_{m}) \subset M_{(S,\phi^{m})}$.

The manifold $M_{(S,\phi^{m})}$ is the cyclic $m$-fold branched cover of $M_{(S,\phi)}$ with branch locus the binding of the open book.
Likewise the surface $F_m$ is the cyclic $m$-fold branched cover of $F$.

Recall that $v_1,\dots,v_n$ are strongly essential negative elliptic points lying on the binding component $C$. 
Note that $v_1,\dots,v_n$ are subset of the branch points of $F$. 
By abuse of notation let $v_{1},\ldots, v_{n} \in F_m$ denote the lifts of $v_{1},\ldots,v_{n} \in F$. 
By the construction, the lifts $v_{1},\ldots, v_{n}$  are also strongly essential negative elliptic points in $\F(F_m)$ and connected to $Nm$ negative hyperbolic points, $h_{1},\ldots,h_{Nm}$, by a singular leaf. 
We assume that $h_{k}$ lies on the page $S_{t_{k}}$ where 
$
0<t_{1}<t_{2} < \cdots < t_{N} < 1
$
and
$
t_{lN + j} =l + t_j$ for $l=1,\dots,n-1$ and $j=1,\dots,N.
$
For $t \neq t_1,\dots,t_{Nm}$ we denote the b-arc in $S_{t}$ that ends at $v_i$ by $b^{i}_{t}$. 
As in the proof of Lemma~\ref{lemma:estimate}, we compute the upper bound of $c(\phi^{m},C)$ by comparing $b^{i}_{0}$ and $b^{i}_{m} = \phi^{-m}(b^{i}_{0})$. 
To this end we introduce the {\em twisting} of $b^{i}_{t}$, $tw(b^{i}_{t})$, and the {\em total twisting} on $S_{t}$ along $C$, $TW_t^C$:

Consider the projection 
$\mathcal P: (S \times [0,m])/\sim_\partial \ \to S$ defined by $\mathcal P(x, t) = x$.
Below, for the sake of simplicity we freely identify an arc in a page $S_t$ and its image under $\mathcal P$. 

Let $\mathbf{b}_{0}= b_0^1 \cup \cdots \cup b_0^n$ be a multi-curve in the page $S_0$. 
Similarly put  $\mathbf{b}_{t}= b_t^1 \cup \cdots \cup b_t^n$.
Isotope the b-arc $b_t^i$ and the multi-curve $\mathbf{b}_{0}$ so that they realize the minimal geometric intersection number, $\mathsf i_S (b_t^i, \ \mathbf{b}_{0})= \mathsf i_S (\mathcal P(b_t^i), \ \mathcal P(\mathbf{b}_{0}))$.

Let $\mathcal{C}: S \rightarrow S/C$ be the quotient map that collapses the boundary component $C$ to a point. 
Let $\mathsf i_{S/C} (\mathcal C(b_t^i), \ \mathcal C(\mathbf{b}_{0}))$ denote the minimum geometric intersection number of the images $\mathcal{C}(b_{t}^{i})$ and $\mathcal{C}(\mathbf{b}_{0})$. 
We define the {\em twisting} of $b^{i}_{t}$ as the difference between the two geometric intersection numbers: 
\[ 
tw(b^{i}_{t}) := \left\{ 
\begin{array}{ll} 
0 & \textrm{ if } b^{i}_{0} \geq   b^{i}_{t} \textrm{ near } v_i,\\
\mathsf i_{S}( b_t^i, \ \mathbf{b}_{0})
- 
\mathsf i_{S/C} ( \mathcal{C}(b_t^i) , \ \mathcal{C}(\mathbf{b}_{0}) ) +1
& \textrm{ if } b^{i}_{t} > b^{i}_{0} \textrm{ near } v_i.\\
\end{array} \right.
\]

Let us call the intersection points of $b_t^i$ and $\mathbf{b}_{0}$ that vanish after collapsing the boundary component  $C$ \emph{boundary intersection points near $C$}. The definition shows that $tw(b_t^{i})$ is equal to the number of boundary intersection points near $C$ plus one. The ``$+1$'' term corresponds to the point $v_i$, which is regarded as a boundary intersection point near $C$.

We may put $b_t^i$ and $\mathbf{b}_{0}$ by isotopy so that all the boundary intersection points near $C$ are contained in a small annular neighborhood $\nu(C)$ of $C$. See Figure \ref{fig:twistingamount}, where the gray regions represent $\nu(C)$ and the dashed lines represent  the multi-curve $\mathbf{b}_{0} \cap \nu(C)$ for $n=6$.
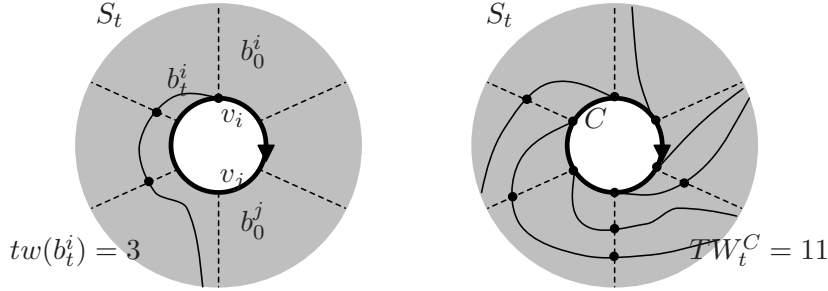
\begin{figure}[htbp]
 \begin{center}
 \SetLabels
(0.05*0.93) $S_{t}$ \\
(0.62*0.93) $S_{t}$ \\
(0.26*0.8)  $b^{i}_{0}$ \\
(0.26*0.2)  $b^{j}_{0}$ \\
(0.15*0.7) $b^{i}_{t}$ \\
(0.23*0.58) $v_{i}$\\
(.23*.37) $v_j$\\
(0.76*.56)  $C$\\
(0*0.1) $tw(b^{i}_{t})=3$\\
(1*0.1) $TW^C_t=11$\\
\endSetLabels
\strut\AffixLabels{\includegraphics*[scale=0.5, width=90mm]{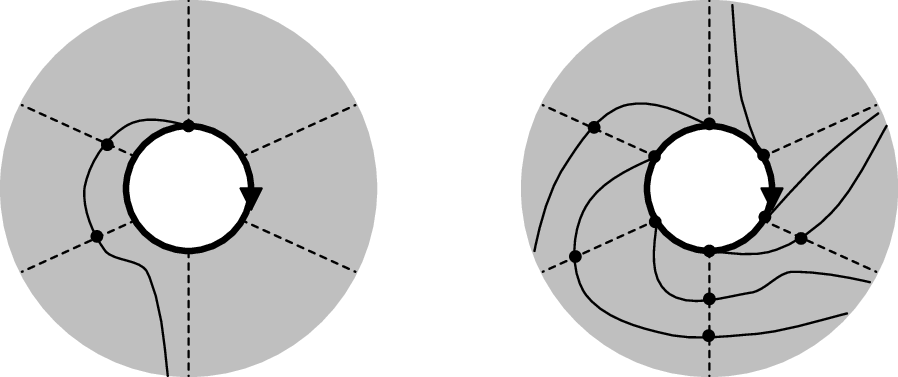}}
 \caption{The twisting $tw(b^{i}_{t})$ and the total twisting $TW^C_t$.}
 \label{fig:twistingamount}
  \end{center}
\end{figure}

Now we consider the case $t=m$. 
Let $\omega_i := \left\lceil \frac{tw(b^{i}_m)}{n} \right \rceil \geq 0$. 
From the definition of $tw(b^{i}_m)$, in $\nu(C)$ the b-arc $b^{i}_m$ winds $\omega_i$ times around $C$ counter-clockwise. Thus  
$$ 
T_{C}^{\omega_i} (b^{i}_m) \geq b_0^i = 
\phi^m(b^{i}_m).
$$

Therefore, Lemma \ref{lemma:fracDehn} implies that
\begin{equation}
\label{eqn:tw}
c(\phi^{m},C) \leq \omega_i = \left\lceil \frac{tw(b^{i}_{m})}{n} \right\rceil \quad \mbox{ for all } i=1,\ldots, n.
\end{equation}

The {\em total twisting} $TW^C_t$ on the fiber $S_t$ along the boundary $C$ is defined by:
\begin{equation}\label{def-of-TW}
TW^C_t := \sum_{i=1}^{n} tw(b^{i}_{t})
\end{equation}
See the right sketch in Figure~\ref{fig:twistingamount}.

As we have seen in the proof of Lemma~ \ref{lemma:estimate}, positive hyperbolic points do not contribute to the total twisting because they turn the curves to the right. So in the next two observations we study the effect of {\em negative} hyperbolic points $h_{1},\ldots,h_{mN}$:

\begin{observation}\label{obs2}
(1): 
For each $j =1,\ldots,mN$, the negative hyperbolic point $h_j$ increases the total twisting by at most $n$. That is, 
$$0 \leq TW^C_{t_{j}  + \e} - TW^C_{t_{j} -\e} \leq n.$$ 
(2):  For each $j = 1, \ldots, \left\lfloor \frac{n}{2} \right\rfloor$ the negative hyperbolic point $h_{j}$ increases the total twisting by at most $2j$.
That is, 
$$0 \leq TW^C_{t_{j}  + \e} - TW^C_{t_{j} -\e} \leq 2j.$$ 
\end{observation}

\begin{proof}
Since we are assuming $\sgn(h_j) = -1$ and $\sgn(v_i)=-1$ we have $$0 \leq TW^C_{t_{j}  + \e} - TW^C_{t_{j} -\e}.$$ 

(1): 
For each $j = 1, \ldots, mN$ we choose a describing arc, $\gamma_j$, of the negative hyperbolic point $h_j$ so that 
\begin{itemize}
\item
$\gamma_j$ and $\mathbf{b}_{0}$ realize the minimum geometric intersection number, and 
\item
all the boundary intersection points of $\gamma_j$ and $\mathbf{b}_{0}$ (the ones that disappear if we collapse $C$ to a point) of $\gamma_j$ and $\mathbf{b}_{0}$ (if they exist) lie in $\nu(C)$. 
\end{itemize}
Then we get an equality: 
\begin{equation}\label{eq:TW}
TW^C_{t_{j}  + \e} - TW^C_{t_{j} -\e} = | \gamma_j \cap \mathbf{b}_{0}\cap \nu(C)|.
\end{equation}

Note that the number of boundary intersection points is at most $n$. 
See Figure~\ref{fig:tot2}, where the points $\gamma_j \cap \mathbf{b}_{0}\cap \nu(C)$ are marked by black dots $\bullet$ and $|\gamma_j \cap \mathbf{b}_{0}\cap \nu(C)|=n=6$. 
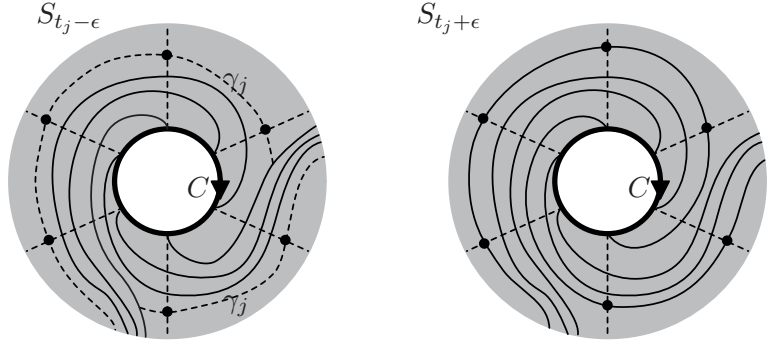
\begin{figure}[htbp]
\SetLabels
(.08*1) $S_{t_j-\epsilon}$\\
(.58*1) $S_{t_j+\epsilon}$\\ 
(.3*.8) $\gamma_j$\\
(.3*.1) $\gamma_j$\\
(.25*.45) $C$\\
(.83*.45) $C$\\
\endSetLabels
\strut\AffixLabels{\includegraphics*[scale=0.5, width=100mm]{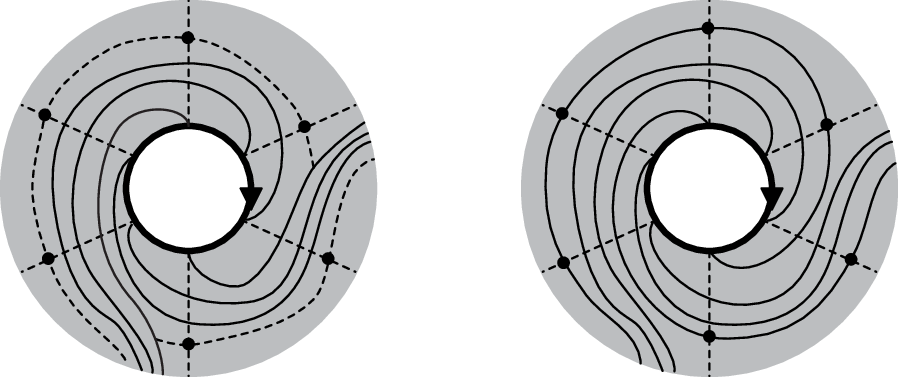}}
\caption{Boundary intersection points in $\nu(C)$, where $n=6$.}
\label{fig:tot2}
\end{figure}
Thus we have $| \gamma_j \cap \mathbf{b}_{0}\cap \nu(C)| \leq n$. This inequality and (\ref{eq:TW}) imply the assertion (1).

(2): 
We show the second assertion. 
Let $j=1$.
Since $\mathbf b_0 = \mathbf b_{t_1-\e}$ in $\nu(C)$ and $\Int(\gamma_1)$ never intersect $\mathbf b_{t_1-\e} = \mathbf b_0$, we have  $| \gamma_1 \cap \mathbf{b}_{0}\cap \nu(C)| \leq 2$. 
The equality ``$=2$'' holds when both of the endpoints of $\gamma_1$ are in $\nu(C)$ as in the left sketch of Figure~\ref{fig:tot1}.

By definition, the describing arc $\gamma_j$, viewed as a curve in $S_{t_j -\e}$, is a properly embedded arc in $S_{t_j -\e} \setminus (S_{t_j -\e} \cap F)$. 
That is, 
$\Int(\gamma_j)$ does not intersect the multi-curve $\mathbf b_{t_j-\e}= \mathbf b_{t_{j-1}+\e}$.
Thus we have 
\begin{equation}\label{eq:+2}
|\gamma_j \cap \mathbf b_0 \cap \nu(C) | \leq |\gamma_{j-1} \cap \mathbf b_0 \cap \nu(C)| + 2
\end{equation}
for $j = 1, \ldots, \left\lfloor \frac{n}{2} \right\rfloor$. 
The equality in (\ref{eq:+2}) is realized when 
\begin{itemize}
\item
$\gamma_j$ and $\gamma_{j-1}$ start from 
consecutive b-arcs in $\mathbf{b}_{0}$ and end at consecutive b-arcs in $\mathbf{b}_{0}$. 
See Figure~\ref{fig:j-1}. 
\item
Both $\gamma_j$ and $\gamma_{j-1}$ turn left in $\nu(C)$ at the endpoints. 

\item
$\mathcal P(\gamma_j)$ and $\mathcal P(\gamma_{j-1})$ are `parallel', $\mathcal P(\gamma_{j-1})$ is closer to the binding $C$ than $\mathcal P(\gamma_j)$,  
and they get out of $\nu(C)$ region from the same sectors (highlighted in light gray in Figure~\ref{fig:j-1}). 
\end{itemize}
By (\ref{eq:TW}) and (\ref{eq:+2}) we obtain $0 \leq TW^C_{t_{j}  + \e} - TW^C_{t_{j} -\e} \leq 2j.$
\begin{figure}[htbp]
\begin{center}
\SetLabels
(.2*1.03) $t=t_{j-1}-\e$\\
(.8*1.03) $t=t_j-\e$\\
(.1*.75) $\gamma_{j-1}$\\
(.43*.4) $\gamma_{j-1}$\\
(.75*.9) $\gamma_j$\\
(.85*.1) $\gamma_j$\\
\endSetLabels
\strut\AffixLabels{\includegraphics*[scale=0.5, width=90mm]{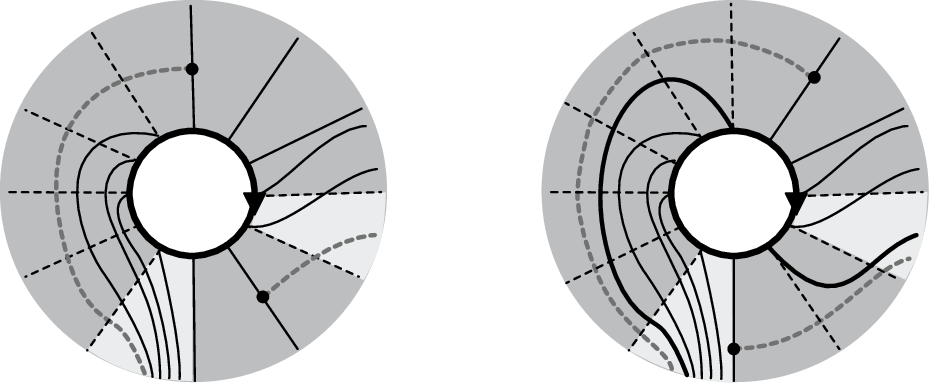}}
 \caption{}
 \label{fig:j-1}
  \end{center}
\end{figure}

In order to understand when the equality $TW^C_{t_{j}  + \e} - TW^C_{t_{j} -\e} = 2j$ holds (in other words the equality holds in (\ref{eq:+2}) for all $k = 1,\dots,j$),  
suppose that the describing arc $\gamma_1$ for $h_1$ joins $b^{i_\circ} \subset {\bf b}_0$ and $b^{i_\bullet}\subset {\bf b}_0$ for some $i_\circ, i_\bullet \in \{1, \ldots, n\}$. 
Then the arc $\gamma_k$ joins $b^{i_\circ+k-1}$ and $b^{i_\bullet+k-1}$ for all $k = 1,\dots,j$ $(\leq\left\lfloor \frac{n}{2} \right\rfloor)$ if and only if the the equality $TW^C_{t_{j}  + \e} - TW^C_{t_{j} -\e} = 2j$ holds. 
For example, 
the middle sketch of Figure~\ref{fig:tot1} shows 
$TW^C_{t_2  + \e} - TW^C_{t_2 -\e} = | \gamma_2 \cap \mathbf{b}_{0}\cap \nu(C)| = 4$. 
\end{proof}
\begin{figure}[htbp]
\begin{center}
\SetLabels
(.05*.4) $\gamma_1$\\
(.2*.45) $\gamma_1$\\ 
(.4*.55) $\gamma_2$\\
(.52*.17) $\gamma_2$\\
(0.15*1.05)  $t = t_1-\e$\\ 
(0.5*1.05)  $t= t_2-\e$\\
(0.85*1.05)  $t=t_2+\e$\\
\endSetLabels
\strut\AffixLabels{\includegraphics*[scale=0.5, width=120mm]{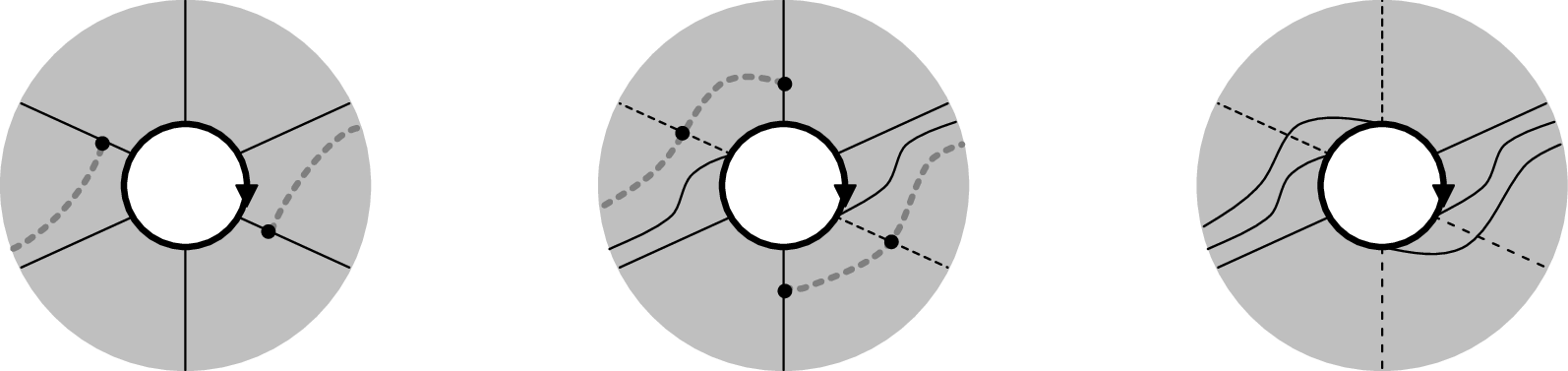}}
\caption{(Observation~\ref{obs2}) 
$TW_{t_1-\e}^C=0$, 
$TW_{t_1+\e}^C=TW_{t_2-\e}^C = 2$ and 
$TW_{t_2+\e}^C =2+4 =6$.
}
\label{fig:tot1}
\end{center}
\end{figure}

From Observation \ref{obs2}, letting $k=
\lfloor \frac{n}{2} \rfloor$ we have:
\begin{eqnarray*}
TW^C_{m} & \leq & 2+4+ \cdots + 2k + (Nm-k)n \\
& = & \left\{ 
\begin{array}{ll}
Nmn - \frac{(n-1)^{2}}{4} & n \textrm{: odd}\\
& \\
Nmn - \frac{n(n-2)}{4} & n \textrm{: even}
\end{array}\right.
\end{eqnarray*}
Thus by (\ref{def-of-TW}) we have:
\[ \min_{i=1,\ldots,n} \left\{ tw(b^{i}_{m}) \right\} \leq \left\{ \begin{array}{ll}
Nm - \frac{(n-1)^{2}}{4n} &n \textrm{: odd}\\
& \\
Nm - \frac{(n-2)}{4} &n \textrm{: even}
\end{array} \right.\]
By (\ref{eqn:tw}) we have for all $m\in\mathbb N$:
\[
m \cdot c(\phi,C)= c(\phi^{m},C) \leq 
\left\{ 
\begin{array}{cl} 
\left\lceil \frac{Nm}{n} - \frac{(1-n)^{2}}{4n^{2}} \right\rceil & n \textrm{: odd}\\
& \\
\left\lceil \frac{Nm}{n} - \frac{n-2}{4n}  \right\rceil  & n \textrm{: even}
\end{array}
\right.
\]
\end{proof}

\subsection{Estimates of $c(\phi, L, C)$} 
\label{sec:estimate_braid}

Now we apply the results in Section~\ref{sec5.1} to obtain estimates of the fractional Dehn twist coefficients for closed braids, $c(\phi, L, C)$. 

Let $L$ be an $n$-stranded closed braid with respect to the open book $(S,\phi)$.   
Let $F$ be either a Seifert surface of $L$ or a closed surface in the complement of $L$.
Suppose that $\F(F)$ is essential. 
Recall that $M_{(S,\phi)}\setminus(L \cup B)\to S^{1}$ is a fibration with the fiber $S'=S\setminus\{n \textrm{ points}\}$ and the monodromy $i(\beta_L) \circ \phi \in \MCG(S, \{x_,\dots,x_n\})$ as discussed in Definition~\ref{def of FDTC braid}. 
Observe that every a-arc of $\F(F)$ is an essential arc in $S'$ and every essential b-arc of $\F(F)$ is an essential arc in $S'$. 
Because of this, there is a difference between estimates of $c(\phi,C)$ and those of $c(\phi, L,C)$. The former requires b-arcs to be strongly essential, while the latter requires just essentiality.

The following are variations of Lemma~\ref{lemma:estimate} and Theorem~\ref{theorem:estimate} with weaker assumptions.

\begin{lemma}
\label{lemma:estimate-braid}
Let $v$ be an elliptic point of $\F(F)$ lying on a binding component $C \subset \partial S$. 
Assume that $v$ is essential.
Let $p$ (resp. $n$) be the number of positive (resp. negative) hyperbolic points that are joined with $v$ by a singular leaf. 
\begin{enumerate}
\item If $\sgn(v)= +1$ then $-n \leq c(\phi, L,C) \leq p.$
\item If $\sgn(v) = -1$ then $-p \leq c(\phi, L,C) \leq n.$
\end{enumerate} 
\end{lemma}

\begin{theorem}
\label{theorem:estimate-braid}
Let $v_{1},\ldots,v_{n} \in \F(F)$ be essential elliptic points lying on the same component $C$ of $\partial S$.
Let $N$ (resp. $P$) be the total number of negative (resp. positive) hyperbolic points that are connected to at least one of $v_{1},\ldots,v_{n}$ by a singular leaf.
Let $f_{\pm} :\mathbb{N} \rightarrow \Q$ be a map defined by
\[ f_{-}(m) = \left\{
\begin{array}{ll}
 \frac{1}{m} \left\lceil  \frac{Nm}{n} -\frac{(n-1)^{2}}{4 n^{2}} \right\rceil
& ( n: \textrm{odd}) \\
& \\ 
\frac{1}{m} \left\lceil \frac{Nm}{n} -\frac{n-2}{4n}  \right\rceil
& ( n: \textrm{even})
\end{array}
\right.
\]
and
\[ f_{+}(m) = \left\{ 
\begin{array}{ll}
\frac{1}{m} \left\lceil \frac{Pm}{n} - \frac{(n-1)^{2}}{4n^{2}} \right\rceil
& ( n: \textrm{odd}) \\
& \\ 
\frac{1}{m} \left\lceil \frac{Pm}{n} -\frac{n-2}{4n} \right\rceil
& ( n: \textrm{even}).
\end{array}
\right.
\]
\begin{enumerate}
\item If $\sgn(v_{1})=\sgn(v_{2}) = \cdots = \sgn(v_{n}) = -1$, then
\begin{equation}
- \inf_{m \in \mathbb{N}} f_{+}(m)\leq c(\phi, L,C) \leq \inf_{m \in \mathbb{N}} f_{-}(m). 
\end{equation}
\item If $\sgn(v_{1})=\sgn(v_{2}) = \cdots = \sgn(v_{n}) = +1$, then 
\begin{equation}
- \inf_{m \in \mathbb{N}} f_{-}(m) \leq c(\phi, L,C) \leq  \inf_{m \in \mathbb{N}} f_{+}(m).
\end{equation}
\end{enumerate}
\end{theorem}

\section{Non-right-veeringness and open book foliation}\label{sec6}

The notion of right-veeringness is closely related to tightness of the contact structure supported by the open book due to the following theorem of Honda-Kazez-Mati\'c \cite[Theorem 1.1]{hkm1}, (cf. \cite[Theorem 2.4]{ik1-2} for an alternative proof using open book foliations). 

\begin{theorem}
\label{theorem:ot}
If $\phi \in \Aut(S, \partial S)$ is not right-veering then $(S,\phi)$ supports an overtwisted contact structure.
\end{theorem}

In this section we study (non) right-veeringness using open book foliations. 
Among the results, Corollary~\ref{cor:n(S,phi)} highlights the fact that the converse of Theorem~\ref{theorem:ot} does not hold in general,
which has been  already observed and studied in \cite{hkm1}. Namely, right-veeringness does not imply tightness of the compatible contact structure.

Recall Definition~\ref{def:trans-ot-disc} of a transverse overtwisted disc. 

\begin{theorem}\label{theorem:nrv}
A diffeomorphism $\phi \in\Aut(S, \partial S)$ is non-right-veering if and only if there exists a transverse overtwisted disc $D$ in $(S,\phi)$ whose graph $G_{--}$ consists of a single negative elliptic point as depicted in Figure~\ref{fig:transot}.  
\begin{figure}[htbp]
 \begin{center}
 \SetLabels
\endSetLabels
\strut\AffixLabels{\includegraphics*[scale=0.5, width=50mm]{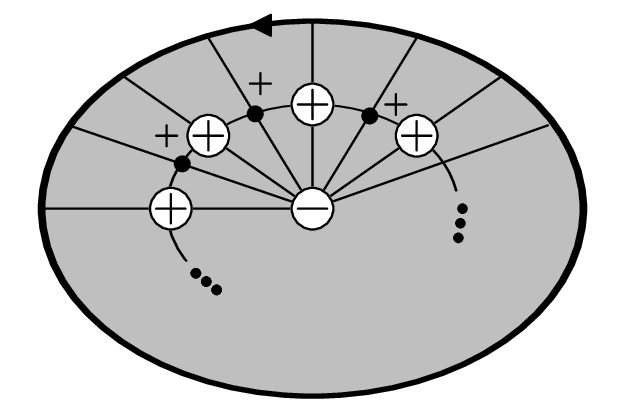}}
 \caption{A transverse overtwisted disc with one negative elliptic point.}
 \label{fig:transot}
  \end{center}
\end{figure}
\end{theorem}

\begin{proof}
($\Rightarrow$) 
If $\phi$ is non-right-veering then there exists an essential properly embedded arc $\gamma$ in $S$ such that $\phi(\gamma) > \gamma$. 
As shown in the proof of \cite[Theorem 2.4]{ik1-2} we can construct a transverse overtwisted disc with exactly one negative elliptic point.

($\Leftarrow$) 
Let $v =G_{--}$ be the negative elliptic point in $\F(D)$.
Let $k$ be the number of positive hyperbolic points in the graph $G_{++}$ of $\F(D)$ and $h_1,\dots,h_k$ be the positive hyperbolic points. 
Let $0< t_1 < \cdots < t_k < 1$ be the numbers such that the page $S_{t_i}$ contains $h_i$.
Let $b_t \subset S_t$ ($t\neq t_i$) be the b-arc emanating from $v$.

To show the monodormy is not right-veering we observe the following.

\begin{claim}\label{claim v}
$v$ is strongly essential.
\end{claim}

\begin{proof}
Let $C \subset \partial S$ be the binding component on which $v$ lies. 
Assume contrary that $v$ is not strongly essential. 
Namely, some b-arc from $v$, say $b_{0}$, is boundary parallel in the page $S_0$. 
Let $w \in \partial b_0$ be the positive elliptic point. 
The arc $\overline{wv} \subset C$ cobounds a disc $\Delta_0 \subset S_0$ with $b_0$. 
There are two cases to consider:
\begin{description}
\item[(i)] The disc $\Delta_0$ lies on the right side of $b_0$ as we walk from $v$ to $w$ (upper-left sketch in Figure~\ref{fig:case_i_ii}). 
\item[(ii)] The disc $\Delta_0$ lies on the left side of $b_0$ as we walk from $v$ to $w$ (lower-right sketch in Figure~\ref{fig:case_i_ii}).
\end{description}

\begin{figure}[htbp]
\begin{center}
\SetLabels
(-.1*.95) Case (i)\\
(.03*.88) $v$\\
(.03*.75) $w'$\\
(.03*.65) $w$\\
(.25*.7) $b_0$\\
(.13*.83) $\Delta_0$\\
(.03*.57) $C$\\
(.2*.78) $h_1$\\
(.6*.88) $v$\\
(.6*.75) $w'$\\
(.6*.65) $w$\\
(.75*.9) $b_{t_1+\e}$\\
(.7*.83) $\Delta_{t_1+\e}$\\
(.6*.57) $C$\\
(-.1*.42) Case (ii)\\
(.03*.37) $w$\\
(.03*.14) $v$\\
(.43*.14) $w''$\\
(.03*.05) $C$\\
(.25*.3) $h_k$\\
(.25*.17) $b_{t_k - \e}$\\
(.6*.37) $w$\\
(.6*.14) $v$\\
(.98*.14) $w''$\\
(.72*.3) $\Delta_0$\\
(.83*.36) $b_0 \approx b_{t_k + \e}$\\
(.6*.05) $C$\\
\endSetLabels
\strut\AffixLabels{\includegraphics*[width=100mm]{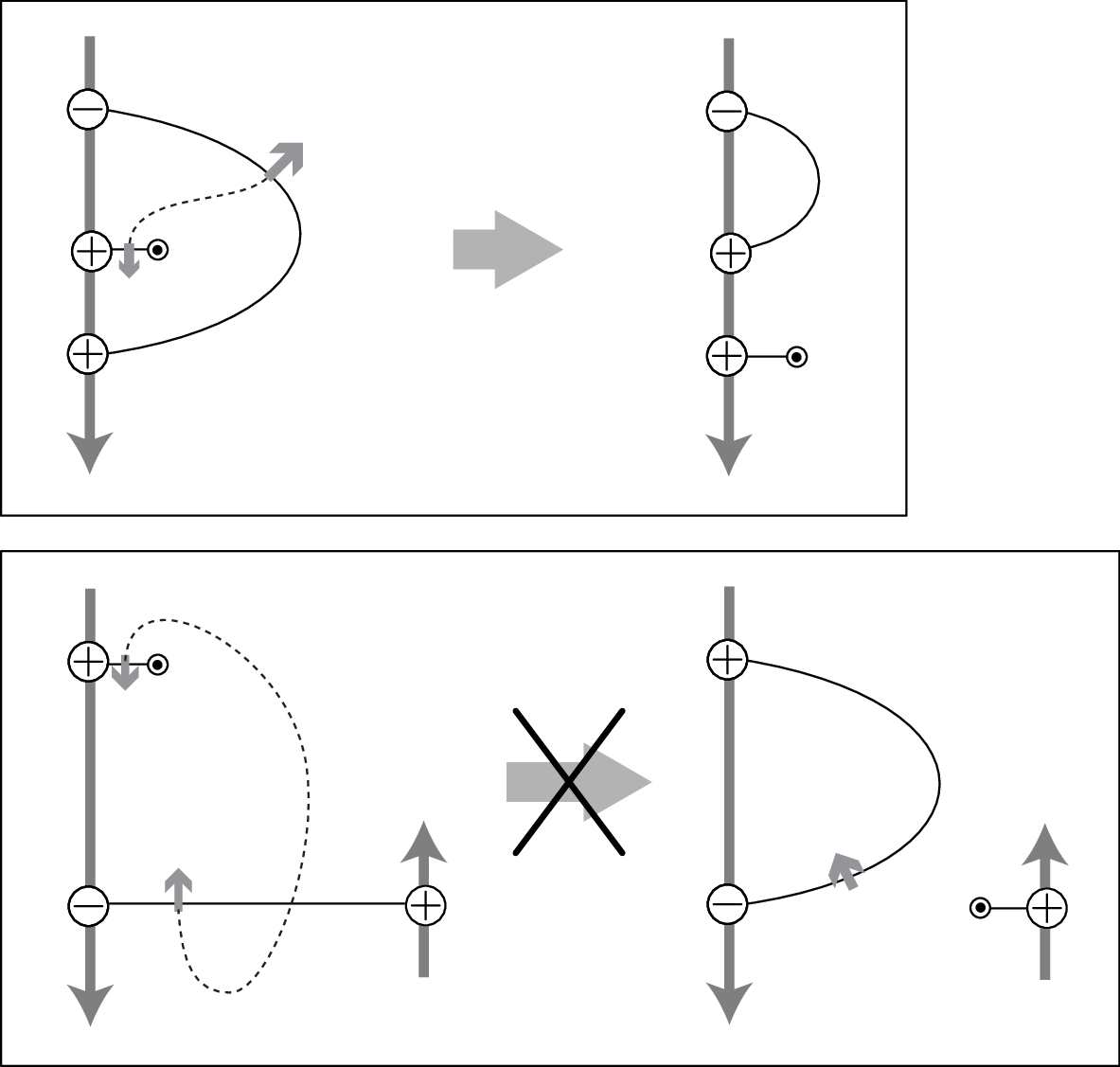}}
 \caption{Proof of Claim~\ref{claim v}.}
\label{fig:case_i_ii}
\end{center}
\end{figure}

Case (i): 
We may assume that the disc $\Delta_0$ is {\em minimal} in the sense that if a b-arc $b_{t}$ cobounds a disc $\Delta_{t} \subset S_t$ with $C$ then $\Delta_{t} \not \subset \Delta_0$ or $\Delta_{t} = \Delta_0$ up to isotopy fixing $\partial S$ and under the natural projection $\mathcal P: S_t \to S$. 
(In the following for the sake of  simplicity we denote the images of arcs and discs under $\mathcal P$ by the same symbols.) 
Since $\sgn(h_1)=+1$ and $\sgn(v)=-1$, Remark~\ref{remark sign describing arc} implies that the describing arc for $h_1$ lies on the right side of $b_0$ as we walk from $v$, hence $b_{t_1+\e}$ is contained in $\Delta_0$. 
That is, $b_{t_1+\e}$ cobounds a disc $\Delta_{t_1+\e}$ inside of $\Delta_0$, which contradicts  the minimality of $\Delta_0$.

Case (ii): 
The family of b-arcs $\{ b_t | \ t_k < t \leq 1, \ 0 \leq t < t_1 \}$ are isotopic relative to the binding. 
Since $\sgn(h_k)=+1$ and $\sgn(v)=-1$, Remark~\ref{remark sign describing arc} implies that the describing arc for $h_k$ lies on the right side of $b_{t_k-\e}$ as we walk from $v$. To get the boundary parallel b-arc $b_{t_k + \e}$ the describing arc must intersect $b_{t_k-\e}$ in the interior, which violates the requirement that a describing arc is properly embedded in $S_t \setminus (S_t \cap F)$. 
\end{proof}

We continue the proof of Theorem \ref{theorem:nrv}. 
Recall that all the hyperbolic points $h_1,\dots,h_k$ are positive. This means the same argument as in the proof of Lemma \ref{lemma:estimate} implies
\[ \phi(b_1) = b_0 > b_{t_1+\e} > b_{t_2 + \e} > \cdots > b_{t_k+\e} = b_1 \mbox{ near } v. \] 
Since $b_1$ is an essential arc in $S_1$ by Claim~\ref{claim v} we conclude that $\phi$ is not right-veering with respect to $C$.  
\end{proof}

Recall that every transverse overtwisted disc contains at least one negative elliptic point and $e_-(\F(F))$ denotes the number of negative elliptic points in the open book foliation $\F(F)$ (Definition~\ref{def of e}). 

\begin{definition}
Let 
\[
n(S, \phi) = \min \left\{
e_-(\F(D)) \ | \ D \subset M_{(S, \phi)}: \mbox{transverse overtwisted disk} \right\}.
\]
If $\xi_{(S, \phi)}$ is tight let $n(S, \phi)=0$. 
It is an invariant of open books and we call it the {\em overtwisted complexity}. 
\end{definition}

\begin{corollary}\label{cor:n(S,phi)}
As a consequence of Theorem~\ref{theorem:nrv} we have:
\begin{itemize}
\item
$n(S, \phi) = 0$ if and only if $\xi_{(S, \phi)}$ is tight (and hence $\phi$ is right veering). 
\item 
$n(S, \phi) = 1$ if and only if $\xi_{(S, \phi)}$ is overtwisted and $\phi$ is non right veering. 
\item 
$n(S, \phi) \geq 2$ if and only if $\xi_{(S, \phi)}$ is overtwisted and $\phi$ is right veering. 
\end{itemize}
\end{corollary}

\begin{example}
Infinitely many examples of $n(S, \phi) \geq 2$ exist: 
The examples in \cite[Example 3.1]{ik1-2} have $n(S, \phi)=2$ and their open books are destabilizable. 
On the other hand, 
Theorem 4.1 of \cite{ik1-2} has infinitely many examples of $n(S, \phi)=2$ with non-destabilizable open books. 
\end{example}

We give more sufficient conditions for non-right-veeringness. 
Recall that a bc-annulus is {\em degenerate} if the two boundary b-arcs are identified. 
Therefore, a degenerate bc-annulus is topologically a disc and neighborhood of the boundary is foliated by c-circles as shown in Figure~\ref{fig:degeneratebc}.
\begin{figure}[htbp]
\begin{center}
\SetLabels
(0.5*0.9) {\small identified}\\
(0.2*0.0) {bc-annulus}\\
(0.78*0.0) {degenerate bc-annulus}\\
\endSetLabels
\strut\AffixLabels{\includegraphics*[width=85mm]{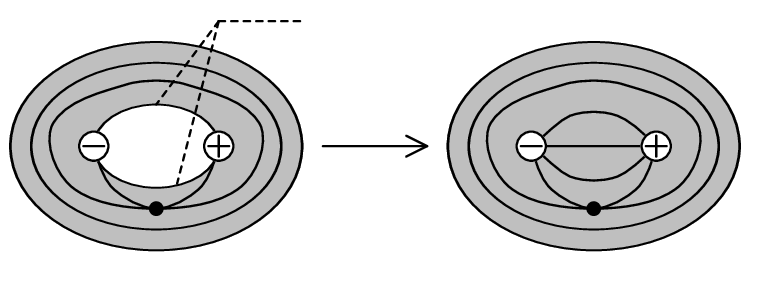}}
\end{center}
\caption{A degenerate $bc$-annulus.}
\label{fig:degeneratebc}
\end{figure}

\begin{proposition}
\label{theorem:otc-circle}
Suppose that there exists a (possibly closed) surface $F$ in $M_{(S, \phi)}$  containing a degenerate  bc-annulus $R$  whose c-circles are essential (Definition~\ref{def of essential}). 
Then $\phi$ is not right-veering.
\end{proposition}

\begin{proof}
Let $S_{t_0}$ be the fiber on which the unique hyperbolic point $h \in \F(R)$ lies.
Let $v^{+}$ and $v^{-}$ be the positive and the negative elliptic points of $\F(R)$. 
As in the proof of Theorem \ref{theorem:nrv}, we first observe the following.

\begin{claim}\label{claim:v}
$v^{\pm}$ are strongly essential.
\end{claim}

\begin{proof} 
Assume contrary that $v^{\pm}$ are not strongly essential. 
Then, for {\em any} $t \neq t_0$ the b-arc $b_{t} \subset R \cap S_{t}$ joining $v^+$ and $v^-$ cobounds a disc $\Delta_{t} \subset S_{t}$ with a binding component.
At $t=t_0$ a c-circle and a b-arc meet and form the hyperbolic point, $h$. After the configuration change they become a single b-arc. 
For $t\in (t_0-2\e, t_0)$ let $c_{t} \subset R \cap S_t$ denote the c-circle.
There are two cases to consider:
\begin{enumerate} 
\item If $c_{t_0-\e} \subset \Delta_{t_0-\e}$ then $c_{t_0-\e}$ bounds a disc $X \subset \Delta_{t_0-\e}\subset S_{t_0-\e}$.
\item If $c_{t_0-\e} \subset (S_{t_0-\e} \setminus \Delta_{t_0-\e})$ then $c_{t_0-\e} \# b_{t_0-\e} \simeq b_{t_0+\e} \subset \partial \Delta_{t_0+\e}$, so $c_{t_0-\e}$ must bound a disc $X \subset S_{t_0-\e}\setminus\Delta_{t_0-\e}$. 
\end{enumerate}

Let $m= | X \cap \partial F |$ and $n=|\Delta_{t_0-\e} \cap \partial F |$. 
As a consequence of the configuration change before and after the hyperbolic point $h$ we have 
$$|\Delta_{t_0 + \e} \cap \partial F | = 
\left\{
\begin{array}{ll}
n-m & \mbox{ for case (1),}\\
n+m & \mbox{ for case (2).}
\end{array}
\right.  
$$
Since $h$ is the only hyperbolic point connected to $v^\pm$ we have 
$$
b_t \simeq \left\{
\begin{array}{cc}
b_1 & \mbox{ for } t_0 < t \leq 1 \\
b_0 & \mbox{ for } 0 \leq  t < t_0 
\end{array}
\right.
$$
where ``$\simeq$'' means isotopic relative to the boundary points $v^{\pm}$. 
Therefore the discs satisfy $\Delta_t \simeq \Delta_1$  (isotopic rel. the binding) for $t_0 < t \leq 1$
and 
$\Delta_t \simeq \Delta_0$ for $0 \leq t < t_0$. 
Note the monodromy $\phi$ identifies $\Delta_0$ and $\Delta_1$.  
Thus we have
\[ |\Delta_{t_0 - \e} \cap \partial F| = |\Delta_{t_0 + \e} \cap \partial F|.\]
This shows $m=0$ for either case, i.e., $c_t$ is inessential, which contradicts our  assumption.
\end{proof}

We continue the proof of Proposition~\ref{theorem:otc-circle}. 

Since the c-circles $c_t$ are essential we have 
$$
\phi(b_1)= b_0 = b_{t_0 - \e}  > b_{t_0 + \e}=b_1 \ \mbox{ near }v^{\mp} \mbox{ if } \sgn(h) = \pm1.
$$
The above inequality can be justified by 
Figure~\ref{sgn_h}, which describes a neighborhood of $b_{t_0-\e}$ in $S_{t_0-\e}$.  
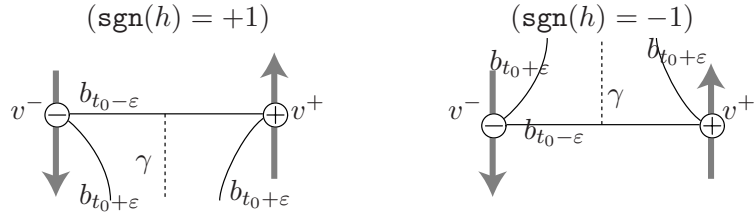
\begin{figure}[htbp]
\begin{center}
\SetLabels
(-.02*.5) $v^-$\\
(.39*.5) $v^+$\\
(.62*.5) $v^-$\\
(1.02*.5) $v^+$\\
(.15*.2) $\gamma$\\
(.1*.6) $b_{t_0-\e}$\\
(.1*-.02) $b_{t_0 +\e}$\\
(.32*0) $b_{t_0 +\e}$\\
(.84*.6) $\gamma$\\
(.75*.38) $b_{t_0-\e}$\\
(.7*.8) $b_{t_0 +\e}$\\
(.93*.87) $b_{t_0 +\e}$\\
(.2*1.07) ($\sgn(h)=+1$)\\
(.82*1.07) ($\sgn(h)=-1$)\\
\endSetLabels
\strut\AffixLabels{\includegraphics*[width=90mm]{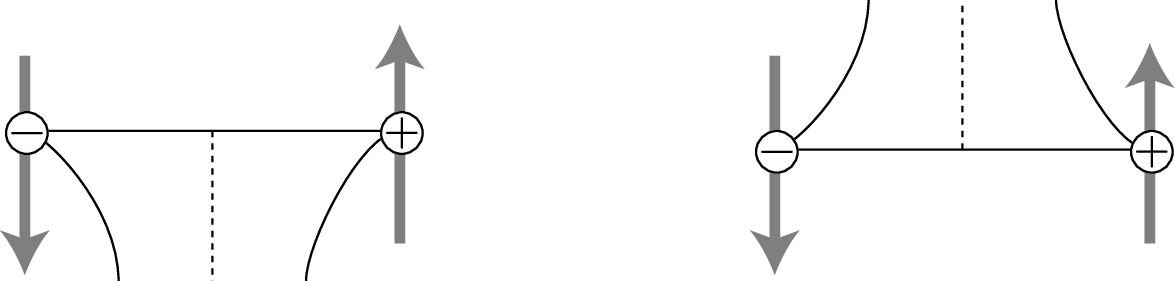}}
\end{center}
\caption{The arc $\gamma$ is a describing arc for $h$ and arrows indicate the orientation of the binding. }
\label{sgn_h}
\end{figure}

Let $C^{\pm} \subset \partial S$ (possibly $C^{+}=C^{-}$) be the binding components of the open book that contains $v^{\pm}$ respectively . 
By Claim~\ref{claim:v} every b-arc $b_t$ is strongly essential, hence $b_t$ is an essential arc in $S$. 
Applying Lemma~\ref{lemma:fracDehn} we obtain 
$ c(\phi, C^{\mp}) < 0$ when $\sgn(h)=\pm1$, respectively, i.e., $\phi$ is not right veering. 
\end{proof}

The conditions in the following corollary guarantee the existence of a degenerate bc-annulus of Proposition~\ref{theorem:otc-circle}. 

\begin{corollary}
Let $D \subset (S, \phi)$ be a disc with an open book foliation $\F(D)$ satisfying
\begin{enumerate}
\item $sl(\partial D, [D])=1$, 
\item $\F(D)$ contains c-circles, and
\item $e_-(\F(D))=1$. 
\end{enumerate}
Then $\phi\in\Aut(S, \partial S)$ is non-right-veering.
\end{corollary}

\begin{proof}
By (2) the region decomposition of $D$ contains at least one ac-annulus, or cc-pants.  
By the condition ($\mathcal F$ ii) of Definition~\ref{def of OBF} we first note that:
\begin{equation}\label{c}
\mbox{
A c-circle cannot be the boundary of the disc } D. 
\end{equation}
First, suppose that $D$ contains a pair of cc-pants. 
An Euler characteristic argument and (\ref{c}) imply that there must exist at least two bc-annuli that fill in two holes of the cc-pants. 
Next, suppose that $D$ contains an ac-annulus and no cc-pants. Again by an Euler characteristic argument and (\ref{c}) the ac-annulus must be glued to a bc-annulus. 
Therefore, in either case the region decomposition of $D$ contains a bc-annulus, which we denote by $R$.

Assume that $R$ is non-degenerate. 
Let $D_c$ (resp. $D_b$) be the sub-disc of $D$ bounded by the c-circle boundary (resp. the union of the two boundary b-arcs) of $R$. 
If $D_{c} \subset D_b$ then the region decomposition of $D_c$ must contain another bc-annulus.
If $D_b \subset D_c$ then the region decomposition of $D_b$ must contain a bb-tile or a bc-annulus. 
In either case $D$ has more than one negative elliptic points which contradicts (3).

Hence the region decomposition of $D$ contains one degenerate bc-annulus $R$.  
By a similar argument we can prove that $D\setminus R$ contains one (possibly degenerate) ac-annulus (see Figure~\ref{pic of D}). 
To apply Proposition \ref{theorem:otc-circle} we show the following: 
\begin{figure}[htbp]
  \begin{center}
  \SetLabels
(0.3*0.58)   $v$\\  
(0.4*.3) ${\sf h}_{t_1}$\\ 
(0.65*.43) ${\sf h}_{t_0}$\\
  \endSetLabels
  \strut\AffixLabels{\includegraphics[width=70mm]{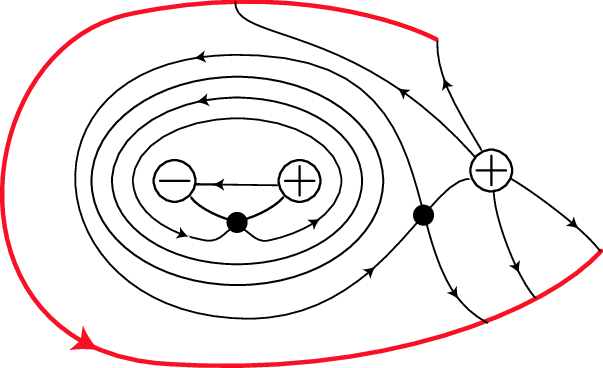}}
  \caption{An ac-annulus surrounds a degenerate bc-annulus.}
  \label{pic of D}
  \end{center}
\end{figure}

\begin{claim}\label{claim:c-essential}
The c-circles in $\F(D)$ are essential.
\end{claim}

\begin{proof} 
We may assume that the hyperbolic point of the ac-annulus (resp. $R$) lives in the page $S_{t_0}$ (resp. $S_{t_1}$) for some $t=t_0$ (resp. $t_1$) and we call it ${\sf h}_{t_0}$ (resp. ${\sf h}_{t_1}$). 
Proposition~\ref{sl-formula-1} and the assumptions (1) and (3) imply that $e_+(\F(D))=h_+(\F(D))$ and $h_-(\F(D))=0$. 
Therefore $\sgn({\sf h}_{t_0})=\sgn({\sf h}_{t_1})=+1$.

Suppose on the contrary that the family of c-circles $c_t \subset (D \cap S_t)$ for $t_0 < t<t_1$ are inessential.
Since $\sgn({\sf h}_{t_0})=+1$ for a very small $\e>0$ the circle $c_{t_0+\e}$ bounds a disc in $S_{t_0+\e}$ on its left side with respect to the orientation of $c_{t_0+\e}$ (see the top row of Figure~\ref{fig:c-circle}).  
On the other hand, since $\sgn({\sf h}_{t_1})=+1$ the circle $c_{t_1-\e}$ bounds a disc in $S_{t_1-\e}$ on the right side of $c_{t_1-\e}$ (see the bottom row of Figure~\ref{fig:c-circle}).
Since $\{ c_t \: | \: t_0 < t<t_1 \}$ is a continuous family of oriented circles they can bound discs only on the same side, we get a contradiction. 
\begin{figure}[htbp]
\begin{center}
\SetLabels
(.1*.93) $t=t_0-\e$\\
(.7*.93) $t=t_0+\e$\\
(.1*.01) $t=t_1-\e$\\
(.8*.01) $t=t_1+\e$\\
\endSetLabels
\strut\AffixLabels{\includegraphics*[width=90mm]{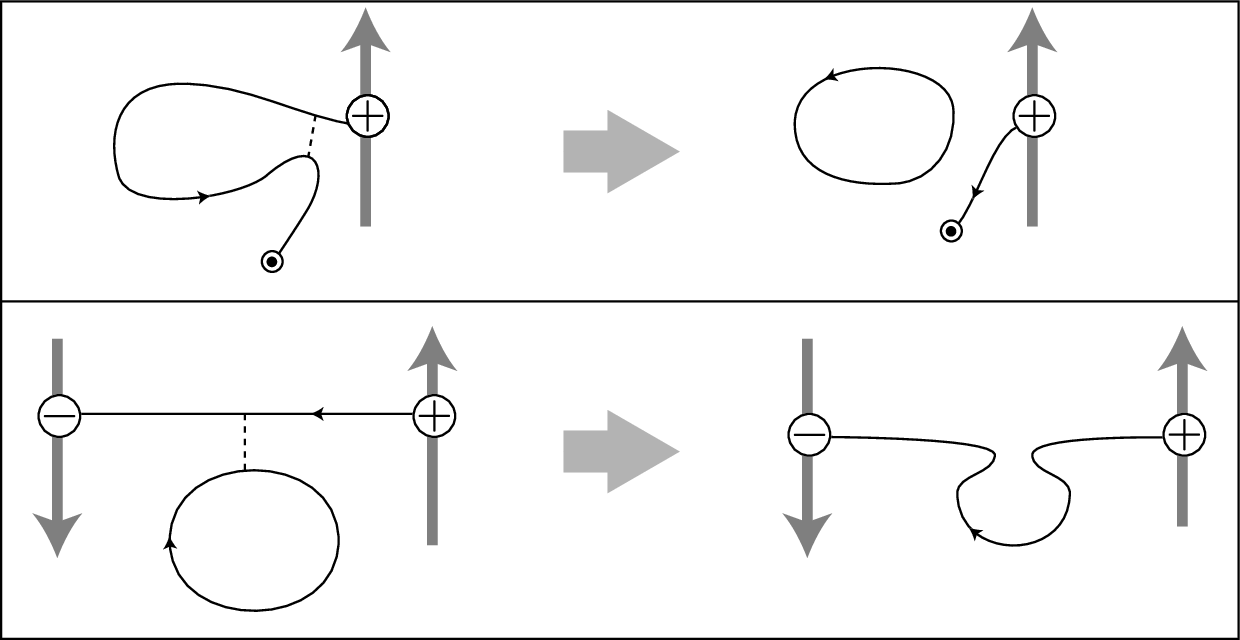}}
\end{center}
\caption{}
\label{fig:c-circle}
\end{figure}

\end{proof}

Claim~\ref{claim:c-essential} and Proposition~\ref{theorem:otc-circle} imply that $\phi$ is not right-veering.
\end{proof}


\section{Topology of open book manifolds}
\label{sec:topol}

In this section we study topology of the $3$-manifolds  $M = M_{(S, \phi)}$. 

\subsection{Incompressible surfaces and $c(\phi, C)$}

\label{sec:topol-surface}
In the following two theorems we establish estimates of $c(\phi, C)$ from topological data on incompressible closed surfaces in $M$.

\begin{theorem}[General case]
\label{theorem:surface}
Suppose that there exists a closed, oriented, incompressible, genus $g$ surface $F$ in $M$ which admits an essential open book foliation and intersects the binding in $2n$ $(n>0)$ points. 
\begin{enumerate}
\item If $g=0$ then $|c(\phi,C)| \leq 3$ for some boundary component $C$.
\item If $g \geq 1$ then $|c(\phi,C)| \leq 4 + \lfloor \frac{4g-4}{n} \rfloor$ for some boundary component $C$.
\end{enumerate}
\end{theorem}

If $S$ has connected boundary we get even sharper estimates: 

\begin{theorem}[Connected binding case]
\label{theorem:surface_connected}
Under the same setting of Theorem~\ref{theorem:surface}, 
assume further that $\partial S$ is connected.
Then
\[ |c(\phi, \partial S) | \leq \inf_{m \in \mathbb N} \mathcal G(m)\] 
where $\mathcal G: \mathbb{N} \rightarrow \Q$ is a map defined by:
\[ 
\mathcal G(m)=  \left\{ 
\begin{array}{ll}
 \frac{1}{m}\left\lceil \frac{(g-1+n)m}{n} - \frac{(n-1)^{2}}{4 n^{2}}  \right\rceil
& ( n \textrm{: odd}) \\
& \\
 \frac{1}{m} \left\lceil \frac{(g-1+n)m}{n} - \frac{n-2}{4n}  \right\rceil 
 & ( n \textrm{: even})
 \end{array}
 \right.
\]
In particular: 
\begin{enumerate}
\item If $g=0$ we have $|c(\phi,\partial S)| \leq 1.$
\item If $g\geq 1$ we have $|c(\phi,\partial S)| \leq g.$
\end{enumerate}
\end{theorem}

Once we have Theorems~\ref{theorem:surface} and \ref{theorem:surface_connected} in hand, we give criteria for irreducible and atoroidal manifolds. 

\begin{corollary}
\label{irreducible-theorem}
Assume that;
\begin{enumerate}
\item $|c(\phi,C)| > 3$ for every boundary component $C$ of $S$, or 
\item $\partial S$ is connected and $|c(\phi,\partial S)| > 1$.
\end{enumerate} 
Then the 3-manifold $M=M_{(S, \phi)}$ is irreducible. 
\end{corollary}

\begin{proof}
Assume that $M$ is reducible. 
Then by Theorem~\ref{theorem:weak-sphere}, there exists an essential sphere $\mathcal S$ that admits an essential open book foliation. Since $(e_+ + e_-) - (h_+ + h_-)=\chi(\mathcal S)=2$ we know that $\F(\mathcal S)$ has elliptic points, i.e., $\mathcal S$ intersects the binding. 
Now Theorem~\ref{theorem:surface}-(1) together with  Theorem~\ref{theorem:surface_connected}-(1) yields the contrapositive of the statement of Corollary~\ref{irreducible-theorem}. 
\end{proof}

\begin{corollary}[First atoroidality criterion]
\label{atoroidal-theorem} 

Assume that $\phi \in \Aut(S, \partial S)$ is of irreducible  type and that;
\begin{enumerate}
\item $|c(\phi,C)| > 4$ for every boundary component $C$ of $S$, or
\item $\partial S$ is connected and $|c(\phi,\partial S)| > 1$.
\end{enumerate} 
Then the 3-manifold $M=M_{(S, \phi)}$ is irreducible and atoroidal.
\end{corollary}

\begin{remark-unnumbered}
See Theorem~\ref{thm:tight-atoroidal} below for a related result. 
\end{remark-unnumbered}

\begin{proof} 
By Corollary~\ref{irreducible-theorem} we know that $M$ is irreducible so it remains to show that $M$ is atoroidal.
Assume contrary that $M$ contains an incompressible torus $T$. 
By Theorem~\ref{theorem:weak} we may assume that $T$ admits an essential open book foliation. 
Theorems \ref{theorem:surface}-(2) and \ref{theorem:surface_connected}-(2) guarantee that $T$ does not intersect the binding, i.e., $\F(T)$ contains no elliptic points. 
Since $(e_+ + e_-) - (h_+ + h_-)=\chi(T)=0$, $\F(T)$ contains no hyperbolic points, that is, all the leaves of $\F(T)$ are c-circles. 
Hence the sets of c-circles $T\cap S_1$ and $T\cap S_0$ are isotopic through the continuous family $T \cap S_t$.

On the other hand the monodromy imposes $\phi(T \cap S_{1}) = T \cap S_{0}$.  
Therefore $\phi(T\cap S_1)=T\cap S_1$ which contradicts the assumption that $\phi$ is of irreducible.
\end{proof}

\begin{remark-unnumbered}

Corollary~\ref{irreducible-theorem} is valid regardless of the type of $\phi$. 
However for pseudo-Anosov case \cite[Theorem 5.3, Corollary 5.4]{go} of Gabai and Oertel yield the following stronger result.  
They view $M_{(S, \phi)}$ as a Dehn filling of the mapping torus of $\phi$ and use essential lamination theory. 
\end{remark-unnumbered}

\begin{theorem}[Gabai, Oertel \cite{go}]\label{theorem:GO}
Let $\phi$ be of pseudo-Anosov. 
For each boundary component $C_i \subset \partial S$ put $c(\phi, C_i) = \frac{p_i}{q_i}$, where $q_i$ is the number of  prongs of the (un)stable lamination around $C_i$. Note that $p_i$ and $q_i$ may not be coprime. If $|p_i|>1$ for all the boundary components $C_i$ then the suspension of the (un)stable lamination of $\phi$ gives an essential lamination in $M_{(S,\phi)}$, hence $M_{(S,\phi)}$ is an irreducible manifold with infinite fundamental group. 
\end{theorem}

Before proving Theorems~\ref{theorem:surface} and \ref{theorem:surface_connected}, we list some observations:

\begin{observation}\label{obs} 
When $F$ is an incompressible closed surface in $M$, there are various convenient properties. 
\begin{enumerate}
\item 
Since $F$ has no boundary, every essential b-arc is strongly essential. In particular, every b-arc in an essential open book foliation is strongly essential. 

\item Since $\F(F)$ does not have a-arcs the region decomposition of $\F(F)$ consists only of bb-tile, bc-annulus or cc-pants.

\item Since $\F(F)$ does not have a-arcs we have $e_+(\F)=e_-(\F)$. 
\end{enumerate}
\end{observation}

Now we prove Theorem \ref{theorem:surface}.
We analyze the Euler characteristic and the region decomposition of a surface to find an elliptic point $v$ such that the number of hyperbolic points connected to $v$ by a singular leaf is small. 
Then we apply Lemma \ref{lemma:estimate} to get the desired estimate of $c(\phi, C)$. 
Our proof is highly motivated by \cite[Theorem 1.2]{i1}, a result in braid foliation theory. 
In fact, Theorem \ref{theorem:surface} can be seen as a generalization of \cite[Theorem 1.2]{i1}.

\begin{proof}[Proof of Theorem \ref{theorem:surface}]
Let $F$ be a surface satisfying the conditions in Theorem \ref{theorem:surface}. 
We construct a cellular decomposition of $F$ by modifying the region decomposition of $\F(F)$. 
To this end, we construct a singular foliation (but not an open book foliation), $\mF'$, on $F$ by replacing each cc-pants of $\F(F)$ with three bc-annuli as shown in Figure~ \ref{fig:modify-cc}:
\begin{figure}[htbp]
\begin{center}
\SetLabels
(1.1*.95) fake \\
(1.1*0.85) elliptic points\\
(1.15*0.2) fake hyperbolic\\
(1.1*0.1) points\\
(.2*-.1) cc-pants\\
(.8*-.1) three bc-annuli\\
 \endSetLabels
\strut\AffixLabels{\includegraphics*[width=90mm]{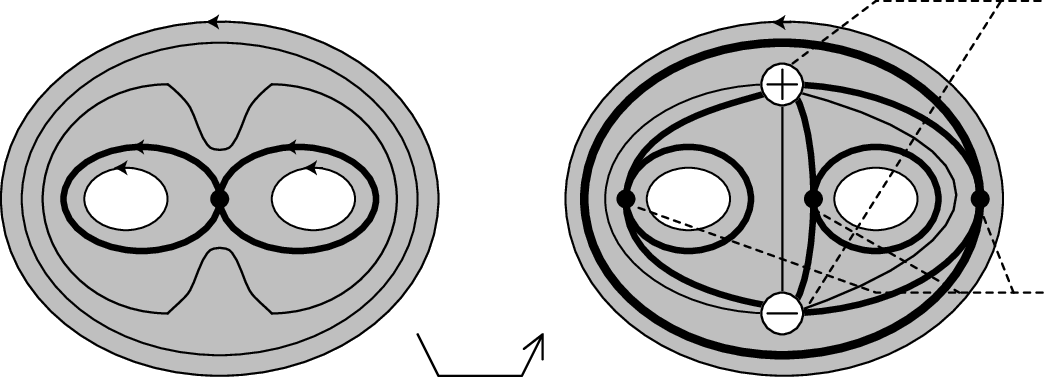}}
 \caption{Construction of a singular foliation $\mF'$.}
\label{fig:modify-cc}
\end{center}
\end{figure}
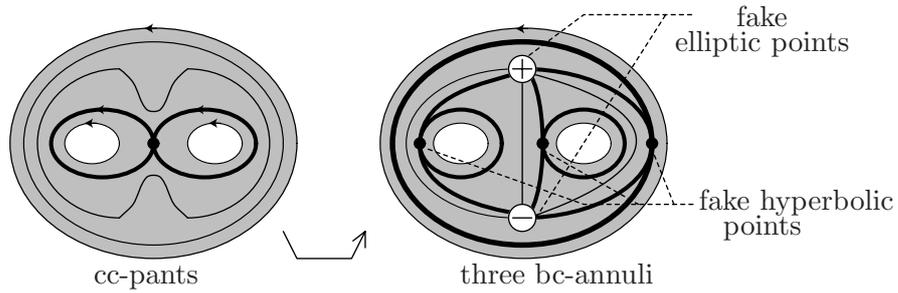
Though $\mF'$ is not derived from the intersection of the surface $F$ with the pages, by abuse of notations we keep using the terminologies of open book foliations,  such as region decomposition, bc-tile, bc-singular points, etc.
We call the newly-inserted elliptic points and hyperbolic points {\em fake} elliptic points and {\em fake} hyperbolic points, respectively. 

The signs of fake elliptic points are canonically determined by the orientation of the c-circles in the original cc-pants. However, the signs of elliptic and hyperbolic points (both fake and non-fake) are not used in the following argument so we may omit signs from now on.

The region decomposition of $\mF'$ consists only of bb-tiles and bc-annuli.
Using Birman and Menasco's idea \cite{BM7} we construct a cellular decomposition of $F$ as follows:  
bc-annului always exist in pair because the c-circle boundary of a bc-annulus is identified with the c-circle boundary of another bc-annulus.
Let $W$ be the annulus obtained by two bc-annuli glued along their c-circle boundaries (note: $W$ is a disc if one of the bc-annuli is degenerate; also ac-annuli do not exist because $F$ is closed).  
Each component of $\partial W$ has two elliptic points. 

We cut $W$ along two disjoint essential arcs connecting the elliptic points of $W$ to obtain two quadrilaterals.
Call such arcs {\em e-edges} and quadrilaterals {\em be-tiles}. 

The two elliptic points of an e-edge do not have to have opposite signs when one of its end point is a fake elliptic point. 
Although it is not necessary in this proof, for the later use we choose e-edges so that each be-tile contains exactly one hyperbolic point.

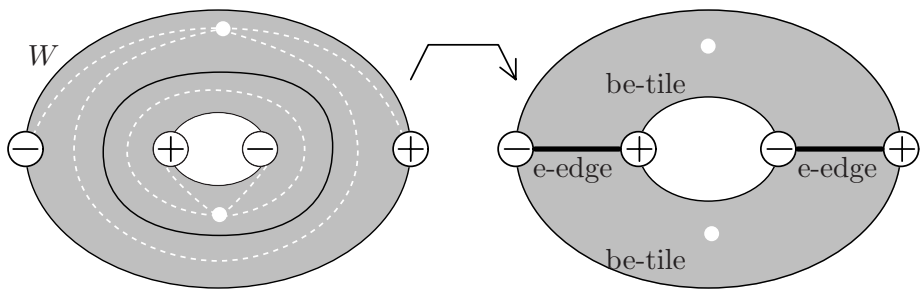
\begin{figure}[htbp]
\begin{center}
\SetLabels
(.04*.8) $W$\\
(0.62*0.4) e-edge \\
(0.91*0.4) e-edge \\
(.7*.7) be-tile\\
(.7*.07) be-tile\\
\endSetLabels
\strut\AffixLabels{\includegraphics*[scale=0.5, width=120mm]{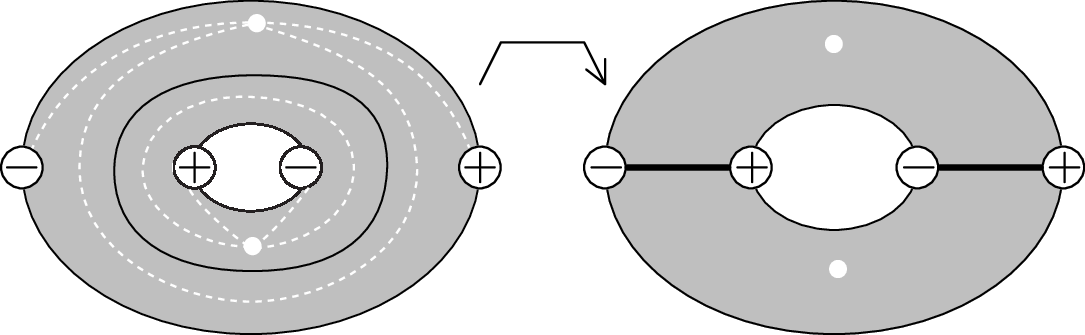}}
\caption{From bc-annuli to be-tiles. (cf. \cite[Fig 18]{BM7})}
\label{fig:be-tile}
\end{center}
\end{figure}

Now the surface $F$ is a union of bb-tiles and be-tiles.
This defines a cellular decomposition of $F$ whose  $0$-cells are the elliptic points and the fake elliptic points, $1$-cells are the boundary b-arcs and e-edges of the bb-tiles and the be-tiles, and $2$-cells are the bb-tiles and be-tiles.

Let $v$ be a $0$-cell and $\Val(v)$ denote the valence of $v$ in the $1$-skeleton graph. 
Let $\Hyp(v)$ be the number of hyperbolic points in the original foliation $\F(F)$ that are connected to $v$ by a singular leaf.

\begin{claim}
\label{claim:fake}
let $v$ be a $0$-cell of the cellular decomposition of $F$. 
\begin{enumerate}
\item If $v$ is a fake elliptic point of $\mF'$, then $\Val(v)=6$.
\item If $v$ is not a fake elliptic point of $\mF'$, then $\Hyp(v) \leq \Val(v)$.
\end{enumerate}
\end{claim}

\begin{proof}[Proof of claim~\ref{claim:fake}]
(1) 
If $v$ is a fake elliptic point then $v$ sits on some cc-pants $P$ in the region decomposition. 
In $\mF'$, $P$ decomposes into three bc-annuli (Figure~\ref{fig:modify-cc}). 
After replacing the bc-annuli with be-tiles, at $v$ three b-arc $1$-cells and three e-edge $1$-cells meet, see Figure~\ref{fig:cc-be}. Thus $\Val(v)=6$. 

\begin{figure}[htbp]
 \begin{center}
\SetLabels
(-.05*.5) $P=$\\
 \endSetLabels
\strut\AffixLabels{\includegraphics*[scale=0.5, width=120mm]{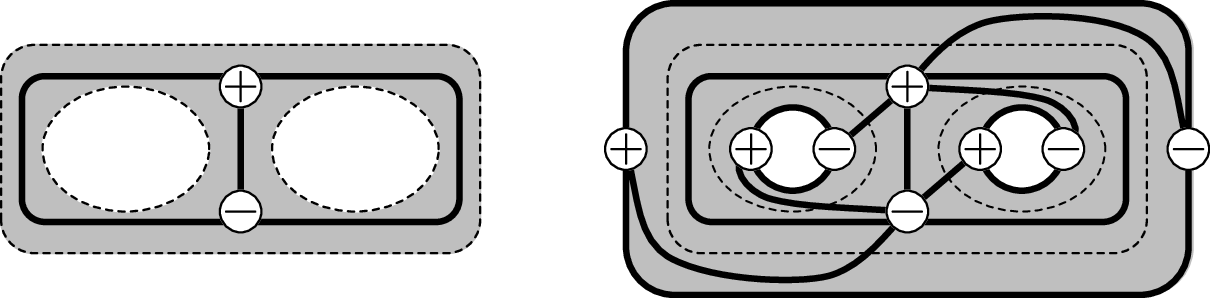}}
 \caption{A fake elliptic point has valence six.}
\label{fig:cc-be}
  \end{center}
\end{figure}

(2) 
We first note that a vertex $v$ is a non-fake elliptic point in $\mF'$ if and only if it is an elliptic point of the original foliation $\F(F)$.
If $v$ lies in the interior of a degenerate bc-annulus in $\F(F)$ then $\Hyp(v)=1 < 2 = \Val(v)$.

Next assume that $x$ $bb$-tiles and $y$ non-degenerate bc-annuli in $\F(F)$ meet at $v$.
Then $\Hyp(v) = x + y \leq x+2y = \Val(v)$. 
\end{proof}

Let us define
\[ s = \min\left\{\Hyp(v) \: | \: v \textrm{ is a 0-cell and a non-fake elliptic point  }\right\}.\]
Let $v$ be a vertex realizing $\Hyp(v)=s$. 
Suppose that $v$ lies on the binding component $C \subset \partial S$. 
Observation~\ref{obs}-(1) guarantees that $v$ is strongly essential. 
By Lemma~\ref{lemma:estimate} 
\[ -s\leq c(\phi,C) \leq s. \]
Our goal is to show that $s\leq3$ if $g=0$ and $s \leq 4 + \lfloor \frac{4g-4}{n} \rfloor$ if $g>0$.

Consider the cellular decomposition of $F$. 
For $i \geq 1$ let $V(i)$ be the number of $0$-cells of valence $i$, and let $E$ be the number of $1$-cells, and $R$ be the number of $2$-cells. Since each $1$-cell is a common boundary of distinct two $2$-cells and the boundary of each $2$-cell consists of four  distinct $1$-cells, we have:
\begin{equation}
\label{eqn:ER}
2E=4R. 
\end{equation}
Since the end points of a $1$-cell are distinct two $0$-cells we have:
\begin{equation}
\label{eqn:VE}
\sum_{i\geq 1} i V(i) = 2E. 
\end{equation}
The Euler characteristic of $F$ is:
\begin{equation}
\label{eqn:euler}
 \sum_{i\geq 1} V(i) - E + R = \chi(F) 
\end{equation}
From (\ref{eqn:ER}), (\ref{eqn:VE})  and (\ref{eqn:euler}), we get the {\em Euler characteristic equality}:
\begin{equation}
\label{eqn:ecequality}
 \sum_{i\geq 1}(4-i)V(i) = 4 \chi(F).  
\end{equation}

(1) 
First we assume that $F$ is a sphere. 
The equality (\ref{eqn:ecequality}) implies: 
\[ 3V(1) + 2V(2) + V(3) = 8 + \sum_{i \geq 4}(i-4)V(i) \]
The right hand side 
is positive.
So there exists a vertex  $v$ with $\Val(v) \leq 3$. 
Claim~\ref{claim:fake}-(1) implies that $v$ is not a fake vertex. 
By Claim \ref{claim:fake}-(2) we obtain
\[ |c(\phi, C)| \leq s \leq \Hyp(v) \leq \Val(v) \leq 3.\]

(2) 
Next we assume that $F$ has genus $g>0$ so $\chi(F)= 2-2g \leq 0$.
The Euler characteristic equality (\ref{eqn:ecequality}) gives:
\[ 0 \leq 3V(1)+2V(2) + V(3) + 8g -8 = \sum_{i \geq 4}(i-4)V(i). \]
If at least one of $V(1),V(2)$ and $V(3)$ is positive then by Claim~\ref{claim:fake} there exists a non-fake elliptic point $v$ such that 
$s \leq \Hyp(v) \leq \Val(v) \leq 3$.
Suppose that $V(1)=V(2)=V(3)=0$.
By Observation~\ref{obs}-(3) the original open book foliation $\F(F)$ contains an even number ($=2n$) of elliptic points. Therefore;
\[ 8g-8 = \sum_{i\geq 4}(i-4)V(i) \geq (s-4)2n, \]
i.e., 
$s \leq 4 + \frac{4g-4}{n}$. 
In either case since $s$ is an integer $s \leq 4 + \lfloor \frac{4g-4}{n} \rfloor$.
\end{proof}

Finally we prove Theorem \ref{theorem:surface_connected} by using Theorem~\ref{theorem:estimate} and Observation~\ref{obs}.

\begin{proof}[Proof of Theorem \ref{theorem:surface_connected}]
We see in Observation~\ref{obs}-(3) that $e_{-}(\F(F)) = e_{+}(\F(F))= n$.
By Proposition \ref{sl-formula-1}-(2),
\begin{equation}\label{h_+ + h_-}
h_{+} + h_{-} = -\chi(F) + e_+ + e_- = 2g-2+2n.
\end{equation} 
Since $\partial S$ is connected all the elliptic points lie on the same boundary component and Theorem~\ref{theorem:estimate} implies: 
\[
\left| c(\phi, \partial S) \right| \leq 
\min\left\{
\inf_{m\in \mathbb N}
f_+(m),  \ \inf_{m\in \mathbb N} f_-(m)
\right\}
\]
where 
$$
f_\pm(m) = \left\{
\begin{array}{ll}
\frac{1}{m} \lceil \frac{h_\pm m}{n} - \frac{(n-1)^2}{4n^2} \rceil 
& (n: {\mbox odd})\\
& \\
\frac{1}{m} \lceil \frac{h_\pm m}{n} - \frac{n-2}{4n} \rceil 
& (n: {\mbox even})
\end{array}
\right.
$$
Hence by (\ref{h_+ + h_-}) we get: 
\[
\left| c(\phi, \partial S) \right| \leq  \left\{
\begin{array}{ll}
\inf_{m\in\mathbb N} 
\left(\frac{1}{m} 
\left\lceil \frac{(g-1+n) m}{n}  
- \frac{(n-1)^2}{4n^2}  \right\rceil \right) & (n: {\mbox odd})\\
 & \\
\inf_{m\in\mathbb N} 
\left(\frac{1}{m} 
\left\lceil \frac{(g-1+n) m}{n} 
 - \frac{n-2}{4n}  \right\rceil \right) & (n: {\mbox even})
\end{array}
\right.
\]
\end{proof}

\subsection{Tight contact and atoroidal manifolds}

With tightness assumption on contact manifolds we refine the atoroidality criterion in Corollary~\ref{atoroidal-theorem}.

\begin{theorem}[Second atoroidality criterion]\label{thm:tight-atoroidal}
Let $(S,\phi)$ be an open book supporting a tight contact structure.
If $\phi$ is of irreducible type and $c(\phi, C)> 2$ for every boundary component $C$ of $S$ then $M_{(S, \phi)}$ is atoroidal. 
\end{theorem}

\begin{proof}
Let $\mathcal{T}$ be an incompressible torus in $M_{(S,\phi)}$.  
By Theorem~\ref{theorem:weak} we may assume that $\mathcal T$ admits an essential open book foliation $\F(\mathcal{T})$. 
If $\mathcal{T}$ does not intersect the binding, the proof of Corollary~\ref{atoroidal-theorem} implies that $\phi$ is reducible, which is a contradiction. 
So $\F(\mathcal T)$ must contain elliptic points.
In fact, $\F(\mathcal{T})$ contains at least two elliptic points since  $e_{+}=e_{-}$ (see Observation~\ref{obs}-(3)).

Consider the cellular decomposition of $\mathcal{T}$ as in the proof of Theorem \ref{theorem:surface} whose  $2$-cells are bb-tiles or be-tiles.

First we assume that the region decomposition of $\mathcal T$ contains no cc-pants. This implies that there are no fake elliptic or hyperbolic points.

For a positive elliptic point $v \in \F(\mathcal T)$ let $N(v)$ denote the number of 2-cells around $v$ containing positive hyperbolic points, and $\Hyp^+(v)$ be the number of positive hyperbolic points that are connected to $v$ by a singular leaf. 
If $v$ lies on the binding component $C \subset \partial S$ then Lemma~\ref{lemma:estimate} and Observation~\ref{obs}-(1) imply
\begin{equation}
\label{eqn:N(v)}
c(\phi,C) \leq \Hyp^+(v) \leq N(v).
\end{equation}
The strict inequality $\Hyp^+(v)< N(v)$ may hold in the following case: 
See Figure \ref{fig:be-tile} again where the annulus $W$, the union of two bc-annuli, is decomposed into two be-tiles. 
Recall that each be-tile contains exactly one hyperbolic point. 
Consider the case that each of the two be-tiles contains a positive hyperbolic point.  
That is, the total contribution of the be-tiles to $N(v)$ is two. 
Next we note that exactly one of the two hyperbolic points is connected to $v$ by a singular leaf.
Thus the total contribution of the two hyperbolic points in $W$ to $\Hyp^+(v)$ is one. 

Let $e_{\pm}$ and $h_{\pm}$ be the numbers of $(\pm)$ elliptic/hyperbolic points in $\F(\mathcal T)$. 
For $i\geq 0$ let $w_{i}$ be the number of positive elliptic points with $N(v)=i$.
We have
\begin{equation}
\label{eqn:eul1}
\sum_{i\geq 0} w_{i} = e_{+}.
\end{equation}
Since each 2-cell contains exactly two positive elliptic points, 
\begin{equation}
\label{eqn:eul2}
\sum_{i\geq 0} iw_{i} = 2h_{+}.
\end{equation}
By the Poincar\'e-Hopf formula,
\[ 
0= \chi(\mathcal{T}) = (e_+ + e_-) - (h_+ + h_-). 
\]
Let $e(\xi)$ be the Euler class of the tight contact structure $\xi=\xi_{(S, \phi)}$ supported by $(S,\phi)$. By the Bennequin-Eliashberg inequality \cite{el2}, we have
\[ |(e_{+} - e_{-}) -(h_{+} - h_{-}) | = |\langle [\mathcal{T}],e(\xi ) \rangle | \leq -\chi(\mathcal{T}) =  0.\]
Since $e_{+} =e_{-} >0$ we conclude that
\begin{equation}
\label{eqn:tight}
e_{+} = e_{-} = h_{+} = h_{-} >0. 
\end{equation}
By (\ref{eqn:eul1}), (\ref{eqn:eul2}) and (\ref{eqn:tight}), we have
\[ 2w_0 + w_1 = \sum_{i>2} (i-2)w_i >0,\]
or $w_2>0$ and $w_i =0$ for $i \neq 2.$
This shows that either $w_0, w_1$ or $w_2$ is positive. 
Hence there exists a positive elliptic point $v$ with $N(v) \leq 2$, so by (\ref{eqn:N(v)}) we get $c(\phi,C) \leq 2$ for the binding component $C$ on which $v$ lies.

Next we consider the case where $\F(\mathcal{T})$ contains a cc-pants. At least one of the three boundary c-circles of the  cc-pants bounds a disc in $\mathcal{T}$ and such disc must contain bc-annuli. 
Among them, consider an innermost bc-annulus, $R$. 
Here `innermost' means that the c-circle boundary of $R$ bounds a disc $D$ such that $R \subset D$ and the (possibly empty) sub-disc $\Delta = D \setminus R$ has no c-circles. 
In other words, $\Delta$ is the union of bb-tiles. 
Let $v^{*}$ be the positive elliptic point of $R$ and $K$ be a transverse unknot in $D$ which is a small perturbation of $\partial \Delta$  (see Figure \ref{fig:innermostbc}).
\begin{figure}[htbp]
  \begin{center}
\SetLabels
(0.3*0.45) {\LARGE $\Delta$}\\
(0.52*0.25) {\large $R$}\\
(0.47*0.76) {\large $K$}\\
(0.17*0.5) {\large $v^{*}$}\\
(0.79*0.82) {foliated }\\
(0.84*0.7) {without c-circles}\\
\endSetLabels
\strut\AffixLabels{\includegraphics*[width=80mm]{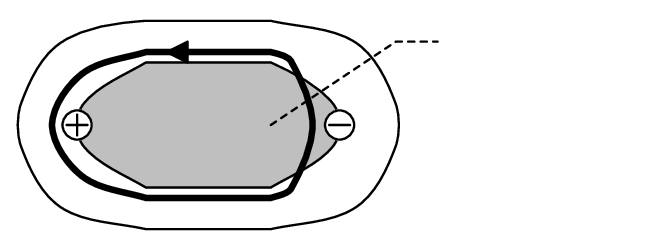}}
 \caption{An innermost bc-annulus $R$ and a subdisc $\Delta$}
\label{fig:innermostbc}
\end{center}
\end{figure}
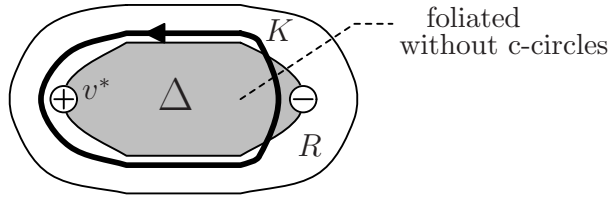

Let $v$ be a positive elliptic point in the closure of $\Delta$. 
As in the case without c-circles, let $N(v)$ denote the number of 2-cells around $v$ in $\Delta$ that contain positive hyperbolic points, and $w_{i}$ be the number of positive elliptic points in $\Delta$ with $N(v)=i$. Then 
\begin{equation}
\label{eqn:N2(v)}
c(\phi,C) \leq  N(v) \textrm{ if } v \neq v^{*}, \ \ 
c(\phi,C) \leq N(v^{*})+1 \textrm{ if } v = v^{*}
\end{equation} 
Here `$+1$' for the case $v=v^{*}$ is needed if the bc-annulus $R$ contains a positive hyperbolic point. 

The rest of the argument is similar to the previous case.
Let $e_{\pm}$ and $h_{\pm}$ be the numbers of $(\pm)$ elliptic/hyperbolic points in $\Delta$.
Then 
\[ 
e_{+}= e_{-}, \ \  \sum_{i\geq 0} w_{i} = e_{+}, \ \ \sum_{i\geq 0}i w_{i} = 2h_{+}. 
\]
By the Poincar\'e-Hopf formula, 
\[ 2 = (e_+ + e_-) - (h_{+} + h_{-}). \]
By Bennequin-Eliashberg inequality for the transverse unknot $K$, 
\[ sl(K) = -1 + (h_{+} - h_{-}) \leq -1. \]
Combining these (in)equalities we conclude
\[ 2w_0+w_1 \geq 2 + \sum_{i\geq 2}(i-2) w_i \geq 2. \]
Hence there exists a positive elliptic point $v$ with $N(v) \leq 1$, so by (\ref{eqn:N2(v)}) we get $c(\phi,C) \leq 2$ for the binding component $C$ on which $v$ lies.
\end{proof}

If the contact structure $\xi_{(S, \phi)}$ is tight one may refine the estimates in Section \ref{sec:topol-surface} by the same technique as in the above proof, namely combination of Bennequin-Eliashberg inequality,  counts of vertices and Euler characteristic argument.

\subsection{Incompressible surfaces and $c(\phi, L, C)$}

In this section we establish estimates of $c(\phi,L,C)$. 

The following two propositions (Propositions~\ref{theorem:surface-braid} 
and \ref{theorem:surface_connected-braid}) 
are variations of Theorems~\ref{theorem:surface} and  \ref{theorem:surface_connected}, respectively. 
Their proofs are similar, except that we apply 
Lemma~\ref{lemma:estimate-braid} and Theorem~\ref{theorem:estimate-braid} instead of 
Lemma~\ref{lemma:estimate} and Theorem~\ref{theorem:estimate}, respectively. 
(We do not use Observation \ref{obs} (1), which does not hold for an incompressible surface in $M \setminus L$.)
This is because, as noted in the second paragraph of Section~\ref{sec:estimate_braid}, to estimate $c(\phi,L,C)$  
b-arcs need not be strongly essential.

\begin{proposition}
\label{theorem:surface-braid}
Let $L$ be a closed $n$-braid in an open book $(S,\phi)$.
Suppose that there exists a closed, oriented, incompressible, genus $g$ surface $F$ in $M-L$ which admits an essential open book foliation and intersects the binding in $2k \ (>0)$ points. 
\begin{enumerate}
\item If $g=0$ then $|c(\phi, L,C)| \leq 3$ for some boundary component $C$ of $S$.
\item If $g>0$ then $|c(\phi, L,C)| \leq 4 + \lfloor \frac{4g-4}{k} \rfloor$ for some boundary component $C$ of $S$.
\end{enumerate}
\end{proposition}

\begin{proposition}[Connected binding case]
\label{theorem:surface_connected-braid}
Under the same setting of Proposition~\ref{theorem:surface-braid}, 
assume further that $\partial S$ is connected.
Then
\[ |c(\phi, L, \partial S) | \leq \inf_{m \in \mathbb N} \mathcal G(m)\] 
where $\mathcal G: \mathbb{N} \rightarrow \Q$ is a map defined by:
\[ 
\mathcal G(m)=  \left\{ 
\begin{array}{ll}
 \frac{1}{m}\left\lceil \frac{(g-1+k)m}{k} - \frac{(k-1)^{2}}{4 k^{2}}  \right\rceil
& ( k \textrm{: odd}) \\
& \\
 \frac{1}{m} \left\lceil \frac{(g-1+k)m}{k} - \frac{k-2}{4k}  \right\rceil 
 & ( k \textrm{: even})
 \end{array}
 \right.
\]
\end{proposition}

\begin{remark}\label{3rd-atoridal-criterion}
The same statements as 
Corollaries~\ref{irreducible-theorem} and \ref{atoroidal-theorem} (atoroidality criterion), where $M$ is replaced by $M-L$ and $c(\phi, C)$ is replaced by $c(\phi, L, C)$, hold. 
\end{remark}

Now we relate $c(\phi, L, C)$ and an incompressible Seifert surface of $L$.
The following theorem (or Corollary~\ref{cor:lower bound of g}) plays an essential role in the proof of our main result Theorem~\ref{theorem:geometry}. 

\begin{theorem}
\label{theorem:genus}
Let $L$ be a null-homologous, closed $n$-braid with respect to an open book $(S,\phi)$.
Let $F$ be a Seifert surface of $L$ realizing the maximal Euler characteristic, $\chi(F)$.
\begin{enumerate}

\item[(1a)]
If $\chi(F) > 0$ then $|c(\phi, L,C)| \leq 3$ for some  boundary component $C \subset \partial S$.
\item[(1b)] 
If $\chi(F) < 0$ and $F$ intersects the  bindings in $k\ (>0)$ points then there exists a binding component $C \subset \partial S$ such that
\[ 
|c(\phi, L, C)| \leq \min \left\{
\left\lfloor -\frac{4}{k}\chi(F) \right\rfloor  + 4, -\chi(F) + k \right\}
\]

\item[(2)]
Moreover, if $\partial S$ is connected and $\chi(F)\leq 0$ then
\[ |c(\phi, L,\partial S)| \leq \frac{n-\chi(F)}{n}. \]
\end{enumerate}
\end{theorem}

\begin{proof}

The idea of the proof is similar to that of Theorem \ref{theorem:surface}, but we need extra arguments because $F$ has non-empty boundary. 
Note that $F$ is incompressible.
By Theorem~\ref{theorem:weak} we may assume that the open book foliation $\F(F)$ is essential. 

(1) 
Consider the closed surface $\widehat{F}$ obtained by identifying each boundary component of $F$ with a point.
As in the proof of Theorem~\ref{theorem:surface}, we get a cellular decomposition of $\widehat{F}$ from the region decomposition of $\F(F)$. 
Since $F$ is not a closed surface we need the following operation in addition to the ones described in Figures \ref{fig:modify-cc} and \ref{fig:be-tile}:  
If there exists an ac-annulus it is paired up with either a bc-annulus as in the left sketch of Figure~\ref{fig:ad-tile} or a ac-annulus. 
In the former case, 
we cut the region into two pentagons along an e-edge and an essential arc, called a {\em d-edge}, joining an elliptic point and the boundary of the surface.  
We call the pentagons {\em abde-tiles}. 
In the latter case we obtain two hexagons called {\em ad-tiles}. 
We may assume that each tile contains one hyperbolic point. 
\begin{figure}[htbp]
\begin{center}
\SetLabels
(.2*.87) ac-annulus\\
(.2*.7) bc-annulus\\
(0.62*0.4) d-edge \\
(0.89*0.4) e-edge \\
(.75*.77) abde-tile\\
(.75*.20) abde-tile\\
(.76*.02) a-arc\\
(.76*.98) a-arc\\
(.76*.35) b-arc\\
(.76*.65) b-arc\\
\endSetLabels
\strut\AffixLabels{\includegraphics*[scale=0.5, width=100mm]{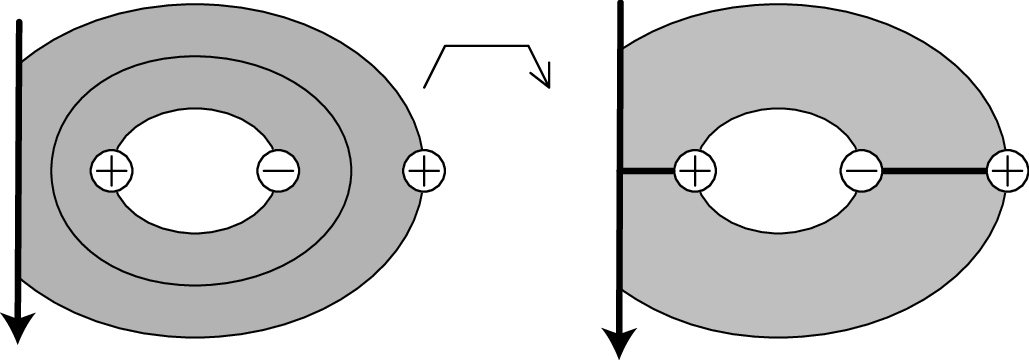}}
\caption{From a pair of ac- and bc-annuli to two abde-tiles. }\label{fig:ad-tile}
\end{center}
\end{figure}
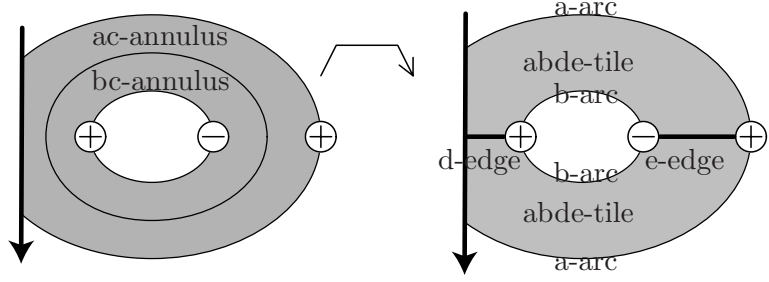

Consider the cell decomposition of $\widehat F$ where 
a $2$-cell is deformation of either aa-, ab-, bb-, be-, ad- or abde-tile under the operation $F \to \widehat F$.
A $0$-cell is either an elliptic point, a fake elliptic point or a newly attached point to a boundary component of $F$, which we call a {\em boundary $0$-cell}. 
Let us call fake and non-fake elliptic points {\em interior $0$-cells}. 
Also call a $1$-cell that ends (resp. does not end) at a boundary $0$-cell {\em boundary $1$-cell} (resp. {\em interior $1$-cell}). 
A boundary $1$-cell is an a-arc or a d-edge, and an interior $1$-cell is a b-arc or an e-edge. 

We say that an interior $0$-cell $w$ is of {\em type $(i,j)$} if $w$ has valence $(i+j)$ and is a common endpoint of $i$ boundary $1$-cells and $j$ interior $1$-cells.
Let $V(i,j)$ be the number of interior $0$-cells of type $(i,j)$.
Let $E_{\partial}$ be the number of boundary $1$-cells, $E$ be the number of interior $1$-cells, and $R$ be the number of $2$-cells. 
Since each $1$-cell is a common boundary of two $2$-cells (degenerate $2$-cells are counted with multiplicity $=2$) and each $2$-cell has four $1$-cells (degenerate $1$-cells are counted with multiplicity $=2$) we have:
\begin{equation}
\label{eqn:ER2}
2(E + E_{\partial}) = 4R .
\end{equation}
Since each boundary $1$-cell contains one interior $0$-cell, 
\begin{equation}
\label{eqn:VE2}
\sum_{n=1}^{\infty} \sum_{i=0}^{n}i V(i,n-i)= E_{\partial}.
\end{equation}
Since both the endpoints of an interior $1$-cell are two distinct interior $0$-cells, counting the number of interior $1$-cells we get: 
\begin{equation}
\label{eqn:VE3}
\sum_{n=1}^{\infty} \sum_{i=0}^{n} (n-i) V(i,n-i)= 2 E
\end{equation}
Let $d$ be the number of boundary components of $F$. The Euler characteristic satisfies: 
\begin{equation}
\label{eqn:euler2}
\chi(\widehat{F}) = d+ \chi(F) = \left( d + \sum_{n=1}^\infty \sum_{i=0}^{n} V(i,n-i) \right) -(E + E_\partial) + R 
\end{equation}
From (\ref{eqn:ER2}), (\ref{eqn:VE2}), (\ref{eqn:VE3}) and (\ref{eqn:euler2}), we get the Euler characteristic equality:
\begin{equation}
\label{eqn:eulerform_seifert}
4\chi(F) = \sum_{n=1}^{\infty} \sum_{i=0}^{n} (4-n-i) V(i,n-i). 
\end{equation}

(1a) 
Suppose that $\chi(F) > 0$.  

If $\F(F)$ contains no hyperbolic points then $F$ is a disc and $\F(F)$ consists of a-arcs and one positive elliptic point. Thus $L=\partial F$ is a meridional circle of some binding component $C$ and we have $c(\phi, L, C) = 0$ as discussed in Example~\ref{ex:unknot}. 

When $\F(F)$ contains hyperbolic points we view (\ref{eqn:eulerform_seifert}) as
\begin{eqnarray*}
&&3V(0,1) + 2V(1,0) + 2V(0,2)+ V(1,1) + V(0,3)\\
&=& 4\chi(F)+  V(2,1) + 2V(3,0) + \sum_{n= 4}^\infty \sum_{i=0}^{n} (n+i-4)V(i,n-i) > 0. 
\end{eqnarray*}
By Claim~\ref{claim:fake} there exists a non-fake elliptic point of valence at most three, and Lemma~\ref{lemma:estimate-braid} shows that $|c(\phi, L, C)| \leq 3$ for some $C\subset \partial S$.

(1b) 
Suppose that $\chi(F) < 0$.

First we show $|c(\phi, L, C)| \leq 
\left\lfloor -\frac{4}{k}\chi(F) \right\rfloor  + 4$. 
As in the proof of Theorem \ref{theorem:surface},
let $\Hyp(v)$ denote the number of non-fake hyperbolic points that are connected to a $0$-cell $v$ by a singular leaf of $\F(F)$.  
Let us define 
\[ s = \min \{\Hyp(v) \: | \: v \textrm{ is a 0-cell and non-fake elliptic point} \}. 
\]
Suppose that $s = \Hyp(v)$. 
Let $C \subset \partial S$ be the binding component that contains $v$. 
We rewrite (\ref{eqn:eulerform_seifert}) as
\begin{eqnarray}
\label{eqn:euler2-1}
0&<& -4 \chi(F) \\
&\leq&  3V(0,1) + 2V(1,0) + 2V(0,2)+ V(1,1) + V(0,3) -4 \chi(F)  \notag\\
&= & V(2,1) + 2V(3,0) + \sum_{n= 4}^\infty \sum_{i=0}^{n} (n+i-4)V(i,n-i). 
\end{eqnarray}
If at least one of $V(0,1)$, $V(1,0)$, $V(0,2)$, $V(1,1)$, $V(0,3)$, $V(2,1)$, and $V(3,0)$ is positive then 
\[ 
|c(\phi, L, C)| \leq s\leq i+j \leq 3 \leq \left\lfloor \frac{-4}{k}\chi(F) \right\rfloor + 4. 
\]
If $V(i, j)=0$ whenever $i+j\leq 3$ then (\ref{eqn:euler2-1}) and a similar argument as in the last paragraph of the proof of Theorem~\ref{theorem:surface}  show that
\[ 
|c(\phi, L,C)| \leq s \leq \left\lfloor \frac{-4}{k}\chi(F) \right\rfloor + 4. 
\]

Next we show $|c(\phi, L,C)| \leq -\chi(F) + k$.
Let $h$ be the number of hyperbolic points of $\F(F)$. 
Since $k$ is equal to the number of elliptic points of $\F(F)$, by the Poincar\'e-Hopf formula we have $ \chi(F) = k -h$. Hence for any binding  component $C \subset \partial S$ that intersects $F$ we have
\[ |c(\phi, L,C) | \leq s \leq h = -\chi(F) + k. \]

(2) Assume that $\partial S $ is connected. 
All the elliptic points of $\F(F)$ lie on the binding $\partial S$ and the algebraic intersection number of $F$ and $\partial S$ satisfies $n= F \cdot \partial S = e_+ - e_-$, where $n$ is the braid index of $L$.  
For $\e,\delta \in \{\pm\}$ let us define 
$f_{\varepsilon \delta}(m):\mathbb{N} \rightarrow \Q$ by 
\[ f_{\varepsilon \delta}(m) = 
\frac{1}{m} \left\lceil
\frac{h_\varepsilon m}{e_\delta} \right \rceil. 
\]
If $e_-=0$, by Theorem~\ref{theorem:estimate-braid}-(2) we have:
\[
\left| c(\phi, L, \partial S) \right| \leq  
\max
\left\{
\inf_{m \in \mathbb{N}} f_{++}(m), \ 
\inf_{m \in \mathbb{N}} f_{-+}(m)
\right\}. 
\]
Since $h_{+} + h_{-} = e_+ + e_- - \chi(F) = n -\chi(F)$ (Proposition \ref{sl-formula-1}-(2)) and $e_\pm, h_\pm \geq 0$ 
we obtain 
\[
\left| c(\phi, L, \partial S) \right| \leq 
\inf_{m \in \mathbb{N}} \frac{1}{m} 
\left\lceil \frac{m(n-\chi(F))}{n} \right\rceil 
\leq \frac{1}{n} \left\lceil \frac{n(n-\chi(F))}{n} \right\rceil
= \frac{n-\chi(F)}{n}. 
\]
Next assume that $e_- > 0$ i.e., $e_{+}=e_{-}+n$. 
By Theorem~\ref{theorem:estimate-braid} we have: 
\[ 
-\min\left\{
\inf_{m \in \mathbb{N}}  f_{+-}(m),
\inf_{m \in \mathbb{N}}  f_{-+}(m) 
\right\} \leq 
c(\phi, L,\partial S) 
\leq \min\left\{ 
\inf_{m \in \mathbb{N}}  f_{++}(m),
\inf_{m \in \mathbb{N}}  f_{--}(m)
\right\}
\]
and hence
\[
\left| c(\phi, L, \partial S) \right| \leq 
\max 
\left\{
\min
\left\{
\inf_{m \in \mathbb{N}}  f_{++}(m),
\inf_{m \in \mathbb{N}}  f_{--}(m)
\right\}, 
\min
\left\{
\inf_{m \in \mathbb{N}} f_{+-}(m) 
\inf_{m \in \mathbb{N}} f_{-+}(m)
\right\}
\right\}.
\]
Since $e_{+}= n+ e_{-} > e_{-}$ and $h_{+} + h_{-} = e_{+} + e_{-} -\chi(F) = n + 2e_{-} -\chi(F)$ we get
\begin{eqnarray}\label{eq:n-chi(F)/n}
\left| c(\phi, L, \partial S) \right| 
& \leq & 
\inf_{m \in \mathbb{N}} 
\frac{1}{m} \left\lceil
\frac{m (h_++h_-)/2}{e_+} \right\rceil \notag\\
&\leq&
\frac{n+2e_{-} - \chi(F)}{2(n+e_-)} \\
&<&
1 - \frac{\chi(F)}{2n} 
\qquad \qquad \mbox{(since $\chi(F) \leq 0$)}\notag\\
&\leq &
\frac{n-\chi(F)}{n}. \notag
\end{eqnarray}
\end{proof}

Theorem~\ref{theorem:genus} gives simple estimates for the genera of null-homologous closed braids.

\begin{corollary}\label{cor:lower bound of g}
Assume that $L$ is a knot of genus $g(L)$. 
\begin{enumerate}
\item 
We have
\[ g(L) \geq \frac{1}{2} \left( \min_{C \subset \partial S} \{ |c(\phi, L,C) |\}  - 3 \right) . \]
\item Assume that $\partial S$ is connected. 
\begin{enumerate}
\item If $g(L)=0$ (i.e., $L$ is an unknot) then
$|c(\phi, L,\partial S)| < 1,$
\item if $g(L)>0$ then
$|c(\phi, L,\partial S)| \leq 2g(L).$
\end{enumerate}
\end{enumerate}
\end{corollary}

\begin{proof}
(1) 
If $g(L)=0$ then by (1a) of Theorem \ref{theorem:genus} we have $|c(\phi, L,C)| \leq 3= 2g(L) + 3$ for some $C\subset \partial S$.  
If $g(L) \geq 1$ then by (1b) of Theorem \ref{theorem:genus} we have for some $C\subset \partial S$: 
\begin{eqnarray*}
|c(\phi, L,C)| & \leq & \min \left\{ \left\lfloor \frac{-4}{k}\chi(F) \right\rfloor + 4 , -\chi(F) + k \right\} \\
& \leq & -\chi(F) + 4 \\
& \leq & 2g(L) + 3  
\end{eqnarray*}

(2a) 
If $g(L)=0$, plugging $\chi(F) =1$ into (\ref{eq:n-chi(F)/n}) we get 
\[ 
|c(\phi, L, \partial S)| \leq \frac{n+2 e_- - 1}{2(n+e_-)} <1. 
\]

(2b) is a direct consequence of Theorem \ref{theorem:genus}-(2).
\end{proof}

\begin{remark-unnumbered}
In the proof of Theorem~\ref{theorem:genus} we use a  weaker form of the estimates in Theorem~\ref{theorem:estimate}. 
By using the original form of Theorem~\ref{theorem:estimate} we may sharpen the estimates in Theorem~\ref{theorem:genus} and Corollary~\ref{cor:lower bound of g}.
\end{remark-unnumbered}

\section{Geometric structures of open book manifolds and braid complements}\label{sec:geometry}

We apply results in Section~\ref{sec:topol} to study geometric structures of open book manifolds.

First we observe that periodic monodromy implies Seifert-fibered in most cases.

\begin{proposition}
\label{periodic-theorem}
Assume that $\phi \in \Aut(S, \partial S)$ is periodic and 
$c(\phi,C) \neq 0$ for every boundary component $C$ of $S$.
Then the 3-manifold $M=M_{(S, \phi)}$ is Seifert fibered.
\end{proposition}

\begin{proof}

Let $C_1,\dots,C_r$ be the boundary components of $S$. 
Let $A_i \subset S$ be an annular neighborhood of $C_i$. 
We identify $A_i$ with the annulus in the complex plane:  
\[ A_{i} \cong \{z \in \mathbb{C} \: | \: 1 \leq |z| \leq 2\} 
\]
where $C_i$ is identified with 
the unit circle $\{|z|=1\}$ oriented {\em clockwise}.   
Suppose that $ c(\phi,C_{i}) = \frac{p_i}{q_i}$ where $(p_{i}, q_{i})$ are coprime integers and $q_i>0$. 
We may arrange the monodromy $\phi$ by isotopy so that:
\begin{itemize}
\item $\phi(A_{i})=A_{i}$ \ (set-wise).
\item $\phi(z) = z \exp (-2\pi\sqrt{-1}(|z|-1)\frac{p_{i}}{q_{i}})$ for $z \in A_{i} =  \{z \in  \mathbb{C} \: | \: 1 \leq |z| \leq 2\}$.
\end{itemize}
In other words, 
putting $A = \cup_i A_i$, the FDTC of $\phi|_{S\setminus A}$ is $0$ for all the boundary components of $S\setminus A$.

Let $B_{i} \subset M_{(S, \phi)}$ be the binding component corresponding to $C_{i}$.
Take a tubular neighborhood, $N_{i}$, of $B_{i}$. 
The complement $M\setminus \bigcup_i N_i =: M_\phi$ is a mapping torus.
Since $\phi$ is periodic Thurston's work \cite{T} implies that   $M_\phi$ is Seifert-fibered. 
The fibers on $\partial N_i \subset \partial M_\phi$ are regular and each represents the homology class $p_{i}[\lambda_i] + q_{i}[\mu_i] \in H_{1}(\partial N_{i};\Z)$ where $\lambda_i$ corresponds to the longitude induced by the pages and $\mu_i$ corresponds to the boundary of a meridian disc of $N_i$.

The assumption $c(\phi,C_{i}) \neq 0$ implies that the Seifert fibration of $M_\phi$ extends to $N_{i}$, by adding the binding $B_i$ as an exceptional  fiber of the Seifert invariant $(\alpha_{i},\beta_{i})$ 
where $0 \leq \beta_{i} < \alpha_{i}=p_i$ and $\beta_{i} \equiv q_{i}$ (mod $p_i$). 
\end{proof}

\begin{remark}
The assumption $c(\phi,C_{i}) \neq 0$ is necessary. 
For example, if $S=S_{g,1}$ is a genus $g>0$ surface with one boundary then $M_{(S,id)} = \#_{2g} (S^{1} \times S^{2})$, which admits no Seifert fibered structure as $\mathbb RP^3\# \mathbb RP^3$ is the only Seifert fibered manifold that is not prime (cf. \cite{Hatcher}). 
\end{remark}

There is tight relationship among Nielsen-Thurston classification, fractional Dehn twist coefficients and geometric structures.
Thurston \cite{T} proved that the mapping torus $M_{\phi}$ of $\phi \in \Aut(S)$ is Seifert-fibered (toroidal, hyperbolic) if and only if $\phi$ is periodic (reducible, pseudo-Anosov). 
In \cite[Theorem 1.3]{i1} the first named author generalized this to the complements of closed braids in $S^{3}=M_{(D^2, id)}$ by using braid foliations. We prove parallel results for $M=M_{(S, \phi)}$ and $M-L$ the braid complement.

\begin{theorem}
\label{theorem:geometry}
Let $(S,\phi)$ be an open book decomposition of a $3$-manifold $M$.
Assume: 
\begin{itemize}
\item $\partial S$ is connected and $|c(\phi,\partial S)|>1$, or 
\item $|c(\phi, C)|>4$ for every boundary component $C$ of $S$.
\end{itemize}
Then we have the following: 
\begin{enumerate}
\item $M$ is toroidal if and only if $\phi$ is reducible.
\item $M$ is hyperbolic if and only if $\phi$ is pseudo-Anosov.
\item $M$ is Seifert fibered if and only if $\phi$ is periodic. 
\end{enumerate}
\end{theorem}

\begin{theorem}
\label{theorem:geometry-braid}
Let $(S,\phi)$ be an open book decomposition of 3-manifold $M$ and $L$ be a closed braid in $(S,\phi)$.
Assume: 
\begin{itemize}
\item $\partial S$ is connected and $|c(\phi, L,\partial S )|>1$, or 
\item $|c(\phi, L, C)|>4$ for every boundary component $C$ of $S$.
\end{itemize}
Then we have the following: 
\begin{enumerate}
\item The complement $M-N(L)$ is toroidal if and only if $i(\beta_L)\circ \phi$ is reducible.
\item The complement $M-N(L)$ is hyperbolic if and only if $i(\beta_L)\circ \phi$ is pseudo-Anosov.
\item The complement $M-N(L)$ is Seifert fibered if and only if $i(\beta_L)\circ \phi$ is  periodic. 
\end{enumerate}
\end{theorem}

As the proofs of Theorems \ref{theorem:geometry} and \ref{theorem:geometry-braid} are almost the same, we prove Theorem \ref{theorem:geometry}.

\begin{proof}[Proof of Theorem \ref{theorem:geometry}] 

A crucial point of the proof is the equivalence (1). 
Once we prove (1) the other equivalences follow from the geometrization theorem.

$(\Rightarrow)$ of (1) follows from Corollary~\ref{atoroidal-theorem}.

$(\Leftarrow)$ of (1): 
Assume that $\phi$ is reducible. 
There exists an essential simple closed curve $c$ in $S$ such that $\phi^n(c)=c$ for some $n\in \mathbb N$. 
Let $c_i = \phi^{i-1}(c)$ where $i=1,\ldots,n$. 
We may assume that $c_i$ are mutually disjoint. 
Let $\mathcal{C}=c_{1} \cup \cdots \cup c_{n}$. 
Then $\phi(\mathcal C)=\mathcal C$ and 
$\mathcal{C} \times[0,1] \subset S \times [0,1]$ gives rise to an embedded torus $\mathcal{T} = \mathcal{T}_{\mathcal C}$ in $M$. 
Our goal is to prove $\mathcal{T}$ is incompressible.

Assume contrary that $\mathcal{T}$ is compressible. 
Compressing $\mathcal{T}$ yields an embedded sphere in $M$, which bounds a $3$-ball in $M$ as $M$ is irreducible by Corollary~\ref{irreducible-theorem}. 
A compressible torus in an irreducible $3$-manifold always bounds a solid torus. 
Let $X$ denote a solid torus bounded by $\mathcal T$ with a compression disc, $D$ $(\subset X)$. 


\begin{claim}\label{claim:X}
Let $S_0$ be the page $S\times\{0\}$ of $(S, \phi)$.  
\begin{enumerate}
\item[(i)]
$\partial S_0 \cap (X \cap S_{0}) \neq \emptyset$.  
\item[(ii)]
$X \cap S_{0}$ is connected.
\end{enumerate} 
\end{claim}

\begin{proof}

(i): If $\partial S_0 \cap (X \cap S_0) = \emptyset$ then $\partial(X \cap S_0)=\mathcal{T} \cap S_0 = \mathcal C$. 
Thus, $X$ is a surface bundle over $S^{1}$ where the fiber is a connected component $S'$ of $X \cap S_{0}$. 
Since every $c_i$ is essential, $S'$ is not a disc. So $\partial X$ is not compressible in $X$, which is a contradiction.

(ii): Since $\phi|_{\partial S}= id$ each component of $X \cap S_0$ intersecting $\partial S_0$ is mapped to itself under $\phi$. Thus if $X \cap S_0$ is not connected then $X$ is not connected, which is a contradiction.
\end{proof}

\begin{figure}[htbp]
 \begin{center}
\SetLabels
(.05*.95) $C_1$\\
(.6*1) $C_2$\\
(.3*.2) $c_1$\\
(.66*.6) $c_2$\\
(.4*.5) $S' = X \cap S_0$\\
(.2*.75) $\gamma_2$\\
 \endSetLabels
\strut\AffixLabels{\includegraphics*[width=60mm]{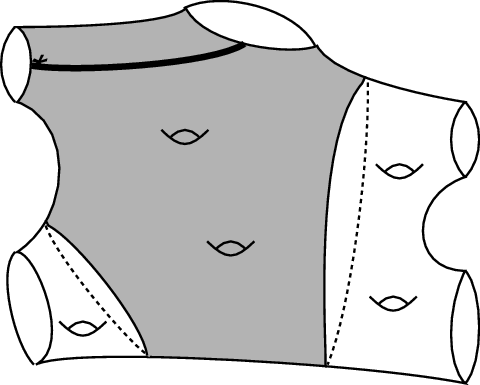}}
\caption{}
\label{fig:X}
\end{center}
\end{figure}

Let $S'=X \cap S_0$ and put $\partial S' = C_1 \cup \dots \cup C_k \cup c_1\cup \dots \cup c_n$, where $C_1 \cup \dots \cup C_k=\partial S_0 \cap \partial S'$ (see Figure~\ref{fig:X}). 
By Claim~\ref{claim:X}, $S'$ is connected and $k>0$. 
Denote the restriction of $\phi$ on $S'$ by $\phi'$.
Take a base point $* \in S'$ on $C_1$ and an arc $\gamma_{i} \subset S'$ that connects $*$ and a point on $C_i$ for $i=2,\ldots,k$. 
Then a presentation of the fundamental group of $X$ can be given as 

\begin{equation*}
\pi_1(X) = \left\langle 
\begin{array}{c|c}
a_i, b_i, C_j, c_l &
\begin{array}{l}
\prod_{i=1}^{g'}[a_i, b_i] \prod_{j=1}^k C_j 
\prod_{l=1}^n c_l, \: 
a_i \phi'_*(a_i^{-1}), \: 
b_i \phi'_*(b_i^{-1}), \\
\\
c_l \phi'_*(c_l^{-1}), \: 
\gamma_j \phi'_*(\gamma_j^{-1})
\end{array}
\end{array}
\right\rangle
\end{equation*}
where $a_1, \dots, a_{g'}, b_1,\dots,b_{g'}, C_1,\dots,C_k, c_1,\dots,c_n$ are the standard generators of $\pi_1(S')$. 
The term $\gamma_j \phi'_*(\gamma_j^{-1})$ is represented by a closed curve because $\phi' = id$ near $C_1,\dots,C_k$. 
See \cite{eo} for a more detailed proof.  
This presentation shows that the curves $c_1,\dots,c_n$ are not null-homotopic in $X$.

Cap off $S'$ with discs along $c_1,\dots, c_n$ and call the resulting surface $\widehat{S'}$. 
Let $\widehat{\phi}$ be the homeomorphism of $\widehat{S'}$ naturally extending $\phi'$ to the attached discs. 
Consider the open book $(\widehat{S'},\widehat{\phi})$ and denote $\widehat{M} = M_{(\widehat{S'},\widehat{\phi})}$. 
The centers of the attached discs give rise to a closed $n$-braid, $L$, in $\widehat{M}$ with respect to the open book $(\widehat{S'},\widehat{\phi})$.

Since $c_{i}$ are not null-homotopic in $X$, we may put the compression disc $D$ of $\mathcal T = \partial X$ by isotopy so that $\partial D$ is positively transverse to the pages. 
Therefore $\partial D$ yields a closed braid, $K$, in $\widehat M$. 
Note that $K$ bounds a disc since $D$ is properly embedded in $X$ and $X$ can be embedded in $\widehat M$. Moreover, $K$ is a cable of $L$  
since $L$ is a core of the solid torus bounded by the image of $\mathcal T$ under the embedding $X\hookrightarrow \widehat M$. 

For each component $C$ of $\partial S_0 \cap \partial S'$ we have
\[ c(\widehat\phi, K, C) = c(\widehat\phi, L, C) = c(\phi|_{S'}, C) = c(\phi,C)
\]
where the first equality holds by Remark~\ref{remark:cable} and the second equality holds since $L$ is protected by the curves $c_1,\dots,c_n$ and $c_i$ are not parallel to $C$. 
If $\partial S$ is connected and $|c(\phi, \partial S)|>1$ then by Corollary~\ref{cor:lower bound of g}-(2a) we have $g(K)\neq 0$. 
If $|c(\phi, C)|>4$ for every boundary component $C$ of $S$ then by Corollary~\ref{cor:lower bound of g}-(1) we have $g(K)\neq 0$. Therefore $K$ cannot bound a disc in $\widehat M$, which is a contradiction.

For each component $C$ of $\partial S_0 \cap \partial S'$ we have
\[ 
c(\widehat\phi, L, C) = c(\phi|_{S'}, C) = c(\phi,C),
\]
where the first equality holds since $L$ is near the orbits of the curves $c_1,\dots,c_n$ and $c_i$ are not parallel to $C$. 

If $\widehat{S'}$ is a disc then Corollary~\ref{cor:lower bound of g} and our assumption on FDTC imply that $L$ is not an unknot. 
Hence its cable $K$ is not an unknot either, which is a contradiction.

If $\widehat{S'}$ is a not a disc then by Remark~\ref{remark:cable} we have $c(\widehat\phi, K, C) = c(\widehat\phi, L, C)$, i.e., $c(\widehat\phi, K, C) = (\phi,C)$. 
Corollary~\ref{cor:lower bound of g} and our assumption on FDTC imply that $g(K)\neq 0$, which is a contradiction.

$(\Leftarrow)$ of (2):
Gabai and Oertel's Theorem~\ref{theorem:GO} implies that $M$ contains an essential lamination, hence $M$ has infinite fundamental group.
Further, Corollary~\ref{atoroidal-theorem} shows that $M$ is atoroidal and irreducible.  Hence the hyperbolization theorem implies that $M$ is hyperbolic.

$(\Rightarrow)$ of (2):
If $M$ is hyperbolic then $\phi$ is irreducible by (1). 
Thus $\phi$ is pseudo-Anosov by Proposition~\ref{periodic-theorem}. 

$(\Leftarrow)$ of (3) is a consequence of Proposition~\ref{periodic-theorem}.

$(\Rightarrow)$ of (3):
If $M$ is atoroidal then (1) and (2) imply that $\phi$ is periodic. 

In the following, we assume that $M$ is toroidal and Seifert-fibered. 
By (2), $\phi$ is reducible so $\phi$ preserves a  multi-curve $\mathcal{C}$ in $S$. We may choose $\mathcal{C} = c_1 \cup \cdots \cup c_k$ so that for each component $X$ of $S\setminus\mathcal{C}$ the restriction $\phi|_{X}$ is of irreducible. Let $\mathcal{T} \subset M$ be the suspension torus (or tori) of $\mathcal{C}$. 

Assume that $\phi|_{X}$ is pseudo-Anosov for a component $X$. Let $M_X$ be the component of $M\setminus\mathcal{T}$ containing $X$. 
If $X \cap \partial S = \emptyset$ then $M_X$ is the mapping torus of a pseudo-Anosov map so it is hyperbolic, which is a contradiction.   
Therefore, $X \cap \partial S \neq \emptyset$. 
Let  $C$ be a component of $X \cap \partial S$. 
By the above argument for $(\Leftarrow)$ of (1), we may regard $M_X$ as the complement of a closed braid $L$. 
Since $\phi$ is reducible we have $c(\widehat{\phi}, L, C)=c(\phi|_X, C) = c(\phi, C)$. 
Now Theorem~\ref{theorem:geometry-braid}-(2) implies that $M_X$ is hyperbolic, which is also a contradiction.

Therefore, $\phi$ is periodic for every component of $S\setminus\mathcal{C}$, hence every component of $M \setminus \mathcal{T}$ is Seifert-fibered. 
Moreover, there exist integers $n>0,$ $m_1, \ldots, m_k$ such that $\phi^{n}$ is freely isotopic to $T_{c_1}^{m_1}\circ \cdots \circ T_{c_k}^{m_k}$.

If some of $m_i$'s are non-zero then at a common torus boundary of two Seifert fibered components their regular fibers are not identified.  
In other words, the Seifert fibration structures of components  do not extend to $M$. 
This means $M$ is a graph manifold but not Seifert-fibered, which is a contradiction.

Hence $m_1=\dots=m_k=0$ and $\phi$ is periodic.
\end{proof}

\begin{corollary}
Theorem \ref{theorem:geometry} implies that if $M_{(S,\phi)}$  admits sol-geometry then $c(\phi,C)\leq 4$ for some boundary component $C$ of $S$.  
\end{corollary}

By stabilizing an open book sufficiently many times one can always make $\partial S$ connected while preserving the topological type of the underlying $3$-manifold $M_{(S, \phi)}$.
However, as shown in the next proposition (see also \cite[Theorem 2.16]{kr}), stabilized open books have ``small'' fractional Dehn twist coefficients. 
This partially explains why we need an open book with ``large'' $c(\phi, C)$ to extract properties of $M_{(S, \phi)}$.

\begin{proposition}
If an open book $(S,\phi)$ is a stabilization of an open book $(S',\phi')$, then there exists a component $C$ of $\partial S$ such that $|c(\phi,C)|\leq 1$. Moreover, if $\partial S$ is connected  then $|c(\phi, \partial S)| \leq \frac{1}{2}$.
\end{proposition}

\begin{proof}
By definition of stabilization we have $\phi=\phi' \circ T_{\gamma}^{\pm 1}$, where $\gamma$ denote a core of the plumbed annulus.  
We may regard $S' \subset S$ and $\phi' \in \Aut(S, \partial S)$. 
In the following we suppose that $\phi=\phi' \circ T_{\gamma}$ (similar arguments hold when $\phi=\phi' \circ T_{\gamma}^{-1}$).  

Let $\delta$ be the co-core of the plumbed annulus, i.e., an essential arc in $S$.  
Let $C$ be a component of $\partial S$ that contains (at least) one of the endpoints of $\delta$. 
By direct calculation of the images we get 
$$T_{C}^{-1}(T_{\gamma}^{-1}\delta) \geq \delta \geq T_{C}(T_{\gamma}^{-1} \delta).$$ 
Since $\phi=\phi' \circ T_{\gamma}$ and $\phi'$ is identity on $S \setminus S'$ we have $\phi(T_{\gamma}^{-1}\delta) = \phi' (\delta) = \delta$. 
This shows
\[  T_{C}^{-1}(T^{-1}_\gamma \delta) \geq \phi(T^{-1}_\gamma \delta) \geq T_{C}(T^{-1}_\gamma \delta) \]
hence by Lemma~\ref{lemma:fracDehn} we have $|c(\phi,C)| \leq 1$.

Next we assume that $\partial S$ is connected, which is possible  only if $\partial S'$ has exactly two  components, say $C_{1}$ and $C_{2}$. 
Take an integer $N \geq \max\{ |c(\phi',C_{1})|, |c(\phi',C_{2})|\}$. 
Viewing $C_{1}$ and $C_{2}$ as simple closed curves embedded in $S$, for any essential arc $l \subset S$ we have:
\begin{equation}\label{eq:C1C2}
T_{\partial S}^{-1} (l) \geq 
T_{C_{1}}^{-N}T_{C_{2}}^{-N}(l) \geq \phi'( l ) 
\end{equation}
\begin{equation}\label{eq:C1C2-2}
\phi'( l ) \geq T_{C_{1}}^{N}T_{C_{2}}^{N} (l) \geq T_{\partial S} (l)
\end{equation}
Observe that
\[ 
T_{\partial S}^{-1} (T_\gamma^{-1} \delta)
\geq T_{\partial S}^{-1} (T_\gamma \delta)
\geq \phi' (T_\gamma \delta) 
=\phi^2(T^{-1}_\gamma \delta)
\]
where the second inequality follows from (\ref{eq:C1C2}) with $l=T_{\gamma} \delta$, and 
\[
\phi^2(T^{-1}_\gamma \delta) 
= \phi'(T_\gamma \delta) 
\geq T_{C_{1}}^NT_{C_{2}}^N (T_\gamma \delta)
\geq T_{\partial S} (T_\gamma^{-1} \delta) 
\]
where the first inequality follows from (\ref{eq:C1C2-2}) and the last inequality is justified by 
Figure~\ref{fig:stabilization}. 
Therefore, we have $T_{\partial S}^{-1} (T_\gamma^{-1} \delta)
\geq \phi^2(T^{-1}_\gamma \delta) \geq T_{\partial S} (T_\gamma^{-1} \delta)$. 
By Lemma~\ref{lemma:fracDehn} and Proposition~\ref{prop:property_of_c}-(1) we conclude $|c(\phi,\partial S)| \leq \frac{1}{2}$.
\begin{figure}[htbp]
 \begin{center}
\SetLabels
(.52*.95) $\delta$\\
(.9*.95) $T_{\partial S}(T_\gamma^{-1} \delta)$\\
(1.05*.76) $T_{\partial S} (T_\gamma^{-1} \delta)$\\
 \endSetLabels
\strut\AffixLabels{\includegraphics*[width=70mm]{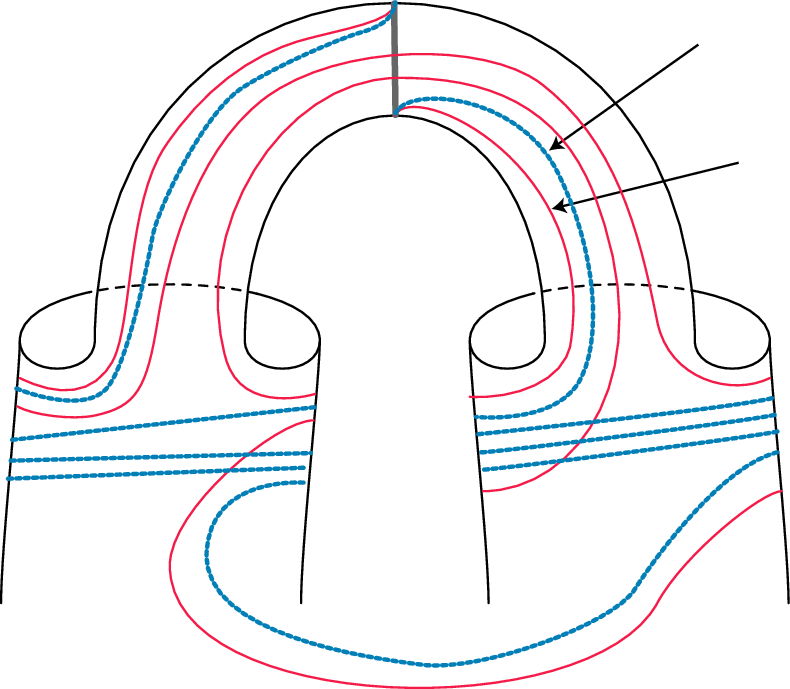}}
 \caption{$T_{C_{1}}^NT_{C_{2}}^N (T_\gamma \delta)
\geq T_{\partial S} (T_\gamma^{-1} \delta)$ where $N=4$.}
\label{fig:stabilization}
  \end{center}
\end{figure}
\end{proof}


\section*{Acknowledgement}

The authors would like to thank Ken Baker, John Etnyre, Charlie Frohman, William Kazez and Dale Rolfsen for helpful conversations, and Sam Brensinger for helping with English. 
They especially thank the referee for numerous instructive comments and for pointing out gaps and typos. 
TI was partially supported by JSPS Postdoctoral Fellowships for Research Abroad.
KK was partially supported by NSF grants DMS-1016138 and DMS-1206770.

\end{document}